\documentclass[reqno, 12pt]{amsart}
\usepackage[T1]{fontenc}
\usepackage[utf8]{inputenc}
\usepackage{amssymb,amsfonts,latexsym,amsmath,amsthm,amscd,graphpap}
\usepackage{geometry,graphicx,epstopdf,mathrsfs,mathabx}   
\usepackage[usenames, dvipsnames]{color}

\usepackage{graphicx}

\newtheorem{theorem}{Theorem}[section]
\newtheorem{proposition}[theorem]{Proposition} 
\newtheorem{definition}[theorem]{Definition}

\newtheorem{lemma}[theorem]{Lemma}

\newtheorem{coro}[theorem]{Corollary}
\newtheorem{cor}[theorem]{Corollary}

\newtheorem{remark}[theorem]{Remark}

\newcounter{compte}
\newenvironment{enum-i}{\begin{list}{\roman{compte})} {\usecounter{compte}
\topsep=1mm \itemsep=0.2mm \leftmargin=5mm  }}{\end{list}}
\newenvironment{enum-1}{\begin{list}{\arabic{compte})} {\usecounter{compte}
\topsep=1mm \itemsep=0.2mm \leftmargin=5mm  }}{\end{list}}
\newenvironment{enum-a}{\begin{list}{\alph{compte})} {\usecounter{compte}
\topsep=1mm \itemsep=0.2mm \leftmargin=5mm  }}{\end{list}}

\renewcommand{\dim}{\operatorname{dim}}

\renewcommand{\ker}{\operatorname{Ker}}

\def\N{\mathbb{N}}

\def\f{\frac}
\def\R{\mathbb{R}}

\def\la{\langle}
\def\ra{\rangle}
\def\calH{\mathcal{H}}

\def\eps{\varepsilon}

\def\I{\infty}
\def\p{\partial}

\def\embed{\hookrightarrow}
\def\calM{\mathcal{M}}

\def\calL{\mathcal{L}}

\numberwithin{equation}{section}
\usepackage{hyperref}
\hypersetup{
    colorlinks=true,       % false: boxed links; true: colored links
    linkcolor=red,          % color of internal links
    citecolor=blue,        % color of links to bibliography
    filecolor=magenta,      % color of file links
    urlcolor=cyan           % color of external links
}
    \usepackage[all]{hypcap}

\begin{document}

\author{N.\ Burq, G.\ Raugel, W.\ Schlag}

\address{ Nicolas Burq: Univ Paris-Sud, Laboratoire de Math\'ematiques d'Orsay,
Orsay Cedex, F-91405; CNRS, Orsay cedex, F-91405, France}
\address{ Genevi\`eve Raugel: CNRS, Laboratoire de Math\'{e}matiques d'Orsay, Orsay Cedex, 
F-91405; Univ Paris-Sud, Orsay cedex, F-91405, France} 
\address{Wilhelm Schlag: University of Chicago, Department of Mathematics, 5734 South University Avenue, Chicago, IL 60636, U.S.A.}

\title[Long time dynamics for damped Klein-Gordon equations] 
{ Long time dynamics for weakly damped nonlinear Klein-Gordon equations}

\begin{abstract}
We continue our study of damped nonlinear Klein-Gordon equations. In~\cite{BRS1} we considered fixed positive damping and proved a form of the soliton resolution conjecture for radial solutions. In contrast, here we consider damping which decreases in time to~$0$. In the class of radial data we again establish soliton resolution provided the damping goes to $0$ sufficiently slowly.  While~\cite{BRS1} relied on invariant manifold theory, here we use the {\L}ojasiewicz-Simon inequality applied to a suitable Lyapunov functional.  
 \end{abstract}

\thanks{ The first author was partially supported by the ANR through ANR-13-BS01-0010-03 (ANA\'E) and ANR-16-CE40-0013 (ISDEEC), the second author  was partially supported by the ANR through ANR-16-CE40-0013 (ISDEEC) and the third author was partially supported by the NSF through  DMS-1500696. The third author thanks the Institute for Advanced Study, Princeton, for its hospitality during the 2017-18 academic year. }

\maketitle

\section{Introduction}

A central question in the theory of nonlinear dispersive evolution equations concerns the long-term behavior of solutions. 
For completely integrable evolution equations the inverse scattering transform yields explicit multi-soliton solutions and the asymptotic shape of solutions (but under
quite restrictive conditions on the data) is known. In absence of complete integrability the theory is much less developed, and largely remains in its infancy. 

For the class of semi-linear equations the last five years have witnessed the influx of new ideas. Around 2012 Duyckaerts, Kenig, and Merle~\cite{DKM12} obtained the complete 
description of radial energy solutions to the three-dimensional critical wave equation
\begin{equation*}
u_{tt} - \Delta u - u^5=0.
\end{equation*}
The data are radial and belong to $\dot H^1\times L^2(\R^3)$, which is the natural space from the perspective of wellposedness. The main result in~\cite{DKM12} states that 
global solutions asymptotically decouple into radiation (possibly of large energy) and finitely many rescaled solitons $W_{\lambda_j(t)}(x)$ where $W(x) = (1+|x|^2/3)^{-\frac12}$ and $W_\lambda := \lambda^{\frac12}W(\lambda \cdot)$. The positive parameters $\lambda_j(t)>0$ depend continuously on time and their ratios either tend to $0$ or~$\infty$ as $t\to\infty$. In other words, the energy decouples into the radiation plus the energy of $W$ counted with the multiplicity of the number of solitons. 
A particular consequence of this result is that global solutions are uniformly bounded in time in the norm of~$\dot H^1\times L^2(\R^3)$. 
For subcritical equations Cazenave~\cite{Caz85} obtained such a bound many years ago, but under a restrictive condition on the nonlinearity. For example, in three dimensions, powers between $3$ and~$5$ are not covered by Cazenave's method. It is unknown at this point if global solutions to focusing nonlinear Klein-Gordon equations remain bounded in $H^1\times L^2(\R^3)$ for powers in that regime. 

For solutions which are not global, \cite{DKM12} establishes a corresponding representation into a fixed pair of functions instead of the radiation, and a sum of solitons. Thus, only type-II blowup can occur.  For nonradial solutions, a partial characterization of an analogous nature was obtained in~\cite{DJKM} along a sequence of times. 

 The channel of energy method, which was pioneered in~\cite{DKM12}, has found a number of applications over the past few years that we will not describe here in detail. 
 It does not apply in the subcritical regime due to the fact that the free radial Klein-Gordon equation does not exhibit nonzero asymptotic energy outside a backward or forward light cone; the energy at times $\pm\infty$ in the region $|x|\ge |t|$ is vacuous. This is due to the fact that for Klein-Gordon the group velocity is $<1$ whereas for the free wave equation it is exactly~$1$, leading to a fixed percentage of the aforementioned exterior energy either in forward or backward times.  The transition from wave to Klein-Gordon is due to passing  from critical to subcritical equations: indeed, in contrast to the critical equation, subcritical wave equations do not exhibit stationary soliton solutions due to the Pohozaev identity. It is therefore necessary to add the mass term for a natural formulation of the {\em subcritical soliton resolution problem}. 

In view of a lack of any techniques currently known to attack the subcritical Hamiltonian problem, the authors of this paper set out in~\cite{BRS1} to study the dissipative case. The idea underlying the addition of a damping term is the availability of methods originating in dynamical systems. In~\cite{BRS1} we relied on results from invariant manifolds and center dynamics, whereas here we employ Lyapunov functionals and the {\L}ojasiewicz-Simon inequality. We now describe the contents of this paper in more detail. 

We consider the damped Klein-Gordon equation
\begin{equation*}
\tag*{$(KG)_\alpha$}
\label{KGalpha}
\begin{split}
&u_{tt}+2\alpha(t) u_t-\Delta u+u-f(u)=0, \cr
&(u(0),u_t(0))=(\varphi_0,\varphi_1)\in\mathcal{H}_{rad},
\end{split}
\end{equation*}
where
$$
\calH=H^1(\mathbb{R}^d)\times L^2(\mathbb{R}^d),
$$
and 
$$
\calH_{rad}=H^1_{rad}(\mathbb{R}^d)\times L^2_{rad}(\mathbb{R}^d),
$$
where $1 \leq d \leq 6$. We shall denote $\vec u(t) = (u(t), u_t(t))$ the solution.
In \cite{BRS1}, we assumed that $\alpha$ is a positive constant. Here we assume that $\alpha(t)>0$ is a  positive function of  class $C^1$ in $t \geq 0$, which converges to $0$ as $t$ goes to $ \infty$, see below for more details.
The nonlinearity  $f: y \in  \mathbb{R} \mapsto f(y) \in  \mathbb{R}$ is an odd $C^1$-function which satisfies $f'(0)=0$. Moreover,  it satisfies a condition of Ambrosetti-Rabinowitz type, i.e.,   there exists 
$\gamma >0$ such that 
\begin{equation*}
\tag*{$(H.1)_f$}
\label{H1f}
\int_{\mathbb{R}^d} \big(2(1+ \gamma) F(\varphi(x)) - \varphi(x) f( \varphi(x))\big)   dx \leq 0, \quad \forall \varphi \in H^1(\mathbb{R}^d),
\end{equation*}
where $F(y) = \int_0^y f(s)  ds$.  In dimensions $d\ge2$ we impose 
  the following growth condition on~$f$ 
{\begin{equation*}
\tag*{$(H.2)_f$}
\label{H2f}
\begin{aligned}
|f'(y)| &\leq  C \max\big( |y|^\beta, |y|^{\theta -1}\big), \quad   \forall  y \in \mathbb{R},  \\
|f'(y_1) - f'(y_2)| &\leq { C |y_1 - y_2 |^{\beta} \big( 1+ |y_1 |^{\theta -1-\beta} + | y_2 |^{\theta -1-\beta} \big)}, \quad   \forall  y_1, y_2
 \in \mathbb{R},
\end{aligned}
\end{equation*}}
where  $1 < \theta < \theta^* $, $ 0<\beta<\theta-1$, {$\beta \leq 1$}, $\theta^* = 2^*-1$ and where $2^*=\infty$ if $d=1,2$ and $2^*=\frac{2d}{d-2}$ if $d\ge3$. 
We notice that, when $d \geq 3$, $ \theta^* = \frac{d+2}{d-2}$. 
 In other words, the growth of $f$ is energy subcritical for large $y=0$,  and we also assume that $f'$ is  
 $\beta$-H\"older continuous. 
 For sake of simplicity in the proofs below, we may assume, without loss of generality, that 
 {$ 0 < \beta < \min(1, \theta -1, \frac{2}{d})$.}
 
{Classical examples of a function $f$ satisfying  hypotheses \ref{H1f} and \ref{H2f} are as follows 
 \begin{equation}
\label{fExample2}
f(u) =  \sum_{i=1}^{m_1} a_i |u |^{p_i-1}u - \sum_{j=1}^{m_2}b_j |u |^{q_j -1}u, 
\text{ with } \begin{cases} & 1< q_j < p_i < \frac{d+2}{d-2}, \forall i,j  \\
& a_i, b_j \geq 0, a_{m_1}>0. \end{cases}
\end{equation}
 }

 If $\alpha(t)$ decays to $0$ too quickly as $t$ goes to $ \infty$, then the equation becomes a perturbation of the conservative case. See~\cite{Wirth1} and \cite{Wirth2} for a discussion in the context of linear equations.  
 Minimal assumptions on the dissipation rate are 
 \begin{equation*}
\tag*{$(H.1)_{\alpha}$}
\label{H1alpha}
\alpha(t)>0,\quad \lim_{t\to\infty} \alpha(t)=  0,\quad \int_{0}^{ \infty} \alpha(s)  ds = \infty,  \quad
{ \alpha(t) }\hbox{ {  non-increasing}}.
\end{equation*}
We will assume more, namely 
\begin{equation*}
\tag*{$(H.2)_{\alpha}$}
\label{H2alpha}
\alpha(t) = \frac{1}{(1 + t)^a}, \quad 0 \leq a < \frac13.
\end{equation*}

 The main results of the paper are as follows.
 
 \begin{theorem} \label{ThBRS2}
Under the conditions \ref{H1f}, \ref{H2f}, and \ref{H2alpha}, any solution $\vec u(t) $ of \ref{KGalpha}  
\begin{enumerate}
\item  either   blows-up in finite time,
\item or   exist globally and converges strongly to an equilibrium point $({Q},0)$ of \ref{KGalpha}, as $t\to \infty$.  {More precisely, $\exists c_0, c_1$;  $\forall t\geq 0$, 
\begin{equation}
\label{eq:convsubexp}
\| \vec u(t) - (Q,0)\|_{\mathcal{H}} \leq c_0 \exp \big( - c_1(1 +t)^{1-a}\big).
\end{equation}}
%where $c_0(t_0)$ is a positive constant depending only of $t_0$.}
\end{enumerate}
\end{theorem}

%In the case where all the equilibria (with energy less than the energy of the initial data) are isolated,  we obtain the same dichotomy as above, \textcolor{red}{by assuming only}
%
%One then obtains the next theorem.
%
 \begin{theorem} \label{ThBRS3}Under the additional assumption   that every equilibrium point $(\varphi, 0) \in \mathcal{H}_{rad}$ of \ref{KGalpha} (with energy smaller than $\ell$) is isolated,  the same conclusions as in Theorem~\ref{ThBRS2} hold (for initial data with energy smaller than $\ell$) with  the assumption~\ref{H2alpha} relaxed to
\begin{equation*}
\tag*{$(H.3)_{\alpha}$}
\label{H3alpha}
\alpha(t) = \frac{1}{(1 + t)^a}, \quad 0 \leq a < \frac12 .
\end{equation*}
 
 \end{theorem}
%
%We notice that the additional hypothesis of isolation of the equilibrium points holds in particular when the equilibria of \ref{KGalpha} are all hyperbolic.
We remark that our arguments below do not depend on the existence or uniqueness of a ground state solution, which in any case is {\em not} guaranteed by 
  Hypothesis \ref{H1f} alone.
We further note that  Hypothesis \ref{H1f} may actually be replaced by the following weaker one 
\begin{equation*}
\tag*{$(H.1bis)_f$}
\label{H1bisf}\int_{\mathbb{R}^d} \big(2(1+ \gamma) F(\varphi(x)) - \varphi(x) f( \varphi(x))\big)  dx \leq 0, \quad  \text{ for } \| \varphi\|_{H^1} \text{ large enough},
\end{equation*}
but, for sake of simplicity, we assume  \ref{H1f}  throughout.

 We denote by  $S_{\alpha}(t,s)$, $\alpha \geq 0$, the local non-autonomous system generated by the equation \ref{KGalpha} on $\mathcal{H}$ as well as on $\mathcal{H}_{rad}$, when the initial data are considered at time $s$.
 We introduce the energy functional (also called Lyapunov functional in the case of positive damping $\alpha (t)>0$)
 \begin{equation}
\label{Eenergie}
E(\varphi_0,\varphi_1) = \int_{\mathbb{R}^d}\left(\frac{1}{2}|\nabla\varphi_0|^2+\frac{1}{2}\varphi_0^2+\frac{1}{2}\varphi_1^2- F(\varphi_0)\right)   dx.
\end{equation}
We recall that, as long as $\vec u(s)=(u(s), u_t(s))$ exists, for $t \geq t_0  \geq0$, we have,
\begin{equation}
\label{Et0t}
E(\vec u(t)) - E(\vec u(t_0))=-2  \int_{t_0}^{t}\alpha(s)  \|u_t(s)\|_{L^2( \mathbb{R}^d)}^2   ds. 
\end{equation}
In the case of constant positive damping, the relation  easily implies that every element in the $\omega$-limit set of a global trajectory is necessarily an equilibrium.  Under the aforementioned properties on the damping, it clearly follows from Theorem \ref{ThBRS2}.
Under weaker damping assumptions,  we do not  {\em apriori} know that  global trajectories are bounded in~$\calH$. Hence,  it is completely unclear whether or not the $\omega$-limit set of a global trajectory is nonempty and whether it contains only equilibria. The following result which requires  only the much weaker~\ref{H1alpha} is thus perhaps surprising.

\begin{theorem} \label{ThBRS1}
 Suppose $\alpha(t)>0$ is continuous and satisfies  Hypothesis \ref{H1alpha}. Let $\vec u$ be a forward global solution of \ref{KGalpha}. Then, for  any $r>0$, there exist a sequence of times $t_n \to \infty$ as $n$ goes to infinity, and an equilibrium point 
 $(Q,0)$ such that
 \begin{equation}
 \lim_{n\rightarrow + \infty}  \| \vec u(t_n) - (Q,0)\|_{\mathcal{H}_{\text{rad}}}=0,\qquad \lim_{n\rightarrow + \infty}\int_{t_n -r}^{t_n +r} \| u_t(s)\|_{L^2( \mathbb{R}^d)}^2   ds=0.
 \end{equation}
 
 \end{theorem}
 %%%%%%%%%%%%%%%%%%%%%%%%%%%%%%%%%%%%%%
 %%%%%%%%%%%%%%%%%%%%%%%%%%%%%%%%%%%%%%%%
  %%%%%%%%%%%%%%%%%%%%%%%%%%%%%%%%%%%%%%%%%%%%
 %%%%%%%%%%%%%%%%%%%%%%%%%%%%%%%%%%%%%%%%%%%
 The plan of the paper is as follows.  In Section 2.1, we state the basic properties of the non-autonomous damped Klein-Gordon equation \ref{KGalpha}, in particular the local existence and uniqueness of mild solutions. We also recall  Strichartz inequalities satisfied by the corresponding linear damped Klein-Gordon equation. In Section 2.2, we recall spectral properties of the linearized elliptic equation around an equilibrium point of \ref{KGalpha}. In Section 3, we introduce the functional $K_0$, which plays a central role in the proofs of Theorems 1.1 to 1.3. We also prove, by  generalizing arguments of \cite{Caz85}, that the
 $L^2$-norm of global solutions is bounded. In Lemma \ref{eqBlowup}, we give a sufficient condition on $K_0$ for blow-up in finite time of the solutions of \ref{KGalpha}. Section 4 is devoted to the construction of special time sequences, which are useful in the proofs of 
our main results. In particular, these time sequences allow to construct a time sequence $t_n^*$ so that $K_0(u(t_n^*))$ tends to $0$ as $n$ goes to infinity (see Section \ref{sub:K0utn}). In Section \ref{sub:PreuveThBRS1}, we use this time sequence $t_n^*$ in order to prove Theorem \ref{ThBRS1}, namely that, if a solution $\vec u(t)$ does not blow up in finite positive time, then the $\omega$-limit set $\omega(\vec u(0))$ contains at least one equilibrium point $(Q,0)$ under the weak hypothesis \ref{H1alpha}. Sections 5 and 6 are the core of this paper. In Section 5, we prove Theorem \ref{ThBRS2} when $1 \leq d \leq 6$ and $d \ne 3$  or when  $d=3$ and the nonlinearity $f$ satisfies the condition $1 < \theta \leq 4$, that is, that, under the more restrictive hypothesis \ref{H2alpha}, any solution $\vec u(t)$ of \ref{H2alpha} either blows-up in finite time or strongly converges 
to the above mentioned equilibrium point $(Q,0)$ as $t$ goes to infinity. { In Section 6, we prove Theorem \ref{ThBRS2} in the more delicate case, where $d=3$ and the nonlinearity 
$f$ satisfies the condition $4 < \theta <5$. In this case, we need to work with averaged quantities in order to exploit (integral) Strichartz estimates.} In this case, we also need to use an observability inequality proved in \cite{BRS_OBS}.
The  short Section 7 is devoted to the proof of Theorem~\ref{ThBRS3}. Finally in the Appendix A, we indicate the proof of the {\L}ojasiewicz-Simon inequality in the framework  of this paper. Our results are true in dimensions $d\leq 6$. It might be possible to extend them to higher dimensions modulo technical complications. For concision, most of the proofs in the paper  will be written only for $3\leq d \leq 6$ (the cases $d=1,2$ being {\em a priori} simpler).
%%%%%%%%%%%%%%%%%%%%%%%%%%%%%%%%%%%%%%%%%%
%%%%%%%%%%%%%%%%%%%%%%%%%%%%%%%%%%%%%%%%%%%%

\section{Basic properties} \label{Basic}
\subsection{Local existence results} \label{re:LocalEx}
Let us first consider the  following linear perturbed Klein-Gordon equation, written as a  first order system
\begin{equation}\label{KGalphalin}
\begin{gathered}
\p_t\binom{u}{u_t} = \left(\begin{matrix} 0 & 1\\ \Delta -1 & 0\end{matrix}\right) \binom{u}{u_t} + \binom{0}{-2\alpha (t) u_t} \equiv A \vec u + \binom{0}{-2\alpha (t) u_t}, \\
 \vec u(t_0) = \begin{pmatrix}\varphi_0\\\varphi_1\end{pmatrix} \in \mathcal{H},
 \end{gathered}
\end{equation}
where  $t_0 \geq 0$ and $\vec u:=\binom{u}{u_{t}}$. It is well-known that $A$ generates a linear $C^0$-group on $\mathcal{H}$ and on $\mathcal{H}_{rad}$. Since $\alpha(t)$ is a continuous function of $t \geq 0$ and that $-2 \alpha(t) u_t$ is a uniformly Lipschitz continuous function on $L^2(\R^d)$, it follows from \cite[Theorem 6.1.2]{Pazy} that, for every 
$\vec u_0 \equiv (\varphi_0,\varphi_1) \in \mathcal{H}$, the system \eqref{KGalphalin} has a unique mild solution $\vec u \in C^0([t_0, +\infty), \mathcal{H})$. Moreover, the mapping $\vec u_0 \to \vec u$ is Lipschitz continuous from $\mathcal{H}$ into 
$C^0([t_0, +\infty), \mathcal{H})$ (resp. from  $\mathcal{H}_{rad}$ into 
$C^0([t_0, +\infty), \mathcal{H}_{rad})$). In addition, since $\alpha(t)$ is a $C^1$ function of $t\geq 0$, then  the solution  $\vec u(t)$ of \eqref{KGalphalin} with $\vec u_0 \in H^2(\R^d) \times H^1(\R^d)$ is a classical solution of \eqref{KGalphalin}, that is, the solution $\vec u(t)$ belongs to $C^0([t_0, +\infty), 
H^2(\R^d) \times H^1(\R^d)) \cap C^1([t_0,+\infty); H^1(\R^d) \times L^2(\R^d))$ and the equation \eqref{KGalphalin} is satisfied for any $t \in [t_0,+\infty)$.

Let $\Sigma_0(.): \mathcal{H} \to \mathcal{H}$ be the linear $C^0$-group generated by the linear conservative Klein-Gordon equation. By \cite[Corollary 4.2.2]{Pazy},  the quantity 
$$
\int_{t_0}^{t} \Sigma_0(t-s) (0,2 \alpha(s) u_t(s))   ds
$$
is a continuous function from $[t_0,+\infty)$ to $\mathcal{H}$ and the solution $\vec u(t)$ of the linear equation \eqref{KGalphalin} can be written as
\begin{equation}
\label{eq:mildKGalphalin}
\vec u(t) = \Sigma_0(t-t_0) \vec u(t_0) + \int_{t_0}^{t} \Sigma_0(t-s) (0, -2 \alpha(s) u_t(s))  ds.
\end{equation}

Arguing exactly as Martinez (see \cite{Martinez1} for example) and introducing the following energy functional on $\mathcal{H}$,
\begin{equation}
\label{eq:energyMart}
\mathcal{E}(\vec v) = \frac{1}{2}\int_{\R^d}( v_2^2 + |\nabla v_1|^2 +v_1^2)   dx , \quad \forall \vec{v}=(v_1,v_2) \in \mathcal{H},
\end{equation}
we may prove the following subexponential decay of the  $\mathcal{H}$-norm of the solution $\vec u(t)$ of \eqref{KGalphalin}.

\begin{theorem} \label{th:Martinez1} Assume that the function $\alpha(t)$ satisfies the hypothesis \ref{H1alpha}. Then, there exist positive constants $\omega$ and $K \geq 1$ such that, if 
$\vec u(t) \in \mathcal{H}$ is a solution of the equation \eqref{KGalphalin}, then
we have, for any $s \geq 0$, for any $t \geq s$,
\begin{equation}
\label{eq:decroitexp}
\mathcal{E}(\vec u)(t) \leq \mathcal{E}(\vec u)(s) \exp (1 - \omega \int_{s}^{t} 2 \alpha(\tau)   d\tau)
\leq \mathcal{E}(\vec u)(s) K \exp(- 2\omega \int_{s}^{t} \alpha(\tau)   d\tau).
\end{equation}
\end{theorem}

\begin{proof}
We first prove the above inequalities for $s=0$.
We do not repeat the proof here since it is the same, mutatis mutandis, as in  \textcolor{red}{\cite[pages 301--304]{Martinez1}}. Afterwards, we prove it for $s>0$ by doing a simple change of time variable, replacing $\vec u(t)$ and $\alpha (t)$ by $\vec u(s+t^*)$ and $\alpha(s+t^*)$. Notice that the constants $\omega$ and $K$ do not depend of the initial time $s$.
\end{proof}

%\begin{remark}
%On Page 294, Martinez gave as exemple the case where $\alpha(t) \approx t^{-\theta}$, for $t$ large, with $\theta \in [0,1]$, then as a consequence of \eqref{eq:decroitexp}, one has,
%\begin{equation}
%\label{eq:decroittheta}
%\begin{split}
%\mathcal{E}(\vec u)(t) \leq &\mathcal{E}(0) \exp (1 - \omega t^{1-\theta}), \quad \forall \theta 
%\in [0,1), \cr
%\mathcal{E}(\vec u)(t) \leq &\mathcal{E}(0) e (\ln t)^{-\omega}, \quad \theta=1.
%\end{split}
%\end{equation}
%For other examples, see also Martinez, on pages 294 to 296.
%\end{remark}
%
%
The above considerations, Theorem \ref{th:Martinez1} together with, for example, \cite[Sections 5.2 and 5.3]{Pazy} (see in particular  \textcolor{red}{\cite[Theorem 5.3.1]{Pazy}}) imply that the equation 
\eqref{KGalphalin} defines a unique evolution process, 
$\Sigma_{\alpha}(t,s)$, $0 \leq s \leq t \leq +\infty$,  and $\vec u(t) = \Sigma_{\alpha}(t, t_0) \vec u(t_0)$.

We now want to show that the solution of \eqref{KGalphalin} also satisfies the classical Strichartz inequalities for the wave equation. First we recall these inequalities in the case of the  affine conservative Klein-Gordon equation for $0 \leq t \leq T$,
\begin{equation}
\label{KG0Affine}
v_{tt}(t) - \Delta v(t) + v(t) = h(t) , \quad ~ \vec v(0) = (v_0, v_1) \in \mathcal{H},
\end{equation}
where $h(t) \in L^1((0,\infty), L^{2}(\R^d))$.

The following proposition is well-known 
\begin{proposition}
\label{prop:free1}
 In all dimensions $d\ge1$, the solution $v$ of \eqref{KG0Affine} satisfies the following energy bounds,
\begin{equation}
\label{eq:venergy}
\sup_{t\ge 0}  \| (v, \partial_t v )(t) \|_{H^1\times L^2} \leq C_0 \big [ 
\| (v_{0},v_{1}) \|_{\mathcal{H}}  +  \int_0^\infty \| h(s)\|_{L^2}   ds \big],
\end{equation}
as well as the Strichartz  estimates, in dimensions $d\ge2$,
\begin{equation}
\label{eq:Strich2} 
\|  v\|_{L^{q}_t L^p( \mathbb{R}^d)} \leq C_0^*\big[ \| (v_{0},v_{1}) \|_{\mathcal{H}} + 
\|  h\|_{L^{\tilde q'}_t L^{\tilde p'}_x   }  \big],
\end{equation}
where $\frac{1}{q}+\f{d}{p}=\f{d}{2}-1= \frac{1}{\tilde q'}+\frac{d}{\tilde p'}-2$, $2\le p,\tilde p<\infty$, $2\le q,\tilde q$, and $\frac{1}{q}+\frac{d-1}{2p}\leq \frac{d-1}{4}$,  
$\frac{1}{\tilde q}+\frac{d-1}{2\tilde p}\leq \frac{d-1}{4}$ and where $C_0^* = C_0^*(p,q, p',q')$ is a positive constant.
\end{proposition}

We next apply the previous proposition to the solution $\vec u(t)$ of the linear equation \eqref{KGalphalin}. Let $(p,q)$ be a pair satisfying the conditions of Proposition \ref{prop:free1}. We denote by $P_1: \vec v \in  \mathcal{H} \to v\in H^1(\R^d)$ the projection onto the first component. Since $\vec u(t)$, $t\geq t_0$, satisfies the integral equality \eqref{eq:mildKGalphalin}, we deduce from the Strichartz estimate \eqref{eq:Strich2} with $h=-2\alpha(t) u_t$ and $(\tilde{p}',\tilde{q}')= (2, 1)$  that, 
%\begin{equation}
%\label{eq:StrichKGalphalin1}
%%\begin{split}
%\| \vec u\|_{L^{q}((t_0, +\infty), L^p( \mathbb{R}^d))}  \leq  & C_0^* \| \vec u(t_0) \|_{\mathcal{H}}  
%+ \| \int_{t_0}^{t} P_1\Sigma_0(t-s) 
%%(0, -2 \alpha(s) u_t(s)) ds \|_{L^{q}((t_0, +\infty), L^p( \mathbb{R}^d))} \cr
% \leq & C_0^* \| \vec u(t_0) \|_{\mathcal{H}} 
%+  \| \int_{t_0}^{t} P_1\Sigma_0(t-s)  1_{t \geq s}
%(0, -2\alpha(s) u_t(s)) ds \|_{L^{q}((t_0, +\infty), L^p( \mathbb{R}^d))},
%\end{split}
%\end{equation}
%which implies by the Minkowski inequality and the Strichartz estimate \eqref{eq:Strich2} with $h=0$, that
\begin{equation}
\label{eq:StrichKGalphalin2}
\|  u\|_{L^{q}((t_0, +\infty), L^p( \mathbb{R}^d))} 
% \leq & C_0^* \| \vec u(t_0) \|_{\mathcal{H}} 
%+   \int_{t_0}^{+\infty} \| P_1 \Sigma_0(t-s)  1_{t \geq s}
%(0, -2\alpha(s) u_t(s))  \|_{L^{q}((t_0, +\infty), L^p( \mathbb{R}^d))} ds \cr
 \leq  C_0^* \big( \| \vec u(t_0)\|_{\mathcal{H}} +  \int_{t_0}^{+\infty} 
 \| 2 \alpha(s) u_t(s) \|_{L^2}   ds \big).
\end{equation}
Finally, we apply Theorem \ref{th:Martinez1} to the inequality \eqref{eq:StrichKGalphalin2} to get, 
\begin{equation}
\label{eq:StrichKGalphalin3}
\begin{aligned}
\|  u\|_{L^{q}((t_0, +\infty), L^p( \mathbb{R}^d))} & \leq C_0^* \| \vec u(t_0) \|_{\mathcal{H}}
\big( 1 +  K^{1/2}  \int_{t_0}^{+\infty} \alpha(s) \exp (- \omega \int_{t_0}^{s} \alpha(\sigma) d\sigma )   ds  \big) \\
 & \leq C_0^* \| \vec u(t_0) \|_{\mathcal{H}}  (1 +   K^{1/2} \omega^{-1} ). 
\end{aligned}
\end{equation}
The decay of the norm of $\vec u(t)$ in $\mathcal{H}$ is subexponential, whereas the upper bound of the $L^q_tL^p( \mathbb{R}^d)$-norm of $u(t)$, obtained in the above inequality, is a simple constant. 
%%{\color{magenta} For any Strichartz pair $(p,q)$, $q>2$, by using the complex %interpolation method (see \cite[Theorem 5.1.2]{BLInterpolation}), one shows  that the %interpolate between $L^{\infty}_t(L^{2^*}_x )$ (in which space $L^{\infty}_t(H^1_x )$ %is contained with continuous injection) and an appropriate space
% $L^{q_1}_t(L^{p_1}_x )$, where $2 <q_1 <q)$, 
% $\hat{\theta}/q_1= 1/q$, $(1-\hat{\theta})/2^* + \hat{\theta}/p_1 =1/p$ and $0 <\hat{\theta} <1$ (notice that $(p_1, q_1)$ is also a Strichartz pair), is the space $L^q_tL^p_{x}$. Applying the complex interpolation method and using Theorem \ref{th:Martinez1} together with the estimate \eqref{eq:StrichKGalphalin3} where $(p,q)$ is replaced by $(p_1,q_1)$, one gets the following bound, for any $t\geq t_0$,
% \begin{equation}
%\label{eq:StrichKGalphalin3BIS}
%\begin{split}
%\|  u\|_{L^{q}_t((t_0, t), L^p( \mathbb{R}^d))}  & \leq \| \vec u(t_0) \|_{\mathcal{H}} \big( C_0^* (1 +    \omega^{-1} ) \big)^{\hat{\theta}} K^{1/2} 
%\exp \big( - \omega (1- \hat{\theta}) \int_{t_0}^{t} \alpha(\tau) d\tau \big). 
%\end{split}
%\end{equation}
%}
%%%%%%%%%%%%%%%%%%%%%%%%%%%%%%%%%%%%%%%%%%%%%

By the above properties and \cite[Theorems 5.2.2, Theorem 5.2.3 and 5.3.1]{Pazy}, the system \eqref{KGalphalin} generates a unique evolution process (or evolution system) 
$\Sigma_{\alpha}(t,s), 0 \leq s\leq t $ in $\mathcal{H}$, and \eqref{KGalphalin}  has a unique solution $\vec u(t) = \Sigma_{\alpha}(t,t_0) \vec u(t_0)$ in $\mathcal{H}$, satisfying the properties of Theorem 5.3.1 of \cite{Pazy}.
And we can state the following result..

\begin{proposition}\label{prop:StirchKGalphalin}
Let $t_0>0$ and $\vec u(t_0) \in \mathcal{H}$. System \eqref{KGalphalin} has a unique solution $\vec u(t) = \Sigma_{\alpha}(t,t_0) \vec u(t_0)$ in $C^0([t_0,+\infty),\mathcal{H})$, satisfying the properties of Theorem 5.3.1 of \cite{Pazy}. And, for any pair $(p,q)$ satisfying the conditions of Proposition \ref{prop:free1}, there exists a positive constant $C_0= C_0(p,q)$ so that, 
\begin{equation}
\label{eq:StrichKGalphalin4}
\|  u\|_{L^{q}((t_0, \infty), L^p( \mathbb{R}^d))}  \leq C_0 \| \vec u(t_0) \|_{\mathcal{H}} .
\end{equation}
\end{proposition}

%%%%%%%%%%%%%%%%%%%%%%%%%%%%%%%%%%%%%%%%%%%
%%%%%%%%%%%%%%%%%%%%%%%%%%%%%%%%%%%%%%%%%

We finally turn to the affine damped Klein-Gordon equation, for $ t_0 \leq t \leq T$,
\begin{equation}
\label{KGalphaAffine}
u_{tt}(t) + 2 \alpha(t) u_{t}(t) - \Delta u(t) + u(t) = G(t) , \quad ~ \vec u(t_0) = (\varphi_0,\varphi_1) \in \mathcal{H},
\end{equation}
where $G(t) \in L^1((0,T), L^{2}(\R^d))$, for $T >0$ (or even $T=+\infty$). Again, by 
\cite[Section 4.2.]{Pazy}, the quantity
$$
g(t) \equiv  \Sigma_0(t-t_0) \vec u(t_0)   +  \int_{t_0}^{t} \Sigma_0(t-s) (0, G(s))    ds
$$
is a continuous function from $[t_0,T]$ into $\mathcal{H}$.  By \cite[Corollary 6.1.3]{Pazy},
the integral equation 
\begin{equation}
\label{eq:solintKGalAffine}
\vec u(t) =  \Sigma_0(t-t_0) \vec u(t_0)   +  \int_{t_0}^{t} \Sigma_0(t-s) (0, G(s))    ds +
\int_{t_0}^{t} \Sigma_0(t-s) (0, -2 \alpha(s) u_t(s))   ds,
\end{equation}
has a unique solution $\vec u(t) \in C^0([t_0,T], \mathcal{H})$. This integral solution coincides actually with the mild solution (in the non-autonomous sense) of the equation \eqref{KGalphaAffine}, that is,
\begin{equation}
\label{eq:coincidSol}
\vec u(t) =  \Sigma_{\alpha}(t,t_0) \vec u(t_0)   +  \int_{t_0}^{t} \Sigma_{\alpha}(t,s) (0, G(s))   ds.
\end{equation}
Using the above integral formula for the solution $\vec u(t)$ of Equation \eqref{KGalphaAffine}  and applying Theorem \ref{th:Martinez1} and Proposition \ref{prop:StirchKGalphalin}, we obtain the following result.
%%%%%%%%%%%%%%%%%%%%%%%%%%%%%%%%%%%%%%%%%%%%
%%%%%%%%%%%%%%%%%%%%%%%%%%%%%%%%%%%%%%%%%%%%
\begin{proposition} \label{th:StrichKGalAffine}
There exists $c_0$ such that, for any  $t_0>0$ and any $\vec u(t_0) \in \mathcal{H}$, System \eqref{KGalphaAffine} has a unique mild solution $\vec u(t)$ in $C^0([t_0,T],\mathcal{H})$ and the following estimate holds, 
\begin{equation}
\label{eq:EnergKGalAffine}
\| \vec  u(t)\|_{\mathcal{H}}  \leq c_0 \big[\| \vec u(t_0) \|_{\mathcal{H}}
\exp (- \omega \int_{t_0}^{t} \alpha(s)   ds) +  \int_{t_0}^{t}  \exp (- \omega \int_{s}^{t} \alpha(\tau) d\tau)\| G(s)\|_{L^2}  ds \big],
 \end{equation}
 where $\omega>0$ has been defined in Theorem \ref{th:Martinez1}. \\
 Moreover, for any pair $(p,q)$ satisfying the conditions of Proposition \ref{prop:free1}, there exists $C>0$ such that for any $t_0< T$,
\begin{equation}
\label{eq:StrichKGalAffine1}
\|  u \|_{L^{q}((t_0, T), L^p( \mathbb{R}^d))}  \leq C [ \| \vec u(t_0) \|_{\mathcal{H}} + 
\int_{t_0}^{T} \| G(s)\|_{L^2( \mathbb{R}^d)}  ds].
\end{equation}
\end{proposition}

\begin{proof}
1) Due to the integral form of the solution $\vec u(t)$ of \eqref{KGalphaAffine}, the inequality is a direct consequence of Theorem \ref{th:Martinez1}. \\
2) {}Using the integral form \eqref{eq:coincidSol} of the solution $\vec u(\tau)$, $t_0 \leq \tau\leq t$, of \eqref{KGalphaAffine} and Proposition \ref{prop:StirchKGalphalin}, we obtain, by applying the Minkowski inequality, that 
\begin{equation*}
%\label{eq:StrichKGalAaux1}
\begin{aligned}
\|  u \|_{L^{q}((t_0, T), L^p( \mathbb{R}^d))} & \leq C \| \vec u(t_0)\|_{\mathcal{H}} 
+  \| \int_{t_0}^{T}| P_1\Sigma_{\alpha}(\tau,s)  1_{\tau \geq s}
(0, G(s)) |   ds \|_{L^{q}((t_0, T), L^p( \mathbb{R}^d))}, \\
& \leq C \| \vec u(t_0)\|_{\mathcal{H}} 
+   \int_{t_0}^{T} \| P_1\Sigma_{\alpha}(\tau,s)  1_{\tau \geq s}
(0, G(s))  \|_{L^{q}((t_0, T), L^p( \mathbb{R}^d))}  ds ,\\
& \leq C \| \vec u(t_0)\|_{\mathcal{H}} + C  \int_{t_0}^{T} \| G(s)\|_{L^2( \mathbb{R}^d)}   ds.
\end{aligned}
\end{equation*}
and we are done. 
\end{proof}
%{\color{magenta}\begin{remark} \label{rem:StrichKGalAffine}
%If, instead of applying Proposition \ref{prop:StirchKGalphalin}, we take into account the more accurate Strichartz estimate \eqref{eq:StrichKGalphalin3BIS}, we  get the following bound, for any $t\geq t_0$,
%\begin{equation}
%\label{eq:StrichKGalAaux2}
%\begin{split}
%\|  u \|_{L^{q}_t((t_0, +\infty), L^p( \mathbb{R}^d))}  \leq & \tilde{C}_0 
%\big[\| \vec u(t_0) \|_{\mathcal{H}}
%\exp \big( - \omega (1- \hat{\theta}) \int_{t_0}^{t} \alpha(\tau) d\tau \big) \cr
%& + \int_{t_0}^{t} \| G(s)\|_{L^2( \mathbb{R}^d)} \exp \big( - \omega (1- \hat{\theta}) \int_{s}^{t} \alpha(\tau) d\tau \big)ds \big],
%\end{split}
%\end{equation}
%\end{remark}
%where $\tilde{C}_0 = \tilde{C}_0(\hat{\theta}) = \big(C_0^* (1 +    \omega^{-1} ) \big)^{\hat{\theta}} K^{1/2}$. }
%
We are now able to state the theorem of local existence of solutions of the equation \ref{KGalpha}. This local existence theorem is analogous to the local existence theorem, given in \cite{BRS1} (see Theorem 2.3 there) in the case of a constant damping 
$\alpha>0$.
The proof of the following theorem follows the lines of the proof of 
\cite[Theorem 2.3]{BRS1} if one replaces the Strichartz estimates given in  
\cite[Lemma 2.2]{BRS1} by the Strichartz estimates given in the above 
Theorem \ref{th:StrichKGalAffine}. 

The theorem below does not only give the local existence of solutions to  Equation \ref{KGalpha}, but also to the slightly more general non-autonomous damped Klein-Gordon equation for $t_0 \geq 0$, 
\begin{equation*}
\tag*{$(KG)_{\alpha,t_0}$}
\label{KGalphat0}
\begin{split}
&u_{tt}+2\alpha(t) u_t-\Delta u+u-f(u)=0, \cr
&(u(t_0),u_t(t_0))= \vec u_0\in\mathcal{H}_{rad}.
\end{split}
\end{equation*}

%%%%%%%%%%%%%%%%%%%%%%%%%%%%%%%%%%%%%%%%%%%%%
%%%%%%%%%%%%%%%%%%%%%%%%%%%%%%%%%%%%%%%%%%%%%%%
\begin{theorem} \label{thm:LocalExist}
Let $d \leq 6$. 
Under  assumptions \ref{H2f} and  \ref{H1alpha},  for any $r>0$ there exists $T\geq C/ r^\delta$, $\delta >0$, such that for any $t_0$ and any  $\vec u_0\in \calH=H^1(\R^d) \times L^2(\R^d), \| \vec u_0 \|_{\mathcal{H}} \leq r$  
  the equation \ref{KGalphat0} has a unique  solution  (which is radial if the initial data are radial) 
  $$\vec u(t) \in \{ \vec v(t)\equiv (v(t),\partial_t v(t)) \in 
  C^0([t_0, t_0 +T], \mathcal{H}), v(t) \in  L^q((t_0, t_0+T); L^p( \mathbb{R}^d))\},$$
   with $(q,p) $ as in Proposition~\ref{prop:StirchKGalphalin} suitably chosen. This solution is bounded by $2 \|\vec u_0\|_{\calH}$ in 
    \begin{multline*}
   X_{t_0,T} =\big\{ \vec u = (u_0, u_1)(t)  \in C([t_0, t_0 +T], H^1(\R^d))\cap C^1([t_0,t_0 +T], L^2(\R^d)),\\
    u(t)\in  L^q((t_0, t_0+T); L^p( \mathbb{R}^d))\big\}.
\end{multline*}
 
   In particular,  if $3 \leq d \leq 6$,  we can take $q= \theta^*, p= 2\theta^*$.  
 Furthermore, the following properties hold.
 
{\bf 1)} If the above solution $\vec u (t) \equiv S_{\alpha}(t,t_0) \vec u_0$  with initial data $\vec u_0 \in \mathcal{H}$  exists for $t \in [t_0,t_0 +\tilde{T}]$, then there exists a neighborhood $\mathcal{V}$ in 
 $\mathcal{H}$ such that, for every $\vec v_0 \in \mathcal{V}$, the equation \ref{KGalphat0} has a unique solution 
 $S_{\alpha}(t,t_0) \vec v_0 \equiv \vec v(t) \equiv (v(t), \partial_t v(t))$  with $v \in X_{t_0,\tilde{T}}$. And  the solution 
  $$
  (t, \vec v_0) \in [t_0,\tilde{T}] \times \mathcal{V} \mapsto  S_{\alpha}(t,t_0) \vec v_0 
  \in \mathcal{H}
  $$ 
  is jointly continuous.
  
{\bf 2)}  For any $t_0 \leq \tau \leq t_0+\tilde{T}$, the map  $ \vec v_0 \in  \mathcal{V} \mapsto S_{\alpha}(t_0,\tau) \vec v_0 \in
  \mathcal{H}$ is Lipschitz continuous on the bounded sets of $\mathcal{V}$.

{\bf 3)}   The map  $\vec v_0  \in \mathcal{V} \mapsto v(t) \in X_{t_0,\tilde{T}} \cap
 L^{\theta^*}((t_0, t_0+ \tilde{T}), L^{2\theta^*}(\mathbb{R}^d))$ is a $C^1$-map. 
 %{\color{magenta} we should prove it. But actually, do we use this property in what follows?}

{\bf 4)} Let $T^* +t_0> T +t_0 >t_0$ be the maximal time of existence. If $T^*<\infty$, then 
  $$
  \limsup_{t\to T^*} \|\vec u(t)\|_{\calH}=+\infty.
  $$
  
{\bf 5)} If  
  $\vec u_{0}\in H^{2}(\R^d)\times H^1(\R^d)$, then 
  \begin{equation*}
  u \in C([t_0,t_0 + T), H^{2}(\R^d))\cap C^1([t_0, t_0+T), H^1(\R^d)).
  \end{equation*}
  
{\bf 6)} The energy~\eqref{Eenergie} decreases: for any $t_2 \geq t_1 \geq t_0$, we have,
\begin{equation}
\label{Et1t2}
E(\vec u(t_2)) - E(\vec u(t_1))=-2\int_{t_1}^{t_2} \alpha(s) \| \partial_t u (s)\|_{L^2( \mathbb{R}^d)}^2  ds ,
\end{equation}
and in particular,
\begin{equation}
\label{Et2}
E(\vec u(t_2)) \leq   E(\vec u(t_0)).
\end{equation} 

{\bf 7)}  If $\|\vec u(t_0)\|_{\mathcal{H}}\ll 1$, then the solution exists globally, and if furthermore
\begin{equation}
\label{H2alphaa}
\alpha(t) = \frac{1}{(1 + t)^a}, \quad 0 \leq a < 1.
\end{equation}
 $\| \vec u(t)\|_{\calH}$ converges sub-exponentially to
$0$ as $t\to\infty$, 
\begin{equation}
\label{eq:2bornesubexpH}
\exists c_0, \omega >0; \| \vec u_0 \|_{\calH} \leq r_0 \Rightarrow \forall t>t_0, \| \vec u(t)\|_{ \mathcal{H}} \leq c_0  \exp (- \omega \int_{t_0}^{t_0 +t} \alpha(s) ds) \|\vec u_0 \|_{\calH}.
\end{equation}
 \end{theorem}
%\begin{remark}
%The proof of point 7) below is remarkably simple. %We probably could have followed other %strategy (stability of  the equilibrium $(0,0)$ for example). The point in this article will %precisely be to extend this convergence of $\vec u (t)$ toward $(
%0,0)$ to the more %general case of high energy global solutions (with $(0,0)$ replaced by %other equilibria). We will be able to do so only under the stronger~\ref{H2alpha}.
%\end{remark}
\begin{proof}

The proofs of this result follows the lines of the proof of the similar result~\cite[Theorem 2.3]{BRS1}.
%%%%%%%%%%%%%%%%%%%%%%%%%%%%%%%%%%%%%%%%%%%%%%%%
%%%%%%%%%%%%%%%%%%%%%%%%%%%%%%%%%%%%%%%%%%%%%%%%%
The points 1) to 6) are proved by using Duhamel formulation and performing a fixed point in a ball of the space 
\begin{equation*}
Y \equiv  Y_{t_0,T} \equiv \{ \vec u \in L^{\infty}((t_0, t_0+ T) , \mathcal{H}) \},
\end{equation*}
in dimension $1;2$ or
\begin{equation}
\label{spaceY}
Y \equiv  Y_{t_0,T} \equiv \{ \vec u= (u_0, u_1)  \in L^{\infty}((t_0, t_0+ T) , \mathcal{H}), 
  u_0 \in L^{\theta^*}((t_0, t_0 +T),L^{2\theta^*}(\mathbb{R}^d)) \},
\end{equation}
in dimensions $3\leq d \leq 6$,
which is possible if $0<T$ is small enough depending on the size of the initial data, due to the subcriticality of the problem.
We  turn to the property (7), that is, we consider the case of small data. We proceed in two steps. We first recall that small initial data generate global solutions
\begin{lemma}[Small data {\em a priori} estimates]
There exists  $r_0>0$  such that for any  $t_0\geq 0 $ and any initial data satisfying
\begin{equation}\label{small} \| \vec u_0\|_{\mathcal{H}} \leq r_0,
\end{equation}
the solution of~\eqref{KGalpha}
exists on $[t_0, + \infty)$. and satisfies  for any $(q,p)$ satisfying the conditions of Proposition~\ref{prop:free1} there exists $C, K>0$ such that for any initial data satisfying~\eqref{small}, the solution satisfies 
\begin{equation}
\label{borne}
 \| \vec u\|_{L^\infty((t_0, T); \mathcal{H})}  \leq C \| \vec u_0\|_{\mathcal{H}} ,
\end{equation}
\begin{equation}
\label{borne2}
 \| u\|_{L^q((t_0, T); L^p( \mathbb{R}^d)) }\leq C(1+T)^K  \| \vec u_0\|_{\mathcal{H}}.
\end{equation}
 \end{lemma}
\begin{proof} According to~\eqref{H2f}, and Sobolev embedding, for the potential part of the energy~\eqref{Eenergie} we have 
$$ \int_{\R^d} |F(u_0|dx \leq C \bigl( \| \vec u_0\|_{\calH} ^{\beta+ 2} + \|\vec  u_0\|_{\calH} ^{\theta+1}\bigr). $$
We deduce, since $2< \beta +2 < \theta+1$,  that for $\epsilon >0$ small enough, on $\mathcal{H}_\epsilon^0$, the connected component of $(0,0)$ in the set 
$$\mathcal{H}_\epsilon= \{ \vec u \in \calH, E(\vec u_0) \leq \epsilon\}, $$
 the energy and $\calH$ norm are equivalent
$$ \forall ( \vec u_0) \in \mathcal{H}_\epsilon^0,  \| \vec u_0 \|^2_{\calH}\leq 2 E(\vec u_0)\leq 3 \| \vec u_0 \|^2 _{\calH}.$$

 The global existence with the bound~\eqref{borne} follows for small initial data from the fact that  the energy decays (and hence solutions starting in $\calH_\epsilon^0$ remain in $\calH_\epsilon^0$). 
 To prove~\eqref{borne2}, we use that from the local existence theory,  for $t_1\geq t_0$, the $L^q(t_1, t_1 + Cr_0^{ - \delta}; L^p$-norm is bounded by $2 \| (u(t_1, u_t (t_1))\|_{\calH}$ norm which in turn is bounded by  $6 \| (u(t_0, u_t (t_0))\|_{\calH}$, and the $T^K$ estimate in~\eqref{borne2} follows from gluying these estimates together. \end{proof}
 
 In a second step we prove the subexponential decay. From Proposition~\ref{eq:EnergKGalAffine}, we have for any $t>s \geq t_0$,
 \begin{equation}\label{subexp}
  \| (\vec u(t)\|_{\calH}\leq c_0 \big[\exp (- \omega \int_{s}^{t} \alpha(s) ds)\| \vec u(s) \|_{\mathcal{H}}.
 + \|f(u) \|_{L^1( s, t); L^2( \mathbb{R}^d))} \big]
\end{equation}
From~\ref{H2f}~\eqref{borne2} (starting from $t_1$) and Sobolev embeddings to control the $L^{1+ \beta}_x$-norm by the $H^1$-norm, we get
\begin{multline}\label{subexp2}
 \|f(u) \|_{L^1((s, t); L^2( \mathbb{R}^d))} \leq C \bigl( \| u\|_{L^{1+ \beta}(( s, t); H^1)}^{1+ \beta} + \|u \|_{L^\theta(( s, t); L^{2\theta})} ^\theta\bigr)\\
 \leq C (1+ |t-s|)^K \bigl( \| \vec u (s) \|_{\calH} ^{1+ \beta} + \|  \vec u (s) \|_{\calH} ^{\theta}\bigr).
 \end{multline}
We are now going to apply~\eqref{subexp} and~\eqref{subexp2} on sequences of intervals 
$$ t_{n} = \kappa n^{\beta}, \beta = \frac {2-a} {1-a} >1, n \geq 1, $$
with $\kappa$ sufficiently large so that 
$$ c_0 \exp ( - \omega\int_{t_n}^{t_{n+1}} a(s) ds \sim_{n\rightarrow + \infty} c_0 \exp {(- \frac \omega {1-a} \beta \kappa}) \leq  \frac 1 8.$$
Let $n_0$ such that for any $n\geq n_0$,
$$  c_0 \exp ( - \omega\int_{t_n}^{t_{n+1}} a(s) ds)\leq \frac 1 4.$$
We get for $n \geq n_0$,
$$ 
 \| (\vec u(t_n)\|_{\calH}\leq  \big[ \frac 1 4 
 + C ( 1+ n^{(\beta-1) K}) \bigl( \| \vec u (t_{n-1}) \|_{\calH} ^{ \beta} + \|  \vec u (t_{n-1}) \|_{\calH} ^{\theta-1}\bigr)\big]\| \vec u(t_{n-1}) \|_{\mathcal{H}}.
 $$
 Fix $N>n_0$ so that for all $p\geq N$, 
 \begin{equation}\label{petit}
  C ( 1+ p^{\beta K})2^{-\beta (p- n_0)}  \leq \frac 1 8 .
  \end{equation}
 Using~\eqref{borne}, we can choose the size of the initial data $r_0>0$ small enough so that
 $$\forall  n_0 \leq n \leq N,  \quad  C ( 1+ n^{\beta K}) \bigl( \| \vec u (t_{n-1}) \|_{\calH} ^{ \beta} + \|  \vec u (t_{n-1}) \|_{\calH} ^{\theta-1}\leq C' ( 1+ N^{\beta K}) \bigl(r_0^{ \beta} +r_0 ^{\theta-1}\bigr)\leq  \frac 1 4.$$
 This will guarantee that for all $n_0 \leq n\leq N$
 $$ 
 \| (\vec u(t_n)\|_{\calH}\leq   \frac 1 2 \|\vec u(t_{n-1}) \|_{\mathcal{H}}\leq {2^{- (n- n_0)} } r_0.$$
 Then if $N$ is large enough and $r_0 \leq 1$, using~\eqref{petit} and a straightforward induction argument shows that for any $n \geq N$
 $$ C ( 1+ n^{\beta K}) \bigl( \| \vec u (t_{n-1}) \| ^{ \beta} + \|  \vec u (t_{n-1}) \| ^{\theta-1}  \leq \frac 1 4, $$

and 
$$ 
 \| (\vec u(t_n)\|_{\calH}\leq   \frac 1 2 \vec u(t_{n-1}) \|_{\mathcal{H}}\Rightarrow \| (\vec u(t_n)\|_{\calH}\leq  C' {2^{- \beta (n- n_0)} } r_0.$$

  Since $t_n \sim n^{\beta}$, this shows the sub-exponential decay along the sequence of times~$t_n$. To get the full decay, we just use~\eqref{borne} between $t_n$ and $t_{n+1}$.  
 \end{proof}

\begin{remark}
 Since here we are mainly interested in the behavior of the solutions $\vec u(t)$ of \ref{KGalphat0} when the time $t>t_0$ goes to $+\infty$, we stated and proved the local existence and continuity properties on time intervals $[t_0+0, t_0+T]$, with $T>0$. Of course, one shows in the same way that there exists $0 < T \leq t_0$ so that the properties (1) to (6) of the above Theorem \ref{thm:LocalExist} also hold on time intervals $[-T + t_0, t_0]$.
\end{remark}

%%%%%%%%%%%%%%%%%%%%%%%%%%%%%%%%%%%%%%%%%%%%%
%%%%%%%%%%%%%%%%%%%%%%%%%%%%%%%%%%%%%%%%%%%%
\subsection{Spectral properties} \label{subsec:Spectral}
 Let us recall first that the kernel of the linearized elliptic operator around an equilibrium point is at most of dimension one in the radial setting (see \cite{BRS1} for a proof).  
 Suppose now that we have a  stationary solution $\varphi_0\in H^1_{rad}(\R^d)$ to  \ref{KGalpha}, namely,
\begin{equation*}
-\Delta \varphi_0 + \varphi_0  - f(\varphi_0)=0.
\end{equation*}
By elliptic theory,  these solutions are exponentially decaying, and lie in $C^{3,\beta^*}$ for some $\beta^*>0$. 
Solving \ref{KGalpha} for $u= \varphi_0+v$ yields
\begin{equation}\label{eq:vN}
v_{tt} + 2\alpha(t) v_t -\Delta v + v -f'(\varphi_0) v = N(\varphi_0,v),
\end{equation}
where $N(\varphi_0,v) = f(\varphi_0+v)-f(\varphi_0)-f'(\varphi_0) v$. Set $\calL=-\Delta + I - f'(\varphi_0)$. 
%Rewrite \eqref{eq:vN} in the form
%\EQ{\label{Sys1}
%\p_t\binom{v}{v_t} = \left(\begin{matrix} 0 & 1\\ - \calL & -2\alpha(t)\end{matrix}\right) \binom{v}{v_t} + \binom{0}{N(\varphi_0,v)}
%}
%Denoting the matrix operator on the right-hand side by $A_\alpha(t)$, and setting 
%$\vec v:= (v, v_{t})^t$, we may write~\eqref{Sys1} in the form
%\begin{equation*}
%\p_{t}\vec v = A_{\alpha}(t)\vec v + \vec N
%\end{equation*}

This leads us to consider the linearized equation below
\begin{equation}
\label{eq:vnonauto}
v_{tt} + 2\alpha(t) v_t -\Delta v + v -f'(\varphi_0) v = 0, \quad t \geq s, \quad \vec v(s)= \vec v_s .
 \end{equation}
We denote  $\Sigma_{\alpha}^*(t,s)$ the evolution operator associated with \eqref{eq:vnonauto}, that is,
$\vec v(t) = \Sigma_{\alpha}^*(t,s)\vec v_s$.

We next recall the 
 spectral properties of $\calL$ (see \cite[Section 2.3]{BRS1}).
\begin{proposition}\label{lem:Lspec}
The operator $\calL$ is self-adjoint with domain $H^{2}(\R^{d})$. The spectrum $\sigma(\calL)$ consists of an essential part $[1,\infty)$, which is absolutely continuous,  
and finitely many eigenvalues of finite multiplicity all of which fall into $(-\infty, 1]$. The eigenfunctions are $C^{2,\beta^*}$ with $\beta^*>0$ and the ones associated with eigenvalues below $1$ are exponentially decaying.  Over the radial functions, all eigenvalues are simple. 
\end{proposition}

%%%%%%%%%%%%%%%%%%%%%%%%%%%%%%%%%%%%%%%%%%%%%%%%%
%%%%%%%%%%%%%%%%%%%%%%%%%%%%%%%%%%%%%%%%%%
\section{Boundedness in  \texorpdfstring{$L^2$}{L-2}  of   global solutions and   the functional  \texorpdfstring{$K_0$}{K0}.}
 \label{BorneL2K0}

\subsection{Definition of  \texorpdfstring{$K_0$}{K-0}} \label{subsec:DefK0}

We consider the functional $K_0: \varphi \in  H^1(\mathbb{R}^d) \mapsto K_0(\varphi) \in \mathbb{R}$, defined by
$$
K_0 (\varphi) = \int_{\mathbb{R}^d}( |\nabla \varphi |^2 + \varphi^2 - \varphi f(\varphi))    dx.
$$
As in~\cite{BRS1},  $K_0$ plays an important role in this paper.
The ``Ambrosetti-Rabinowitz'' hypothesis $(H.1)_f$ allows one to prove the following lemmas, which will be used throughout this paper.  

\begin{lemma}
\label{SolBorne}
%Assume that Hypothesis $(H.1)_f$ holds. 
%Assume that $\varphi \in H^1(\mathbb{R}^d)$ %is such that $K_0(\varphi) > -\delta$, where $\delta >0$. 
For any $(\varphi, \psi)\in H^1(\mathbb{R}^d) \times L^2(\mathbb{R}^d)$, we have
\begin{equation}
\label{eqSolBorne}
\gamma ( \| \varphi\|_{H^1}^2 + \| \psi\|_{L^2}^2 ) \leq 2(1 + \gamma) E((\varphi, \psi)) -K_0(\varphi). 
\end{equation}
\end{lemma}
\begin{proof}
We simply write
\begin{multline}
\gamma ( \| \varphi\|_{H^1}^2 + \| \psi\|_{L^2}^2 )\\
=  2(1 + \gamma) E((\varphi, \psi)) 
- K_0(\varphi) -  \| \psi\|_{L^2}^2
 + \int_{\mathbb{R}^d} \big(2(1+ \gamma) F(\varphi) - \varphi(x) f( \varphi(x))\big)    dx \\
 \leq   2(1 + \gamma) E((\varphi, \psi)) - K_0(\varphi),
\end{multline}
where the integral is nonpositive by \ref{H1f}. 
\end{proof}

\begin{cor}\label{lem:u global}
Suppose $\vec u(t)= (u(t), \partial_{t} u(t))$ is a strong solution of \ref{KGalpha} defined on the maximal interval $0\le t< T^*$. Assume 
$$
\inf_{0\le t<T^*} K_0(u(t))> -\infty .
$$ 
Then $T^*=\infty$, i.e., the solution is global. 
\end{cor}

The proof of the next lemma uses a convexity argument. In the case of $\alpha =0$, it has been proved in  \cite{PaSat} and \cite[Corollary 2.13]{NaS}.  In the case where 
$\alpha$ is a positive constant, it has been proved  in  \cite[Lemma 2.7]{BRS1}.  For the following two results it is sufficient to impose condition~\ref{H1alpha}. 

\begin{lemma}
\label{eqBlowup}
% Assume that the hypotheses $(H.1)_f$, $(H.2)_f$ and \ref{H1alpha}  hold.
Assume that $\vec u(t) \equiv (u(t),  \partial_t u (t))$ is a solution of \ref{KGalpha} defined on $[0,T^*)$ where $T^*\in (0,\I]$ is maximal. If  
$K_0(u(t)) \leq -\delta$ (where $\delta >0$), for $t_0\le t <T^*$, then $T^*<\infty$, i.e.,  the solution blows up in finite time.
\end{lemma}
\begin{coro}\label{coro-blow} 
%Assume that the hypotheses $(H.1)_f$, $(H.2)_f$ and \ref{H0alpha}  hold. 
Suppose that the initial energy $E(\vec u_0) $ is non positive (and that the solutionis not identically zero). Then the solution blows-up in finite time $T^* < \infty$. { In particular, there does not exist a non-trivial equilibrium point $\vec u_0$ satisfying $E(\vec u_0) \leq 0$.}
\end{coro}
\begin{proof}
The case of negative energy is clear from \eqref{eqSolBorne} and Lemma~\ref{eqBlowup}. If the initial energy is zero, and the solution is not stationary, then the energy becomes negative in finite time and we are done. 
If the solution is stationary and equal to $\phi$, then $K_0(\phi)=0$. However, this contradicts~\eqref{eqSolBorne}. 
\end{proof}

\begin{proof}[Proof of Lemma~\ref{eqBlowup}] 
We assume without loss of generality that $t_0=0$, and towards a contradiction that $T^*=\infty$. In order to show that $\vec u(t)$ blows up in finite time, we  use a standard convexity 
argument see~\cite{PaSat} or ~\cite{BRS1}.  
We set 
$$
y(t) = \frac{1}{2}\|u(t)\|_{L^2}^2.
$$
We have
\begin{equation}
\label{eq:2.4}
\begin{aligned}
\dot{y}(t)&=(u(t),\dot{u}(t)),\\
\ddot{y}(t)&=\|\dot{u}(t)\|_{L^2}^2+(u(t),\ddot{u}(t))\\
&=\|\dot{u}(t)\|_{L^2}^2+(u(t),(\Delta u-u+f(u))(t)) - 2\alpha(t) (u(t),\dot{u}(t))\\
&=\|\dot{u}(t)\|_{L^2}^2-K_0(u(t)) - 2 \alpha(t) (u(t),\dot{u}(t))\\
&= \|\dot{u}(t)\|_{L^2}^2-K_0(u(t)) - 2 \alpha(t) \dot y(t). 
\end{aligned}
\end{equation}
Hence, for all $t\ge t_0$, 
\begin{equation*}
\ddot y(t) + 2\alpha(t) \dot y(t) \ge \delta,
\end{equation*}
or with $A(t)= 2\int_{t_0}^t \alpha(\tau)   d\tau$,   
\begin{equation*}
\frac{d}{dt} \big( e^{A(t)} \dot y(t)\big) \ge \delta e^{A(t)},\qquad  \dot y(t) \ge e^{-A(t)} \dot y(t_0) + \delta \int_{t_0}^t e^{A(s)-A(t)}   ds.
\end{equation*}
By our assumptions on $\alpha(t)$, $A(t)\to\infty$ as $t\to\infty$ and for any $L\ge1$ and all $s\in [t-L,t]$ we have $A(t)-A(s)\le 1$ provided $t$ is sufficiently large. 
By the preceding remark, 
$$
\dot y(t)\ge e^{-A(t)} \dot y(t_0) + \delta L e^{-1},
$$ 
for all large $t$.  In particular, $\dot y(t)\to\infty$ as $t\to\infty$ and so $y(t)\to\infty$ as $t\to\infty$. 
Next, we note that 
\begin{align} 
\label{eq:2.7}
\ddot{y}(t) + 2\alpha(t) \dot y(t) &= \|\dot{u}(t)\|_{L^2}^2-K_0(u(t)) \\
& =(2+\gamma)\|\dot{u}(t)\|_{L^2}^2+\gamma \|u(t)\|_{H^1}^2-2(1+\gamma)E(t) \\
&\qquad - \int_{\R^d} \big(2(1+\gamma) F(u(t)) - u(t) f(u(t)) \big)   dx,
\end{align}
where we have set for simplicity $E(t)=E((u(t),\dot{u}(t)))$. {Using \ref{H1f}}, we can also write, for $t\geq t_0 $ sufficiently large,
\begin{align}
\label{eq:2.8}
\ddot{y}(t) + 2\alpha(t) \dot y(t) &\ge (2+\gamma)\|\dot{u}(t)\|_{L^2}^2+\gamma\|u(t)\|_{H^1}^2-2(1+\gamma)E(0) \\ 
&\ge (2+\gamma)\|\dot{u}(t)\|_{L^2}^2+\frac{\gamma}{2}\|u(t)\|_{H^1}^2  \\
&\ge \frac{2+\gamma}{2} \frac{\dot y^2(t)}{y(t)} + \gamma y(t).
\end{align}
Using the Young inequality, one easily shows that, for any $\varepsilon >0$, and all sufficiently large $t$, 
\begin{equation*}
\varepsilon \frac{\dot y^2(t)}{y(t)} + \gamma y(t) \ge 2\alpha(t) \dot y(t).
\end{equation*}
In conclusion, there exist $\kappa>0$ and  $T$ large enough so that 
\begin{equation*}
\ddot{y}(t) \ge (1+\kappa) \frac{\dot y^2(t)}{y(t)},   \quad \forall t\ge T.
\end{equation*}
This means that $y(t)^{-\kappa}$ is concave which contradicts that $y(t)\to\infty$ {  (for details of this last step, see~\cite{PaSat} or~\cite{BRS1}). }
\end{proof}

\subsection{Convexity argument and  \texorpdfstring{$L^2$}{L-2} bound} \label{L2}

Our aim here is to check that global (forward) solutions $\vec u(t) = (u(t), u_t(t))$ have the property that $\| u(t)\|_{L^2}$ remains uniformly bounded with respect to $t$. We will closely follow the work  of Cazenave  \cite[Proposition 3.1, Page 42]{Caz85} and~\cite[Lemma 2.4, Page 39]{Caz85}). 

Let $\vec u(t)$ be a global solution of \ref{KGalpha}, not identically~$0$. By Corollary~\ref{coro-blow}, 
$E(\vec u(t))  >0$ for any $t\geq 0$. To simplify the notation in this section, we set
$$
h(t) = \|u(t)\|_{L^2}^2 = 2 y(t),
$$
where $y(t)$ has been defined in the proof of Lemma \ref{eqBlowup}. It has been proved there that $y(t)$, and thus $h(t)$, is of class $C^2$.

Using the condition \ref{H1f}, we deduce from the equality    \eqref{eq:2.7}  that
\begin{multline}
\label{fseconde}
h''(t) \geq 
 (4 + 2\gamma) \| u_t\|_{L^2}^2 +  2\gamma \| \nabla u\|_{L^2}^2 
+  2\gamma \| u\|_{L^2}^2 - 4 \int \alpha(t) u(x,t) u_t(x,t)     dx \\
 -  2(2+  2\gamma) E(\vec u(t)).
\end{multline}
This inequality together with the estimate \eqref{Et0t} imply that, for $t \geq t_1 \geq 0$,
\begin{multline}
\label{fesecondeBis}
h''(t) \geq  (4 +  2\gamma) \| u_t\|_{L^2}^2 +  2\gamma\| \nabla u\|_{L^2}^2 
+  2\gamma \| u\|_{L^2}^2 - 4 \int \alpha(t) u(x,t) u_t(x,t)    dx \cr
 -  2(2+  2\gamma) E(\vec u(t_1)).
\end{multline}
We next choose $t_1 >0$ so that, for $t \geq t_1$, 
\begin{equation}
\label{t0condit}
2 \alpha(t) \leq \gamma.
\end{equation}
Thus, the inequalities \eqref{fesecondeBis} and \eqref{t0condit} imply that, for 
$t \geq t_1$,
\begin{equation}
\label{fesecondeTer}
h''(t) \geq \tilde{\Phi}(\vec u(t)) - 2(2+  2\gamma) E(\vec u(t_1)),
\end{equation}
where
\begin{equation}
\label{Ftilde}
 \tilde{\Phi}(u,v) \equiv (4 +\gamma) \| v\|_{L^2}^2 + 2\gamma \| \nabla u\|_{L^2}^2 + \gamma \| u\|_{L^2}^2 .
\end{equation}
Next we remark that, for any $\eta >0$ and any $B>0$, we can write 
\begin{equation}
\label{fprimeBis}
 B |h'(t)| \leq B\eta \| u(t)\|_{L^2}^2 + \frac{B}{\eta} \| u_t(t)\|_{L^2}^2
\leq \gamma \| u(t)\|_{L^2}^2 + (4 +\gamma) \| u_t(t)\|_{L^2}^2.
\end{equation}
To achieve this inequality, we set 
$$
\eta = \sqrt {\gamma(4+ \gamma)^{-1}}, \quad  B = 
\sqrt {\gamma (4+ \gamma)}.
$$ 
The inequalities \eqref{fprimeBis} and \eqref{fesecondeTer} imply the second inequality in the lemma below.
We thus have proved and analog of  \cite[Lemma 2.4]{Caz85}.

\begin{lemma} \label{Lemfseconde}
%Assume that \ref{H0alpha} holds.
Let $\vec u(t)$ be a nonzero global solution of \ref{KGalpha}. %, which is global in positive time, such that $E(\vec u(t))  >0$ for any $t\geq 0$. 
Then there exists $t_1 >0$, such that for 
$t \geq t_1$, 
\begin{equation}
\label{inequa1}
h''(t) \geq \tilde{\Phi}(\vec u(t)) - 4(1+ \gamma) E(\vec u(t_1)),
\end{equation}
and 
\begin{equation}
\label{inequa2}
h''(t) \geq  \sqrt {\gamma (4+ \gamma)} \big( |h'(t)| - 
\frac{4(1 +\gamma)}{\sqrt {\gamma (4+ \gamma)}}E(\vec u(t_1)) \big).
\end{equation}
\end{lemma}

With this lemma, we can now show the boundedness of the $L^2$-norm of $u(t)$ by following \cite[Section 3]{Caz85}. In this step, we only  require $\alpha(t)\to0$ in infinite time.

\begin{proposition} \label{propL2}
Let $\vec u(t)$ be a solution of \ref{KGalpha}, which is global in positive time. Let $t_1 >0$ be as in Lemma~\ref{Lemfseconde}. We have the estimates 
\begin{equation}
\label{PropEst1}
\frac{d}{dt}(|2\gamma h(u(t)) - 4(2+ 2\gamma)E( \vec u(t_1))|^{+} )\leq 0, \quad \forall t_1 \leq t <  \infty,
\end{equation}
\begin{equation}
\label{PropEst2}
h(u(t)) \leq \sup \big( h(u(t_1)), \frac{4(1+ \gamma)}{\gamma} E( \vec u(t_1)) \big), \quad \forall t \geq t_1,
\end{equation}
and there exists $\tau_1 \geq t_1$ such that, for $t \geq \tau_1$,
\begin{equation}
\label{PropEst4}
 h(u(t)) \leq \frac{8(1+ \gamma)}{\gamma} E( \vec u(t_1)).
\end{equation}
\end{proposition}

\begin{proof} The estimates hold trivially if the solution vanishes identically. Thus we may assume that $E( \vec u(t))>0$ for all $t\ge0$ and we prove~\eqref{PropEst1} by contradiction, as  in \cite[Page 43]{Caz85}. One sets
$$
g(t) = h(t) - \frac{4(1+ \gamma)}{\gamma}E( \vec u(t_0)).
$$
If \eqref{PropEst1} does not hold, there exists $t_2 \geq t_1$ so that 
$$
g'(t_2) >0, \quad g(t_2) >0.
$$
Moreover, from \eqref{inequa1}, we infer that
$$
g'' (t) \geq  \gamma g(t) , \quad \forall t \in [t_1,  \infty).
$$
Hence $g$ is a convex increasing function on $[t_2,  \infty)$ with $\lim_{t \to  \infty}g(t)
=  \infty$. Since $g(t) \geq 0$ for $t \geq t_2$, \eqref{inequa1} implies that, for $t \geq t_2$,
$$
h''(t) \geq (4 +\gamma) \| u_t(t)\|_{L^2}^2.
$$ 
Multiplying by $h(t)$, we get
$$
h(t) h''(t) \geq \frac{4 + \gamma}{4} (h'(t))^2.
$$
Hence $(h(t))^{-\gamma/4}$ is concave on $[t_2,  \infty)$. Since $(h(t))^{-\gamma/4}\to 0$ as $t \to  \infty$, this is impossible and \eqref{PropEst1} holds. Now~\eqref{PropEst2} is a direct consequence of \eqref{PropEst1}, and to prove~\eqref{PropEst4} we also proceed by contradiction. If \eqref{PropEst4} is not true, there exists $t_2 >\tau_1$, where,
$$\tau_1\equiv t_1 + (2+ 2\gamma)^{-1} E( \vec u(t_1))^{-1}|h'(t_1)/2|, $$ such that
\begin{equation}
\label{eq:Condht2}
h(t_2) > \frac{8(1 +\gamma)}{\gamma}E( \vec u(t_1)).
\end{equation}
The property \eqref{PropEst1} or \eqref{PropEst2} implies that 
\begin{equation}
\label{eq:htht2}
h(t) \geq h(t_2), \quad \forall t_1<t<t_2.
\end{equation} 
If we apply the inequality  \eqref{inequa1}, we deduce from  the inequalities \eqref{eq:Condht2} and \eqref{eq:htht2} that, for any $t_1 \leq t \leq t_2$,
$$ 
h''(t) \geq 8(1 +\gamma)E( \vec u(t_1)) 
- 4(1+ \gamma) E(\vec u(t_1)) \geq 4(1+ \gamma) E(\vec u(t_1)) >0.
$$
Therefore,
$$
h'(t_2)  \geq 4(1+\gamma) E(\vec u(t_1))(t_2 -t_1)  + h ' (t_1) > | h ' (t_1) | + h'(t_1) ,
$$
and hence
$$
h'(t_2) >0.
$$
This contradicts the property \eqref{PropEst1}  for $t=t_2$. Thus \eqref{PropEst4} is true. 
\end{proof}
%%%%%%%%%%%%%%%%%%%%%%%%%%%%%%%%%%%%%%%%%%%%%%%%%%%%%%%%%%%%%%%%%%%%
%%%%%%%%%%%%%%%%%%%%%%%%%%%%%%%%%%%%%%%%%%
%%%%%%%%%%%%%%%%%%%%%%%%%%%%%%%%%%%%%%%%%%
\section{Construction of special time sequences} \label{nbetagamma}
\subsection{Preliminary lemmata} \label{sub:preliminaire}
We begin with a few  elementary lemmata, in which we use the decay property of the energy $E(\vec u(t))$ when $\vec u$ is a global solution of \ref{KGalpha}. Indeed, in this case, $E(\vec u(t)) \geq 0$ and thus,
 \begin{equation}
\label{BorneEnergie}
2 \int_0^{ \infty} \alpha(s) \|u_t(s) \|_{L^2}^2    ds \leq E(\vec u(0)).
\end{equation}

\begin{lemma}\label{lem:suite1}
Assume  \ref{H1alpha}. Then, for any global solution $\vec u(t)$ of 
\ref{KGalpha} and for any $r>0$, there exists a sequence 
$t_n \to  \infty$, as $n$ goes to infinity, so that
\begin{equation}
\label{eq:tn1}
\lim_{n \to  \infty} \int_{t_n}^{t_n+r} \| u_t(s)\|_{L^2}^2   ds  =0.
\end{equation}
Moreover, there exist  three sequences $\tilde{t}_{n}^{j}$, $j=0,1,2$ so that 
$\tilde{t}_{n}^{j} \in [t_n +jr/3, t_n +(j+1)r/3)$ and
\begin{equation}
\label{eq:tn1j}
\lim_{n \to  \infty}  \| u_t(\tilde{t}_{n}^{j})\|_{L^2}  = 0.
\end{equation}
\end{lemma}
\begin{proof}   Assume that the property \eqref{eq:tn1} does not hold. Then there exist $\varepsilon >0$ and $t_1 >0$ so that, for $t\geq t_1$,
$$
 \int_{t}^{t+r} \| u_t(s)\|_{L^2}^2   ds  \geq \varepsilon,
$$
which implies, due to the hypothesis \ref{H1alpha}, that
\begin{equation}
\begin{split}
\label{eq:diverge1}
 \int_{t_1}^{ \infty}\alpha(s) \| u_t(s)\|_{L^2}^2   ds &\geq \sum_{n=0}^{ \infty}
 \int_{t_1 +nr}^{t_1 +(n+1)r}\alpha(s) \| u_t(s)\|_{L^2}^2   ds \cr
& \geq \varepsilon \sum_{n=0}^{ \infty} \alpha(t_1 + (n+1)r) =  \infty.
 \end{split}
 \end{equation}
 This leads to a contradiction since by \eqref{BorneEnergie} the left-hand side is bounded, and therefore, for any $n$ large enough, there exists $t_n \geq n$ so that,
\begin{equation}
\label{eq:tn1n}
\int_{t_n}^{t_n+r} \| u_t(s)\|_{L^2}^2   ds \leq \frac{1}{n},
\end{equation}
and thus \eqref{eq:tn1} holds.   The property \eqref{eq:tn1j} is an obvious consequence of \eqref{eq:tn1}.
\end{proof} 

 We will need a stronger quantitative decay.

\begin{lemma}\label{lem:suite3}
Assume that the hypothesis  \ref{H2alpha} holds. Let  $\vec u(t)$ be a global solution  of 
\ref{KGalpha}. There exists a sequence $n_m$, $m \in \N$, so that the following properties hold:
\begin{equation}
\label{eq:nmbeta}
\lim_{m \to  \infty} \int_{n_m^{3/2}}^{(n_m +1)^{3/2}}   s^{\frac13}  \| u_t(s)\|_{L^2}^2   ds \equiv \lim_{m \to  \infty}\varepsilon_m(r) = 0 ,
\end{equation}
% \begin{equation}
%\label{eq:agamma}
%(n_m +1)^{\frac32} - n_m^{\frac32} \geq n_{m}^{\frac12}, \hbox{ for any } 
% m \hbox{ large enough }
%\end{equation}
and 
 \begin{equation}
\label{eq:ellgamma}
\int_{n_m^{3/2}}^{(n_m +1)^{3/2}}  \| u_t(s)\|_{L^2}    ds \leq C \varepsilon_m^{1/2}  .
\end{equation}
\end{lemma}
\begin{proof} 
We also prove~\eqref{eq:nmbeta}  by contradiction. Assume that~ \eqref{eq:nmbeta} does not hold. Then, there exist $\delta >0$ and $n_0$, such that, for $n \geq n_0$
$$
 \int_{n^{3/2}}^{(n+1)^{3/2}} s^{\frac13} \| u_t(s)\|_{L^2}^2    ds  \geq \delta ,
$$
which implies  that
\begin{equation}
\begin{split}
\label{eq:diverge3}
 \int_{n_0^{3/2}}^{ \infty} \alpha(s) \| u_t(s)\|_{L^2}^2   ds &\geq 
 \sum_{n=n_0}^{ \infty}
 \int_{n^{3/2}}^{(n+1)^{3/2}}\alpha(s) \| u_t(s)\|_{L^2}^2    ds \cr
& \geq \delta \sum_{n=n_0}^{ \infty} \frac{1}{(1+ (n+1)^{\frac32a}){(n +1)^{\frac12}}} =  \infty,
 \end{split}
 \end{equation}
 if 
 \begin{equation}
\label{eq:auxabg1}
\frac32 a + \frac12 \leq 1,
\end{equation}
which leads to a contradiction. 
Now, using Cauchy-Schwarz inequality, we get
\begin{equation}
\label{eq:aux3}
\begin{split}
\int_{n_m^{3/2}}^{(n_m +1)^{3/2}}  \| u_t(s)\|_{L^2}   ds &\leq 
\frac{((n_m +1)^{3/2} - n_m^{3/2})^{1/2}}{n_m^{1/4}}
\big(\int_{n_m^{3/2}}^{(n_m +1)^{3/2}} s^{\frac13} \| u_t(s)\|_{L^2}^2    ds \big)^{1/2}\cr
& \leq C  \varepsilon_m^{1/2}
\frac{((n_m +1)^{3/2} - n_m^{3/2})^{1/2}}{n_m^{1/4}}\le C  \varepsilon_m^{1/2} .
\end{split}
\end{equation}
                                                                               
We notice that in what follows, we will use 
\begin{equation}\label{eq:aux3bis}
(n_{m} + 1)^{3/2} - n_{m}^{3/2} \geq n_{m}^{3a/2} .
\end{equation}
\end{proof}

Under the weaker hypothesis~\ref{H3alpha}, the same proof gives the following result. 

\begin{lemma}\label{lem:suite4}
Assume that the hypothesis  \ref{H3alpha} holds. Let  $\vec u(t)$ be a global solution  of 
\ref{KGalpha}. There exists a sequence $n_m$, $m \in \N$, so that the following property holds:
\begin{equation}
\label{eq:nmbetaBis}
\lim_{m \to  \infty} \int_{n_m^{2}}^{(n_m +1)^{2}}    \| u_t(s)\|_{L^2}^2   ds \equiv \lim_{m \to  \infty}\varepsilon_m = 0 ,
\end{equation}
\end{lemma}

\subsection{The vanishing of  \texorpdfstring{$K_0(u(t_n^*))$}{K-0(u(t-n-*))} in the limit  \texorpdfstring{$n\to\infty$}{n tends to infinity}}\label{sub:K0utn}

We have the following trichotomy for the forward evolution of \ref{KGalpha}
\begin{enumerate}
\item[(FTB)]  $ \vec u(t)$ blows up in finite positive time.  
\item[(GEB)]   $ \vec u(t)$ exists globally and the forward trajectory 
$\vec u(t)$ is bounded in $\mathcal{H}_{rad}$ uniformly in $t \geq 0$.
\item[(GEU)]   $\vec u(t)$ exists globally and is unbounded in forward time. 
 \end{enumerate}
 
 We shall later exclude the third possibility.  
 \begin{lemma}\label{lem:GEUseq} 
 Under the hypothesis ~\ref{H1alpha},  let $\vec u(t)$ be a global trajectory of~\ref{KGalpha}. And let $r>0$ be fixed. 
 There exist a sequence of times $t_n^*$ and a sequence of numbers 
 $\delta_n$, such that $t_n^* \rightarrow \infty$  
as $n \rightarrow  \infty$ and 
\begin{equation}
\label{K0n0}
\lim_{n \to  \infty} \int_{t_n^*-r}^{t_n^* +r}\| u_t(s)\|_{L^2}^2   ds =0 \qquad
\lim_{n \to  \infty} K_0(u(t_n^*)) \equiv \lim_{n \to  \infty} \delta_n =  0.
\end{equation}
\end{lemma}
\begin{proof}
We will argue by contradiction. We assume for simplicity $r=1$.

1) By Lemma \ref{lem:suite1}, there exist sequences $t_n \to  \infty$, $\varepsilon_n \to 0$  ( $n\rightarrow \infty$), and three sequences  $\tilde{t}_{n}^{j} \in [t_n +j/3, t_n +(j+1)/3)$, $j=0,1,2$, such that
\begin{equation}
\label{eq:suitetnaux1}
\int_{t_n-1}^{t_n+2} \| u_t(s)\|_{L^2}^2   ds + \sum_{j=0}^3  \| u_t(\tilde{t}_{n}^{j})\|_{L^2}  \leq {\varepsilon}_n.
\end{equation}

  2) We argue by contradiction and assume first that there exist $\delta>0$ and an infinite number of intervals
$I_n = [t_n, t_n +1]$ so that  
 \begin{equation}
\label{eq:K0delta}
\forall t \in I_n, K_0(u(t)) \geq \delta.
\end{equation}
 By Proposition \ref{propL2}, the function
$y(s)= \| u(s)\|^2_{L^2}/2$ is bounded by a positive constant $M^2$. Using  
\eqref{eq:2.4} we get
 \begin{equation*}
\ddot y(t) = \|\dot{u}(t)\|_{L^2}^2 - K_0(u(t)) - 2 \alpha(t) (u(t),\dot{u}(t))  \leq  \frac{3}{2} \|\dot{u}(t)\|_{L^2}^2 
 - \delta +2M^2\alpha^2(t),
\end{equation*}
which implies by integration between  ${\tilde t}_{n}^{0}$ and $\tilde{t}_{n}^{2}$ 
  \begin{equation}
\label{eq:yK0delta1}
\dot y(\tilde{t}_{n}^{2}) - \dot y(\tilde{t}_{n}^{0}) \leq \frac{3}{2} \varepsilon_n -\delta/3 + 2M^2\max_{t_n\le t\le t_n+1}\alpha^2(t) \leq 
  -\delta/4.
\end{equation}
On the other hand, 
 \begin{equation}
\label{eq:yprimeK0delta1} 
  -\delta/4 \geq \dot y(\tilde{t}_{n}^{2}) - \dot y(\tilde{t}_{n}^{0}) =  (u(\tilde{t}_{n}^{2}),\dot{u}(\tilde{t}_{n}^{2})) 
- (u({\tilde{t}}_{n}^{0}),\dot{u}(\tilde{t}_{n}^{0}))  \geq - 2 M {\varepsilon}_n.
\end{equation} 
 But, for $n$ large enough $2 M \tilde{\varepsilon}_n < \delta/4$, which  lead to a contradiction and shows that~\eqref{eq:K0delta} cannot be true. 

 3) Assume next that there exist $\delta>0$ and an infinite number of intervals
$I_n $ so that  
 \begin{equation}
\label{eq:K0-delta}
\forall t \in I_n, K_0(u(t)) \leq  -\delta.
\end{equation}
In this case, there exists $n_0$  such that, for $n \geq n_0$, for any $t \in I_n$,
\begin{align}
\ddot y(t) &= \|\dot{u}(t)\|_{L^2}^2 - K_0(u(t)) -  2 \alpha(t) (u(t),\dot{u}(t))  \\
& \geq  \frac{1}{2} \|\dot{u}(t)\|_{L^2}^2 
+ \delta  -2M^2\max_{t_n\le t\le t_n+1}\alpha^2(t) \geq  \frac{3}{4} \delta,
\end{align}
which implies by integration between  $\tilde{t}_{n}^{0}$ and $\tilde{t}_{n}^{2}$ that, for $n \geq n_0$,
 \begin{equation}
\label{eq:yK0-delta2}
2 M \tilde{\varepsilon}_n \geq \dot y(\tilde{t}_{n}^{2}) - \dot y(\tilde{t}_{n}^{0}) \geq\frac{1}{4}\delta .
\end{equation}
Choosing $n$ large enough again leads to a contradiction. Hence, the property 
\eqref{eq:K0-delta} cannot be true on an infinite number of intervals $I_n$.

  4) The properties 2) and 3) imply that there exist a subsequence $n_{p,p} \to \infty$, two times $\check{t}_{n_p} \in I_{n_{p,p}}$ and $\hat{t}_{n_p} \in I_{n_{p,p}}$
so that
\begin{equation}
\label{eq:suitenpp}
K_0(u(\check{t}_{n_p})) \geq -\frac{1}{p}, \quad K_0(u(\hat{t}_{n_p})) \leq \frac{1}{p}.
\end{equation}
If 
$$
 -\frac{1}{p} \leq  K_0(u(\hat{t}_{n_p})) \leq \frac{1}{p},
 $$
 we may set $t_n^*= \hat{t}_{n_p}$ and the properties \eqref{K0n0} are  proved. 
 If $K_0(u(\hat{t}_{n_p})) <  -\frac{1}{p}$, then, by \eqref{eq:suitenpp} and the intermediate value theorem, there exists $t_n^* \in I_{n_{p,p}}$ so that $K_0(u(t_n^*))
 = -\frac{1}{p}$ and again the properties \eqref{K0n0} hold.  
\end{proof}

\begin{remark} \label{TrueNonradial}
We point out that the results of  Sections \ref{sub:preliminaire} and \ref{sub:K0utn} also hold in the non-radial case, whereas the convergence property of the next section uses the radial assumption.
\end{remark}

\subsection{Proof of Theorem~\ref{ThBRS1}} \label{sub:PreuveThBRS1}

This is  similar to the proof of~\cite[Theorem 3.3]{BRS1}. For the reader's convenience we present the argument. We recall that  the sequence $t_n^*$, $n \geq 0$, is given in Lemma \ref{lem:GEUseq} together with the constant $r>0$.
%First, one remarks that the properties \eqref{K0n0} in Lemma \ref{lem:GEUseq} still hold %with $1$ replaced by any positive constant~$r$. 
>From Lemma~\ref{SolBorne}, we conclude that
\begin{equation*}
\sup_{n\ge 0} \| (u(t_n^*),  \partial_t u (t_n^*)) \|_{\calH} <\I.
\end{equation*}
 We consider the equations
\begin{align}
\tag*{$(KG)_\alpha^n$}
\label{eq:kgn}
\begin{cases}
\partial_{tt} u_{n}(t) + 2 \alpha(t_n^* +t)  \partial_t u_{n}(t) -\Delta u_n(t) + u_n(t) -f(u_n(t))=0, \\
(u_n(0), \partial_t u_{n}(0))=(u(t_n^*), \partial_t u (t_n^*)).
\end{cases}
\end{align}
By wellposedness (see Theorem \ref{thm:LocalExist}),  there exist $T>0$ and $C>0$  such that, for any $n$,  the solution $(u_n(t), \partial_t u_{n}(t))$ exists on $[-T,T]$ and, for $-T\leq t\leq T$,
\begin{align}
\label{eq:10'}
\|(u_n(t), \partial_t u_{n}(t))\|_{\mathcal{H}}\leq C.
\end{align}
Without loss of generality, we may choose $0 <T \leq r$.

In dimensions $d=1$ or $d=2$,  \eqref{eq:10'} implies that 
$\|u_n\|_{L^{\infty}((-T,T), L^p( \mathbb{R}^d))}\leq C$, for all $2 \leq p <\infty$. For dimensions 
$3 \leq d \leq 6$, the  Strichartz estimates give 
\begin{equation}
\label{eq:Strichartz3.6}
 \| u_n\|_{L^{\theta^*}((0,T), L^{2\theta^*}(\R^d))}\leq C,
\end{equation}
where $\theta^* = \frac{d+2}{d-2}$. 
By uniqueness,  $u_n(t)= u(t_n^*+t)$.
For any $s,t\in[-T,T]$, we have the following  time equicontinuity
\begin{multline*}
\int_{\mathbb{R}^d}|u_n(t)-u_n(s)|^2   dx=\int_{\mathbb{R}^d}\left|\int_s^t \p_{t}u_{n}(\sigma)d\sigma\right|^2   dx\\
\leq|t-s|\int_{\mathbb{R}^d}\int_s^t|\p_{t}u_{n}(\sigma)|^2d\sigma   dx
\leq |t-s|\int_{s+t_n^*}^{t+t_n^*}\| \partial_t u (\sigma)\|_{L^2}^2   d\sigma,
\end{multline*}
and thus, by Lemma \ref{lem:GEUseq}, 
\begin{align}
\label{eq:11}
\|u_n(t)-u_n(s)\|_{L^2}^2& \leq |t-s|\int_{s+ t_n^*}^{t+ t_n^*}\| \partial_t u (\sigma)\|_{L^2}^2   d\sigma\\
&\leq 2r \int_{t_n-r}^{t_n+r}\| \partial_t u (\sigma)\|_{L^2}^2   d\sigma\longrightarrow0, \quad\text{ as }n\to+\infty.
\end{align}
For $s,t\in[-T,T]$, and any fixed $p\in (2,2^{*})$, we deduce from the above inequality, by using an interpolation argument, that there exist $b \equiv b(p) \in (0,1)$ and a uniform constant $\tilde{C}>0$ so that, 
\begin{equation}
\label{eq:12'}
\begin{split}
\|u_n(t)-u_n(s)\|_{L^p}&\leq \|u_n(t)-u_n(s)\|_{L^{2^{*}}}^{b}
\|u_n(t)-u_n(s)\|_{L^2}^{1- b}\\
&\leq \tilde{C} |t-s|^{\frac{1- b}{2}}\bigl(\int_{t_n^*-r}^{t_n^*+r}\| \partial_t u (\sigma)\|_{L^2}^2   d\sigma\bigr)^{\frac{1-b}{2}}
\rightarrow 0  \text{ as }
n \to +\infty.
\end{split}
\end{equation}
We choose  $2<p_{0}<p_{1}<2^{*}$ (to be fixed later) and define  $\mathcal{X}=  \mathcal{X}_{p_0,p_1} \equiv L^{p_{0}}(\R^{d})\cap L^{p_{1}}(\R^{d})$. 
%{By our choice of $f$ we remark that we can choose $2<p_{0}<p_{1}<2^{*}$ so that furthermore 
%$p_2 \equiv 2\beta p_0/(p_0 -2)$ satisfies the inequality
%\begin{equation}
%\label{eq:p2}
%2 \leq p_2 \leq p_1.
%\end{equation}
We consider the family of functions $(u_n(t))_n$ in $C^0([-T,T];\mathcal{X})$. By the property \eqref{eq:10'}, we know that 
$(u_n)$ is bounded in $C^0 [-T, T]; H^1 \cap C^1[ -T,T]; L^2$, hence also  in $C^\alpha[-T,T]; H^{1- \alpha}$ and for any $2<p< 2^*$ in $C^\alpha(p) (-T, T]; L^p$. 
Due to the compact embedding of $H^1_{rad}(\mathbb{R}^d)$ into $L^q$, $2<q<2^*$, we deduce that 
the sequence $(u_n)$ is  equicontinuous in $C^0([-T,T];\mathcal{K})$, where $\mathcal(K)$ is a compact subset of $\mathcal{X}$. 
Thus, by the Ascoli's theorem, (after possibly extracting a subsequence) the sequence $(u_n)$ converges in $C^0([-T,T];\mathcal{X})$ to a function 
$u^*$ which according to~\eqref{eq:11} is constant on $[-T,T]$.  
We have proved 
\begin{itemize}
\item $u_n(t)\to u^*$ as $n\to+\infty$ in $C^0([-T,T];\mathcal{X})$ and $u^* :=  u^*(t)$
\item $\p_{t}u_{n}(t)\to0$ as $n\to+\infty$ in $L^2((-T,T);L^2(\mathbb{R}^d))$
\item $(u_n(t),\p_{t}u_{n}(t))_n$ is uniformly bounded in $L^\infty((-T,T);\mathcal{H})$ and, in particular in $L^2((-T,T);\mathcal{H})$.
\end{itemize}
 To pass from these weak convergences to strong convergences in 
 $$C([-T,T]; H^1) \cap C^1([-T,T]; L^2( \mathbb{R}^d)), $$
 we are first going to show that $u^*$ is an equilibrium point of the (undamped) Klein Gordon equation. To pass to the limit in the equation~\eqref{KGalpha} satisfied by $u_n$  it is enough to show that we can pass to the limit in the non linearity (all the linear terms passing easily to the limit in the distribution sense).  The argument here is that we can combine the slack from,  {\em sub criticality} of the non linearity with the compact embedding $H^1_{\text{radial}} \rightarrow L^p, 2< p <2^*$. 

\begin{lemma}
We can choose  $p_{0}, p_{1}$ sufficiently close to $2, 2^{*}$, respectively (depending on~\ref{H2f}) so that 

\begin{equation}\label{eq:funu^{*}}
\lim_{n\rightarrow + \infty}\sup_{t\in [-T,T]} \int_{\R^{d}}| f(u_{n}(t))u_{n}(t)-f(u^{*})u^{*} |    dx =0.
\end{equation}
\end{lemma}

\begin{proof}
Using Hypothesis \ref{H2f} and the H\"older inequality, we see that for $t \in [-T,T]$, 
\begin{multline}
\label{eq:majfunfu*}
|\int_{\R^{d}}(f(u^{*})u^{*} - f(u_{n}(t))u_{n}) (t)   dx | \\
\leq c \int_{\R^{d}} |u^{*}-u_{n}(t)| (|u_{n}(t)|^{\theta} + 
|u_{n}(t)|^{\beta +1} +  |u^{*}|^{\theta} +   |u^{*}|^{\beta + 1} ) 
   dx,
\end{multline}
and the result follows from the choice 
$$ 2< p_0 \leq \beta+2< \theta+2\leq p_1 <2^*,$$ and the convergence of $u_n$ to $u^*$ in $C([-T,T] ; L^{p_0} \cap L^{p_1})$.

\end{proof}
Since $K_0 (u_n)(0) \rightarrow 0$, $n\rightarrow + \infty$, we get 
\begin{multline}
\label{eq:13'}
\lim_{n\to+\infty}\|u_n(0)\|_{H^1}^2= \lim_{n\rightarrow + \infty}\int_{\mathbb{R}^d} f(u_n(0))u_n(0)\\
= \int_{\mathbb{R}^d} f(u^{*})u^{*}   dx=\int_{\mathbb{R}^d}(|\nabla u^*|^2+(u^*)^2)   dx = \|u^*\|_{H^1}^2,
\end{multline}
where in the third equality we used that $u^*$ is an equilibrium of \ref{KGalpha}. 

On the other hand, the sequence  $(u_n(0))$ converges to $u^*$ in $\mathcal{X}$, hence weakly in $H^1$. From ~\eqref{eq:13'} we deduce that the $H^1$ convergence is actually strong. We also have 
\begin{multline}
\label{eq4.28}u(t)= u(0) + \int_0^s \partial_s u(s)   ds \\
\Rightarrow \| u_n(t)- u^*\|_{L^\infty(-T,T); L^2(\R^d)}
\leq  \| u_n(0)- u^*\|_{L^2(\R^d)} + \| \partial_s u\|_{L^1((-T,T); L^2(\R^d))} \rightarrow_{n\rightarrow + \infty}0 .
\end{multline}

%  is uniformly bounded in $\mathcal{H}$,  up to passing to a subsequence we have $u_n(0)
%\rightharpoonup u^*$ as $n\to+\infty$ weakly in $H^1(\mathbb{R}^d)$.
%
%Since $u^*$ is an equilibrium point of \ref{KGalpha}, the following equality holds
%\begin{align}
%\label{eq:14}
%\int_{\mathbb{R}^d} f(u^*)u^{*}dx
%\end{align}
%The equalities \eqref{eq:13'} and \eqref{eq:14} imply that 
%\begin{align}
%\label{eq:15}
%\lim_{n\to+\infty}\|u_n(0)\|_{H^1}^2=\|u^*\|_{H^1}^2
%\end{align}
%and thus, since $u_n(0) \rightharpoonup u^*$ as $n\to+\infty$ weakly in 
%$H^1(\mathbb{R}^d)$, the convergence of $u_n(0)$ towards $u^*$ in fact holds in the strong sense in $H^1(\mathbb{R}^d)$. Moreover,  the strong convergence of 
%$u_n(0)$ towards $u^*$ in $L^2(\mathbb{R}^d)$ and the property \eqref{eq:11} imply the
%strong convergence of $u_n(s)$ towards $u^*$ in $L^2(\mathbb{R}^d)$, uniformly in $s \in [-T,T]$. In summary, 
%$$
%u_n(.)\to u^* \text{ in }C^{0}([-T,T],L^2(\R^{d})).
%$$
To finish the proof of the theorem it remains to show that
\begin{equation}\label{ii}
\p_{t}u_n(0)\to 0 \text{ in }L^{2}(\R^{d}).
\end{equation}
The difference  $\tilde u_{n}:=u_n-u^*$  satisfies 
\begin{align}
\label{eq:18}
\begin{cases}
\partial_{tt} \tilde u_n -\Delta \tilde u_n+\tilde u_n =f(u_n)-f(u^{*})-2\alpha(t + t_n^*)\partial_t \tilde u_n\\
\tilde u_n(0)=u_n(0)-u^*\to0\quad\text{ as }n\to+\infty\quad\text{ in }H^1(\mathbb{R}^d)\\
\partial_t \tilde u_n(0)=\partial_t u_{n}(0).
\end{cases}
\end{align}
One has  $u_n-u^*=w_n+v_n$ where $w_n$ and $v_n$ are solutions of the   equations
\begin{equation}
\label{eq:19}
\begin{cases}
\partial_{tt} w_{n}-\Delta w_n+w_n=f(u_n)-f(u^{*})-2\alpha(t+ t_n^*) \partial_t u_n \\
w_n(0)=u_n(0)-u^*, \qquad \partial_t w_{n}(0)=0,
\end{cases}
\end{equation}
and
\begin{align}
\label{eq:20}
\begin{cases}
\partial_{tt} v_{n}-\Delta v_n+v_n=0\\
v_n(0)=0, \qquad \partial_t v_{n}(0)= \partial_t u_{n}(0).
\end{cases}
\end{align}
By energy estimates  for all $-T\leq t\leq T$,
\begin{multline} 
\label{eq:21}
\|(w_n, \p_{t}w_{n})(t)\|_{\calH}\leq C\big[ \|u_n(0)-u^*\|_{H^1}+C\sqrt{2T}\left(\int_{-T}^T\|\partial_t u_{n}(s)\|_{L^2}^2   ds\right)^{\frac{1}{2}}\\
+\int_{-T}^T\|f(u_n)(s)-f(u^{*})\|_{L^2}   ds\big],  
\end{multline}
and since the two first terms in the r.h.s. of~\eqref{eq:21} tend to $0$, the following lemma allows to conclude that 
\begin{equation}
\label{eq.35}
lim_{n\rightarrow + \infty} \|(w_n, \p_{t}w_{n})(t)\|_{L^\infty((-T,T);\calH)}=0.
\end{equation}

\begin{lemma}
We have 
$$ \lim_{n\rightarrow + \infty} \|  f(u_n)(s)-f(u^{*}\|_{L^1((-T,T); L^2( \mathbb{R}^d))} =0.
$$
\end{lemma}
\begin{proof} In view of Hypothesis \ref{H2f} and the fact that $0 < \beta < \theta -1$, one has,
\begin{multline}
\label{eq:32}
\int_{-T}^T\| f(u_n)(s)-f(u^{*})\|_{L^2}   ds\\
\le C \int_{-T}^T\| |u_n(s)-u^{*}|(|u_{n}|^{\beta}  + |u^{*}|^{\beta} + |u_{n}|^{\theta-1}  + |u^{*}|^{\theta-1} ) \|_{L^2}  ds  \\
\le 2C \int_{-T}^T\| |u_n(s)-u^{*}|(1 + |u_{n}|^{\theta-1}  + |u^{*}|^{\theta-1} ) \|_{L^2}   ds  \\
\le 2C \big[ \| u_n- u^{*}\|_{L^{\infty}((-T,T),L^2( \mathbb{R}^d))} +  \int_{-T}^T\| |u_n(s)-u^{*}|( |u_{n}|^{\theta-1}  + |u^{*}|^{\theta-1} ) \|_{L^2(\R^d)}  ds.
 \end{multline}
Since $u_{n}\to u^{*}$ in $C^{0}(I,L^2(\R^{d}))$,  the first term on the right-hand side of~\eqref{eq:32}
vanishes in the limit $n\to\I$. Since, by elliptic regularity, $u^{*}$ is uniformly bounded in $L^{\infty}(\R^d)$, using~\eqref{eq4.28}, we have 
\begin{multline}
\label{eq:thetaLinfiniAux}
 \int_{-T}^T \| (u_n-u^{*}) |u^{*}|^{\theta-1} ) \|_{L^2}   ds  \leq 2T 
\| u_n-u^{*} \|_{L^{\infty}(L^2( \mathbb{R}^d))} \|u^{*} \|^{\theta-1}_{L^{\infty}(L^{\infty})} \\
\leq 
C \| u_n-u^{*} \|_{L^{\infty}(L^2( \mathbb{R}^d))} \rightarrow 0, (n \rightarrow + \infty).
\end{multline}
To bound the remaining term in~\eqref{eq:32}, we argue as in the proof of Theorem  \ref{thm:LocalExist}, by using the H\"{o}lder and Strichartz inequalities (see the estimates \eqref{eq:Strich2}), and we obtain
\begin{multline}
\label{eq:Xpq}
 \int_{-T}^T \|(u_n(s)-u^{*})|u_{n}(s)|^{\theta-1}\|_{L^2}   ds \leq   \| u_n(s)\|_{L^{\infty}(I,L^2( \mathbb{R}^d))}^{\frac{(\theta -1)\eta_1}{\theta}} 
 \| u_n(s)-u^{*}\|_{L^{\infty}(I,L^2( \mathbb{R}^d))}^{\frac{\eta_1}{\theta}} \\
 \times (2T)^{\eta_1}\big( \int_{-T}^{T} \| u_n(s)- u^{*} \|_{L^{2\theta^*}}^{\theta^*}    ds\big)^{\frac{1-\eta_1}{\theta}} 
  \times \big( \int_{-T}^{T} \| u_n(s) \|_{L^{2\theta^*}}^{\theta^*}    ds\big)^{\frac{(\theta -1)(1-\eta_1)}{\theta}},
   \end{multline}
 where  $\eta_1 = \frac{d+2 -\theta (d-2)}{4}$. The right-hand side of the 
inequality \eqref{eq:Xpq} tends to $0$ as $n$ goes to infinity. 
\end{proof}

By construction,  $v_n=(u_n-u^*)-w_n$ and, in particular, $\partial_t  v_{n}=\partial_t  u_{n}- \partial_t  w_{n}$. From \eqref{eq:tn1} and~\eqref{eq.35}, we infer that
\begin{equation}
\label{eq:24}
\lim_{n\rightarrow+ \infty} \|\partial_t  v_{n}\|_{L^2((-T,T);L^2(\mathbb{R}^d))} = 0.
\end{equation}

To conclude the proof of Theorem~\ref{ThBRS1}, we turn this   $L^{2}_{t}$ averaged estimate into a uniform estimate with
 
\begin{lemma}[ \protect{\cite[Lemma~3.4]{BRS1}}]\label{lem:obs}
For any $T_0>0$, there exists a positive constant $c(T_0)>0$, independent of $n$, such that for any solution of the{\em linear} Klein Gordon equation~\eqref{eq:20},
\begin{align}
\label{eq:25}
\|\partial_t  v_{n}(0)\|_{L^2(\mathbb{R}^d)}^2\leq c(T_0)\int_{-T_0}^{T_0}
\int_{\mathbb{R}^d}|\partial_t  v_{n}|^2   dxds.
\end{align}
\end{lemma}

\subsection{Trapping the trajectory near an equilibrium for a long time}

The following lemma is the key mechanism by which a global trajectory will be trapped near an equilibrium.  This lemma requires nothing on
the dissipation other than $\alpha\ge0$. So it also holds in the conservative case. 

\begin{lemma}
\label{lem-key}
Let $\vec u$ be a global trajectory and suppose that for some  equilibrium $Q$ and for some time 
$t_0$  $$ \| \vec u(t_0) - (Q,0) \|_{\calH} <\rho,$$
where $0<\rho<1$ is arbitrary but fixed. There exists a small $\delta_0=\delta_0(\rho,d,Q,f)$ (independent of $\alpha$) with the following property if 
$\| u(t) - Q\|_{L^2} \leq \delta_0$ for all $t\in [ t_0, t_1]$ where $t_1\ge t_0$ is arbitrary, then 
\begin{equation}\label{eq:uQHclose}
\| \vec u(t) - (Q,0) \|_{\calH} \leq 2\rho  ,
\end{equation}
for all $t\in [ t_0, t_1] $. 
\end{lemma}
\begin{proof}
Define 
\begin{equation*}
\mu(t):= \frac12 \| u(t)-Q\|_{H^1}^2 + \frac12\| u_t\|_{L^2}^2.
\end{equation*}
In view of $\langle u,v\rangle_{H^1} = \langle (-\Delta+1) u,v\rangle_{L^2}$  one has 
\begin{align*}  
\dot\mu(t) &= \langle u(t), u_t\rangle_{H^1}  - \langle  Q, u_t\rangle_{H^1} +  \langle \Delta u - u,  u_t\rangle_{L^2} - 2\alpha(t) \| u_t(t)\|_{L^2}^2 + \langle f(u(t)), u_t\rangle_{L^2} \\
&\le  \frac{d}{dt} \big[ - \langle (-\Delta +1)  Q, u \rangle_{L^2}  + \int_{\R^d}  F(u(t)) dx \big],
\end{align*}
whence for all $t_0\le t\le t_1$ and with constants depending on $Q$ and the nonlinearity $f$ via the constants in~\ref{H2f}, 
\begin{align*} 
\mu(t) &\leq \mu(t_0)  - \langle f( Q), u(t)-u(t_0)  \rangle_{L^2}  +  \int_{\R^d} F(u(t)) - F(u(t_0)) dx \\
&\leq  \frac12 \rho^2 + C\delta_0  + C\int_{\R^d} ( |u(t)|+|u(t_0)| + |u(t)|^\theta + |u(t_0)|^\theta) |u(t)-u(t_0)| dx \\
&\leq \frac12 \rho^2 + C\delta_0  + C (\| u(t)\|_{\theta+1}^\theta + \| u(t_0)\|_{\theta+1}^\theta)  \| u(t)-u(t_0) \|_{\theta+1}.
\end{align*} 
Using a classical Sobolev embedding and an interpolation inequality, we may write that, for some $\eta= \eta(\theta,d) \in (0,1)$, 
\begin{equation*}
 \| u(t)-u(t_0) \|_{\theta+1} \le C  \| u(t)-u(t_0) \|_{2}^\eta   \| u(t)-u(t_0) \|_{H^1}^{1-\eta} ,
\end{equation*}
which further implies that for all $t_0\le t\le t_1$, 
\begin{equation}\label{eq:contmu}
\mu(t) \leq  \frac12 \rho^2 + C\delta_0  + C\delta_0^\eta (1+\mu(t))^{(1-\eta)(\theta+1)}.
\end{equation}
By the method of continuity, it follows that for $\delta_0$ sufficiently small (independently of $t_0$) one has $\mu(t)\le 2\rho^2$ for all $t_0\leq t \leq t_1$. 
\end{proof} 

For future reference we remark that the condition on $\delta_0$ is 
\begin{equation}\label{eq:delrho}
C\delta_0^\eta \le \rho^2,
\end{equation}
with constant $C=C(Q,d,f)$ and $\eta=\eta(d,\theta)=\eta(d,f)$, as can be seen from~\eqref{eq:contmu}. 
The following immediate corollary of Lemma~\ref{lem-key} is the most useful form of the   trapping property for our purposes. 

\begin{coro}
\label{cor:key}
Let $\vec u$ be a global trajectory and suppose that for some   equilibrium $Q$ and for some times 
$0\le t_0\le t_1$  
$$
 \| \vec u(t_0) - (Q,0) \|_{\mathcal{H}} <\rho,\qquad \| u(t_0) - Q \|_{L^2}+ \int_{t_0}^{t_1} \| u_t(t)\|_{L^2} dt \le \delta_0,
 $$
where $0<\rho<1$ is arbitrary, and $\delta_0=\delta_0(\rho,d,Q,f)$ is small. Then 
\begin{equation}\label{eq:uQHclose2}
\| \vec u(t) - (Q,0) \|_{\calH} \leq 2\rho  ,
\end{equation}
for all $t\in [ t_0, t_1] $. 
\end{coro}
\begin{proof}
Lemma~\ref{lem-key} applies since 
\begin{equation*}
\max_{t_0\le t\le t_1} \| u(t) - Q\|_{L^2} \leq \| u(t_0) - Q \|_{2}+ \int_{t_0}^{t_1} \| u_t(s)\|_{L^2} ds  \leq \delta_0,
\end{equation*}
and we are done. 
\end{proof} 

Combining this corollary with Lemma~\ref{lem:suite3} yields the following result which guarantees trapping along increasingly long intervals $I_m$. 

\begin{proposition} 
\label{prop:H3estconv}
Under the hypothesis \ref{H2alpha}, let $\vec u(t)$ be a global solution of \ref{KGalpha}. Let  $n_{m}$, $m \in \N$ be the sequence obtained in   Lemma~\ref{lem:suite3}. 
Then there exists an equilibrium $(Q,0)$ of \ref{KGalpha} and, for any $0<\rho <1$  there exists $m_0$ such that 
\begin{equation}
\label{eq:H3estconv}
\|  \vec u (s ) - (Q,0) \|_{\mathcal{H}} \leq \rho, \quad \forall s \in I_{m} := 
[n_m^{3/2}+1, (n_{m}+1)^{3/2}],
\end{equation}
for all $m \geq m_0$. 
\end{proposition}

\begin{proof} 
The proof of Theorem~\ref{ThBRS1} applies to any sequence $J_m\subset I_m$ of intervals of fixed length, say $|J_m|=1$.  In particular, we may take $J_m=[n_m^{3/2}, n_m^{3/2}+1]$. 
This yields a sequence of times 
$t_{n_m}^{*,1}$ for which the properties \eqref{K0n0} hold, and such
 that $\vec u(t_{n_m}^{*,1})$ converges strongly to an equilibrium point $(Q,0)$. 
 
In view of \eqref{eq:ellgamma} we may apply the previous corollary to the interval $[t_{n_m}^{*,1}, (n_{m}+1)^{3/2}]$. 
 \end{proof}

\subsection{Trapping the trajectory near an equilibrium in the case of isolated equilibria}

\begin{proposition} 
\label{prop:H4Hyperb}
Assume that the hypothesis \ref{H3alpha} holds. Let $\vec u(t)$ be a global solution of \ref{KGalpha}. Suppose moreover that all the equilibria $(Q,0)$ with energy  $E(Q,0) \leq E(\vec u(0))$ are isolated.  Let  $n_{m}$, $m \in \N$, be the sequence obtained in   Lemma~\ref{lem:suite4}. 
Then there exist an equilibrium $(Q,0)$ of \ref{KGalpha} and $\rho_0 >0$, so that
for any $0<\rho < \rho_0$, there exists $m_0>0$ such that 
\begin{equation}
\label{eq:H4Hyperbest}
\|  \vec u (s ) - (Q,0) \|_{\mathcal{H}} \leq \rho, \quad \forall s \in I_{m}^2 := 
[n_m^{2}+1, (n_{m}+1)^{2}],
\end{equation}
for all $m \geq m_0$. 
\end{proposition}
\begin{proof}
The proposition is proved in two steps. \\
{\bf Step 1} 
As in the previous proposition, we remark that the proof of Theorem~\ref{ThBRS1} applies to any sequence $J_m\subset I_m^2$ of intervals of fixed length, say 
$|J_m|=1$.  In particular, we may take $J_m=[n_m^{2}, n_m^{2}+1]$. 
This yields a sequence of times 
$t_{n_m}^{*,1}$ for which the properties \eqref{K0n0} hold, and such
 that $\vec u(t_{n_m}^{*,1})$ converges strongly to an equilibrium point $(Q,0)$. 
Since $(Q,0)$ is isolated, there exists $\delta >0$ such that the ball 
$B_{\mathcal{H}}((Q,0), \delta)$ does not contain another equilibrium. We next choose 
$0 <\rho_0 < \delta/2$. By the uniform continuity of $S_{\alpha}(t,s)$ at $(Q,0)$, for any $0<\rho <\rho_0$, there exist $\eta >0$, so that, if 
 \begin{equation}
\label{eq:auxEta1}
\| (Q,0) - \vec w \|_{\mathcal{H}} \leq \eta,
\end{equation}
then
 \begin{equation}
\label{eq:auxEta2}
\| (Q,0) - S_{\alpha}(t, t_{n_m}^{*,1}) \vec w \|_{\mathcal{H}} < \rho, \quad n_m^{2}\leq t \leq n_m^{2} + 1.
\end{equation}
 Thus, since $\vec u(t_{n_m}^{*,1})$ converges to $(Q^{*},0)$, there exists $m_1$, so that, for $m \geq m_1$, we have
  \begin{equation}
\label{eq:auxEta3}
\| (Q,0) -  \vec u(t) \|_{\mathcal{H}} < \rho, \quad n_m^{2}\leq t \leq n_m^{2} + 1.
\end{equation}

 \noindent {\bf Step 2}
If the property \eqref{eq:auxEta3} does not hold for every $t \in I_m^2$, then there exists $s_{n_m} \in (n_m^2 +1,(n_{m}+1)^{2}] $ such that $\vec u(s_{n_m})$ belongs to the sphere $S((Q,0), \rho)$ of center $(Q,0)$ and radius $\rho$. 

We next choose $0 < r <1$ smaller than the local time $T>0$ of existence of solutions of \ref{KGalpha} with initial data in the ball $B((Q,0), 2\delta)$ of center $(Q,0)$ and radius $2\delta$. 

We now consider the sequence of times $s_{n_m}$. We first apply Lemma \ref{lem:GEUseq} to the sequence of intervals $[s_{n_m}-r, s_{n_m}+r]$, obtaining then a sequence of times $s_{n_m}^{*}$ for which the properties \eqref{K0n0} hold. Afterwards, we notice that the proof of Theorem~\ref{ThBRS1} also applies to the sequence of intervals $[s_{n_m}-r, s_{n_m}+r]$ and thus that $\vec u(s_{n_m}^{*})$ converges to 
an equilibrium denoted $(Q_2,0)$. Arguing as in Step 1, we show that there exists $m_2 \geq m_1$ such that, for $m \geq m_2$, we have 
  \begin{equation}
\label{eq:auxEta4}
\| (Q_2,0) -  \vec u(t) \|_{\mathcal{H}} < \rho, \quad s_{n_m}-r \leq t \leq  s_{n_m}+r,
\end{equation}
and in particular,
\begin{equation}
\label{eq:auxEta5}
\| (Q_2,0) -  \vec u(s_{n_m}) \|_{\mathcal{H}} < \rho.
\end{equation}
The property \eqref{eq:auxEta5} also implies that
 \begin{equation}
\label{eq:auxEta6}
\| (Q_2,0) -  (Q,0)) \|_{\mathcal{H}}  < 2\rho < \delta,
\end{equation}
 which leads to a contradiction, unless $Q_2 = Q$. And Proposition \ref{prop:H4Hyperb} is proved.
%%%%%%%%%%%%%%%%%%%%%%%%%%%%%%%%%%%%%%% 

\end{proof}
%%%%%%%%%%%%%%%%%%%%%%%%%%%%%%%%%%%%%%%%%%%%%%%%%%%%%%%%%%%%%%%%%%%%
%%%%%%%%%%%%%%%%%%%%%%%%%%%%%%%%%%%%%%%%%%%%%%%%%

\section{Proof of Theorem \ref{ThBRS2} if \texorpdfstring{$d>3$}{d>3} or \texorpdfstring{$d=3$}{d=3} and \texorpdfstring{$\theta \leq 4$}{theta<=4}}
\label{sec:Main1}
 
\subsection{The main functional  \texorpdfstring{$H_\nu$}{H-nu}} 

For technical reasons we first present these cases. The remaining one, i.e., 
  $d=3$ and $4<\theta<5$, is more complicated and we turn to it later. The proof proceeds by considering suitable functionals of Lyapunov type. However, the analysis is
  somewhat delicate as the functionals only give limited information (for example, only locally in time).  We will rely on the {\L}ojasiewicz-Simon inequality, which is a well-known
  method for gradient systems~ (see for example,\cite{Simon1983, HaJen99, HaJen07, HaJen13}) and used first by Simon \cite{Simon1983}.  Throughout this section $\vec u$ is a global solution, and $I_m$ is the interval from Proposition~\ref{prop:H3estconv}. 
  In particular, $(Q,0)$ will always denote the equilibrium of that proposition.   Let $\rho_0$ be given by the {\L}ojasiewicz-Simon theorem of Appendix~A, and we set $\rho=\rho_0$ in 
  Proposition~\ref{prop:H3estconv}.

\begin{definition}\label{defi:51}
Assume hypothesis~\ref{H2alpha}. Let $\vec u$ be a global trajectory, and 
let $(Q,0)$ be the equilibrium from Proposition~\ref{prop:H3estconv}.  Fix a large positive number $R_1$ so that  $R_1 \ge 8R_0$, with $R_0\ge \max(4R,2\rho_{0})$ and $R=\| Q\|_{H^1}$. 
For any $t\ge0$ let   
 \begin{equation}
\label{eq:functionH}
H_{\nu}(t):= E(\vec u(t)) - E(Q,0) + \frac{\varepsilon_0}{(1+t)^{a \nu}} \langle -\Delta u(t) +u(t)  -f(u(t)), u_t (t)\rangle_{H^{-1}},
\end{equation}
where $1\ge  \varepsilon_0>0$ will be specified later  and where $\nu >1$ is chosen so that $a\nu <\frac13$.  

 Let 
$T_{m}> n_m^{ 3/2}$ be the first time where $\| \vec u(t)\|_{\mathcal{H}}= R_1$, or $T_{m}=\infty$ if this never occurs, and define $\tilde I_{m}:=[n_m^{ 3/2}, T_m]$. 
\end{definition}

 Note that by Proposition~\ref{prop:H3estconv} and our choice of parameters, we have $T_m\ge (n_m+1)^{\frac32}$.  Thus, $\tilde I_{m}$ extends $I_{m}$ strictly to the right. 
See Lemma~\ref{lem:Hnu} for basic properties of $H_\nu$ relating to differentiability and monotonicity. 
The conditions to be satisfied by 
$\varepsilon_0$ will appear in the proof below. 
We notice that the condition $a\nu <1/3$ imposes an additional condition on $a$ only if $\nu >1$ is large. In fact, it suffices to take
$\nu >1$  close to~$1$.  The functional $H_{\nu}$ with $a= 0$ has been used by Haraux and Jendoubi in \cite{HaJen07} on  bounded domains and for globally bounded trajectories in the case where 
$1 < \theta \leq d/(d-2)$ and provided the damping is a positive constant. A functional of this type with $\nu=1$ was used by Haraux and Jendoubi in~\cite{HaJen13} for bounded trajectories for finite-dimensional gradient systems, and damping $\alpha(t) >0$ under a condition similar to~\ref{H1alpha}.
Amongst several others, our challenge is that we are dealing with potentially unbounded trajectories. 
However, we will use in a crucial way   that the trajectories are bounded on the time intervals $I_m$, with $m$ large enough. 
If we knew a priori that the trajectory is bounded, then the proof of Theorem \ref{ThBRS2} would be much simpler. 

We begin with an estimate on the nonlinear expression arising in $H_\nu'$. 

\begin{lemma} \label{lem:estmHprime}
Assume that  Hypothesis \ref{H2f} holds, with the restriction $1 <\theta \leq 4$ in the case $d=3$. Then, the following inequality holds
\begin{equation}
\label{eq:5ex3}
| \langle f'(u)u_t,u_t \rangle_{H^{-1}} | \leq c \| u_t\|_{L^2}^2 \| u\|_{H^1}^{\theta -1}.
\end{equation}
\end{lemma}
\begin{proof} 
{ a) The case $d=3$:} Assume that $1 < \theta \leq 4$, then we can for example  write
\begin{equation}
\label{eq:5ex00}
| \langle f'(u)u_t,u_t \rangle_{H^{-1}} | \leq \| f'(u)u_t\|_{H^{-2}} \| u_t\|_{L^2} 
\leq \sup_{\| \varphi\|_{H^2}=1} |\int_{\R^3} f'(u) u_t \varphi  dx | \| u_t\|_{L^2}.
\end{equation}
Using condition \ref{H2f} on the nonlinearity with  $\| \varphi\|_{H^2}=1$ fixed we have 
\begin{align} 
 |\int_{\R^3} f'(u) u_t \varphi  dx | &\leq \sup_{|v|\le 1} |f'(v)|  \| u_t\|_{L^2} \|\varphi\|_{L^2} + C \| |u|^{\theta-1} u_t \|_1 \|\varphi\|_\infty \\
 &\le C(1+ \| |u|^{\theta-1}  \|_{L^2} )  \|u_t\|_{L^2}, 
\end{align}
which implies, since $2(\theta-1) \leq 6$, that
\begin{equation}
\label{eq:5ex4}
| \langle f'(u)u_t,u_t \rangle_{H^{-1}} | \leq C (1+ \| u\|_{L^{2(\theta-1)}}^{\theta-1} )
 \| u_t\|_{L^2}^2 \leq C (1+\| u\|_{H^1}^{\theta-1} ) \| u_t\|_{L^2}^2,
\end{equation}
 and the estimate \eqref{eq:5ex3} is proved in this case.  
 
{b) The case $5 \leq d \leq 6$:} In this case, $H^2(\R^d)$ is embedded in $L^{2d/(d-4)}(\R^d)$. Applying the estimate \eqref{eq:5ex00} and the H\"older inequality, we get
\begin{align}
\label{eq:5ex6}
| \langle f'(u)u_t,u_t \rangle_{H^{-1}} | &\leq C\| u_t\|_{L^2}^2   \sup_{\| \varphi\|_{H^2}=1}\big( 
 \| \varphi\|_{L^{2d/ (d-4)}} \| u\|_{L^{2d/(d-2)}}^{\theta^*-1}  +    \| \varphi\|_{L^2} \big) \\
& \leq C  \| u_t\|_{L^2}^2 (1+\| u\|_{H^1}^{\theta^*-1}),
\end{align}
 and the estimate \eqref{eq:5ex3} is again proved in this case. 
 
 {c) The case $d=4$:} If $d=4$, $H^2(\R^d) \subset L^p(\R^d)$, for any $1 \leq p <  \infty$. This case is even simpler and is left to the reader.
\end{proof}
 
We can now derive basic properties of $H_\nu$. 

\begin{lemma}\label{lem:Hnu}
Assume that $\vec u \ne (Q,0)$. Then 
 $H_{\nu}$ is a decreasing $C^1$-function  as long as $\| \vec u(t)\|_{\mathcal{H}}<R_1$. In particular, $H_{\nu}$ is decreasing   on
 $\tilde I_{m}$ with  $m\geq m_0$.
\end{lemma}
\begin{proof}
The functional $H_{\nu}$ is well-defined and  continuous. Indeed, it suffices to use the continuity properties of $\vec u$ and the following estimates,
 \begin{equation}
\label{eq:5ex1}
| \langle -\Delta u +u, u_t \rangle_{H^{-1}} | \leq  \| u\|_{H^1} \| u_t \|_{H^{-1}} \leq C  \| u\|_{H^1} \| u_t\|_{L^2},
\end{equation}
and, due to Hypothesis \ref{H2f},
 \begin{equation}
\label{eq:5ex2}
| \langle f(u), u_t \rangle_{H^{-1}} | \leq  \| f(u)\|_{H^{-1}} \| u_t \|_{H^{-1}}
 \leq C  ( \| u\|_{H^1} + \| u\|_{H^1}^{\theta^*})
 \| u_t\|_{L^2}.
\end{equation}
To obtain the final estimate here, we note that 
\begin{equation}
\label{eq:5ex2BIS}
\begin{split}
\| f(u)\|_{H^{-1}} &\leq C \big( \| u\|_{L^2} + \| |u|^{\theta^*}\|_{H^{-1}}\big)
\leq C \big( \| u\|_{L^2} + \| |u|^{\theta^*}\|_{L^{2d/(d+2)}}\big) \cr
&\leq C \big(\| u\|_{L^2} + \| u\|_{H^1}^{\theta^*}\big),
\end{split}
\end{equation}
since $L^{2d/(d+2)}(\R^d) \embed H^{-1}(\R^d)$ and $\theta^* \cdot\frac{2d}{d+2} = \frac{2d}{d-2}= 2^*$.
Next, one needs to show that $H_{\nu}$ is differentiable and that the derivative is continuous. To show that $H_{\nu}$ is differentiable and $C^1$ in $t$, we first assume that the initial data $\vec u(0)$ belong to $H^2(\R^d) \times H^1(\R^d)$. Then the solution $\vec u(t)$ is classical and belongs to $C^1([0,  \infty), H^1(\R^d)) \cap C^0([0,  \infty), H^2(\R^d)) \cap C^2([0, \infty), L^2(\R^d))$. In this case, one can also compute the derivative of $H_{\nu}$ in a classical way and one finds that, for $t \geq 0$,
\begin{equation*}
\begin{split}
H_{\nu}'(t) = & -2\alpha(t) \| u_t(t)\|_{L^2}^2 - 
\frac{a \nu \varepsilon_0 }{(1+t)^{a\nu +1}} \langle -\Delta u+u -f(u), u_t 
\rangle_{H^{-1}} \cr
& + \frac{\varepsilon_0}{(1+t)^{a\nu}} \langle -\Delta u + u -f(u), u_{tt} \rangle_{H^{-1}} \cr
& +  \frac{\varepsilon_0}{(1+t)^{a\nu}} \langle -\Delta u_t + u_t - f'(u)u_t,u_t 
\rangle_{H^{-1}},
\end{split}
\end{equation*}
which implies that, for $t \geq 0$,
\begin{equation}
\label{eq:5auxequal1}
\begin{split}
H_{\nu}'(t) = & -2\alpha(t) \| u_t(t)\|_{L^2}^2 - 
\frac{\varepsilon_0 \nu a}{(1+t)^{a\nu +1}} \langle -\Delta u+u -f(u), u_t 
\rangle_{H^{-1}} \cr
& - \frac{\varepsilon_0}{(1+t)^{a\nu}} \|  -\Delta u + u -f(u)\|_{H^{-1}}^2 \cr
& - \frac{2\varepsilon_0}{(1+t)^{a(\nu +1)}} \langle -\Delta u + u -f(u), u_t \rangle_{H^{-1}} \cr
& +  \frac{\varepsilon_0}{(1+t)^{a\nu}} \langle -\Delta u_t + u_t - f'(u)u_t,u_t 
\rangle_{H^{-1}}.
\end{split}
\end{equation}
In the above expression of $H_{\nu}'(t)$, the first to the fourth terms are always well-defined (see \eqref{eq:5ex1} and \eqref{eq:5ex2}) and one has good bounds even if the initial data are only in $\mathcal{H}$. It remains to prove that it is also the case for the last term.
In the previous lemma we established the bound 
\begin{equation*}
| \langle f'(u)u_t,u_t \rangle_{H^{-1}} | \leq C \| u_t\|_{L^2}^2 ( 1+ \| u\|_{H^1}^{\theta -1}),
\end{equation*}
 see \eqref{eq:5ex3}.  In addition, one has the following   estimate
\begin{equation}
\label{eq:5ex7}
|\langle -\Delta u_t + u_t, u_t \rangle_{H^{-1}}| \leq C \| u_t\|_{L^2}^2.
\end{equation}
We next apply Lemma \ref{lem:estmHprime} and the Cauchy-Schwarz inequality to \eqref{eq:5auxequal1}. Thus there exist $\varepsilon_0 >0$ small enough, and $t_0>0$ large enough, so that for all $t \geq t_0$, 
\begin{equation}
\label{eq:5ex8}
H_{\nu}'(t)  \leq -\frac{(2-\varepsilon_0 C(R_1))}{(1+t)^{a}}\| u_t\|_{L^2}^2  - \frac{\varepsilon_0}{4(1+t)^{a\nu}}
\|  -\Delta u + u -f(u)\|_{H^{-1}}^2 \leq 0,
\end{equation}
 provided that $\| \vec u(t) \|_{\mathcal{H}} \leq R_1$,
and, say, 
 \begin{equation}
\label{eq:5ex9}
2-\varepsilon_0 C(R_1)  \geq 3/2. 
\end{equation}
If $H_{\nu}'(\tilde t)=0$, then $u_{t}=0$ and $-\Delta u + u - f(u)=0$ at time $t=\tilde t$.  Therefore, $u$ is an equilibrium $(\tilde Q,0)$. By Proposition~\ref{prop:H3estconv} we then must have $Q=\tilde Q$.  Thus in fact $H_{\nu}'(t)<0$. 
These estimates are still valid when the initial data belong to $\mathcal{H}$ only. Thus, arguing by density, we can show that $H_{\nu}$ is still of class $C^1$ even if the initial data only belong to~$\mathcal{H}$.
 \end{proof}
 
We derive the following conclusions from the previous lemma: 
Assume that the trajectory $\vec u(t)$ is bounded in $\mathcal{H}$ for $t \geq 0$ by say the constant $R_1$. In Theorem \ref{ThBRS1}, we have seen that there exists a sequence of times $t_{n_m}$ such that $\vec u(t_{n_m})$ converges to $(Q,0)$ as $t_{n_m}$ goes to infinity, which means that $H_{\nu}(t_{n_m})$ goes to $0$ when $t_{n_m}$ goes to infinity. On the other hand, since $\vec u(t)$ is bounded in $\mathcal{H}$ by the constant $R_1$, the function $H_{\nu}(t)$ is decreasing for $t\geq t_0$. Since it is bounded from below, the function $H_{\nu}(t)$ converges as $t$ goes to infinity and thus we may conclude that
 \begin{equation}
\label{eq:5auxest10}
\lim_{t \to + \infty} H_{\nu}(t) =0.
\end{equation}
Thus, without loss of generality, we may assume that, by choosing the above time $t_0>0$
 large enough, $H_{\nu}(t) >0$ for $t\geq t_0$.
  
  If we do not know a priori that the trajectory $\vec u(t)$ remains bounded in $\mathcal{H}$, then we cannot conclude that $H_{\nu}(t)$ stays positive for $t\geq t_0$. The above arguments (see \eqref{eq:5ex8}) only show that, for $t \in [t_0,T_1]$, where $[t_0,T_1]$ is defined as the maximal time interval (beginning at $t_0$) on which  
$\| u(t) \|_{\mathcal{H}} \leq R_1$,  the function $H_{\nu}$ is non increasing. It could be that on this time interval the function $H_{\nu}$ becomes negative. Indeed, the function $H_{\nu}$ could become again positive later, when $\| u(t) \|_{\mathcal{H}} > R_1$. This possibility makes the proof below more delicate, as several different cases will need to be considered.

 \subsection{\texorpdfstring{$H_\nu$}{H-nu} becomes negative at some time} 
  
The main technical result in this case is the following one. 

\begin{proposition}\label{prop:H-}
Assume that there exists $t_1 \in I_{m_0}$ so that $H_{\nu}(t_1)\le 0$, with $m_{0}$ sufficiently large. Denote the right end-point of $\tilde I_{m_{0}}$ by~$T_{1}$.   Then $H_{\nu}(t)$ is decreasing and thus $H_{\nu}(t) <0$ on $(t_1,T_1]$. Furthermore, if  $\nu$ is sufficiently close to~$1$, then 
\begin{equation}\label{eq:Hnegprop}
\int_{t_1}^{T_1}\| u_t(s)\|_{L^2}    ds  \leq C  t_1^{-a(\nu -1)},
\end{equation}
where $C$ is a positive absolute constant (independent of $m_0$).
\end{proposition} 

We remark that $T_{1}$ always lies to the right of $I_{m_0}$. 
The proof of Proposition~\ref{prop:H-} is quite involved, so we first show how to deduce $T_1=\infty$ from it. This is the key step in the proof  of Theorem~\ref{ThBRS2}. 
To this end, fix a small positive constant $\delta_0 \leq \frac{\rho_0}{8}$ (the precise condition will be specified later).
By Proposition~\ref{prop:H3estconv}, there exists $m_0$, so that, for $m \geq m_0$, we have,
\begin{equation}
\label{eq:5auxdefm0}
\| (Q,0) - \vec u(s)\|_{\mathcal{H}} \leq \frac{\delta_0}{2} \leq \frac{\rho_0}{16}, \quad \forall s \in I_{m}.
\end{equation}
We also choose $m_0$ large enough so that, for $t \geq t_0 \equiv n_{m_0}^{3/2}$, the non-increasing property \eqref{eq:5ex8} of Lemma~\ref{lem:Hnu} holds, as long as $t \geq t_0$ and provided that $\| \vec u(t)\|_{\mathcal{H}} \leq R_1$. 

\begin{cor}
\label{cor:H-1}
Suppose $\delta_0>0$ is sufficiently small and $m_0$ is large. Then under the conditions of Proposition~\ref{prop:H-}, 
 \begin{equation}
\label{eq:uQclose}
 \| \vec u(t) - (Q,0)\|_{\mathcal{H}} \leq \frac{\rho_0}{2},
\end{equation}
for all times $t\ge t_1$. In particular, $T_1= \infty$. 
\end{cor}
\begin{proof}
We  choose $m_0$ large so that 
\begin{align}
\label{eq:5exG12}
\int_{t_1}^{T_1}\| u_t(s)\|_{L^2}   ds &\leq C   t_1^{-a(\nu -1)} \leq \frac{\delta_0}{2}, \\
\| u(t_{1})-Q\|_{L^{2}} &\leq \frac{\delta_{0}}{2}.
\end{align}
By Corollary~\ref{cor:key} for all  $t \in [t_1,T_1]$,
 \begin{equation}
\label{eq:5exG15}
 \| \vec u(t) - (Q,0)\|_{\mathcal{H}} \leq \frac{\rho_0}{2},
\end{equation}
if $\delta_0>0$ is chosen small enough.  

Recall that $[t_1, T_1]$ is the maximal interval on which 
$\| \vec u(t)\|_{\mathcal{H}} \leq R_1$.   The property \eqref{eq:5exG15} says that 
on this same time interval, cf.~Definition~\ref{defi:51}, 
$$ 
\| \vec u(t)\|_{\mathcal{H}} \leq \frac{\rho_0}{2} + R_0/4 < R_0/2 < R_1/16,
$$
which leads to a contradiction if $T_1 <  \infty$. Thus we have shown that, if there exists $t_1 \in I_{m_0}$ so that $H_{\nu}(t_1)<0$, then $\vec u(t)$, $t \geq t_1$ never leaves the ball $B_{ \mathcal{H}}((Q,0),R_1/16)$ and thus that, for $t\geq t_1$, the estimate 
\eqref{eq:5exG15} holds for all times.
\end{proof}

 \begin{proof}[Proof of Proposition~\ref{prop:H-}]
 If $H_{\nu}(t)\le 0$, then
\begin{equation}\label{eq:Hkleiner0}
0 < E( \vec u(t)) - E((Q,0)) \leq - \frac{\varepsilon_0}{(1+t)^{a \nu}} \langle -\Delta u(t) +u(t)  -f(u(t)), u_t (t)\rangle_{H^{-1}}.
\end{equation}
Therefore on this time interval, we have the following two kinds of inequalities
\begin{equation}
\label{eq:5exG1}
\begin{split}
0 < E( \vec u(t)) - E((Q,0)) &\leq \frac{2 \varepsilon_0}{(1+t)^{a (\nu +1)}}  
\| u_t(t)\|_{H^{-1}}^2 + \frac{ \varepsilon_0}{2(1+t)^{a \nu}}\frac{\partial}{\partial t}
\| u_t(t) \|_{H^{-1}}^2 ,\cr 
0 < E( \vec u(t)) - E((Q,0)) &%\leq \frac{\varepsilon_0 C_{1}}{(1+t)^{a \nu}} \|  u_t (t)\|_{H^{-1}} 
\leq \frac{\varepsilon_0 C_{2}}{(1+t)^{a \nu}} \|  u_t (t)\|_{H^{-1}}\leq  \frac{\varepsilon_0 C_{2}}{(1+t)^{a \nu}} \|  u_t (t)\|_{L^2}  ,
\end{split}
\end{equation}
for some $C_2=C_2(R_1)>0$.  The estimates in the first line follows from~\eqref{eq:Hkleiner0} and the PDE, whereas the second line follows from \eqref{eq:Hkleiner0} and the estimates \eqref{eq:5ex1} and \eqref{eq:5ex2}.
To bound the time integral in Proposition~\ref{prop:H-}, we introduce the  {\em positive} functional
\begin{equation}
\label{eq:5exG2}
G(t) = E( \vec u(t)) - E((Q,0)),
\end{equation}
the derivative of which is 
\begin{equation}
\label{eq:5exG3}
G'(t) = -2 \alpha(t) \| u_t(t)\|_{L^2}^2.
\end{equation}
Positivity of $G$ is guaranteed since $\vec u$ is not stationary by assumption. 
The second line of~\eqref{eq:5exG1} and the equality \eqref{eq:5exG3} imply that $G$ satisfies the differential inequality 
\begin{equation}
\label{eq:5exG4}
G'(t) + C_3 (1+ t)^{a(2\nu -1)} G^2(t) \leq 0 
\end{equation}
on $[t_1,T_1]$, 
where $C_3= 2(\varepsilon_0C_2)^{-2}$. Integrating this inequality, we find that, for $t \in 
[t_1,T_1]$,
\begin{equation}
\label{eq:5exG5}
G(t) \leq   \Big(G^{-1}(t_1) + \frac{C_3}{a(2\nu -1) +1} [(1+t)^{a(2\nu -1) +1}
- (1+t_1)^{a(2\nu -1) +1}] \Big)^{-1} ,
\end{equation}
and also, by using the inequalities \eqref{eq:5exG1} and \eqref{eq:5auxdefm0},
\begin{equation}
\label{eq:5exG6}
\begin{split}
G(t) \leq &  \Big(\frac{8(1+t_1)^{a\nu}}{C_4 \rho_0} + \frac{C_3}{a(2\nu -1) +1} [(1+t)^{a(2\nu -1) +1}- (1+t_1)^{a(2\nu -1) +1}] \Big)^{-1}\cr
\leq &  \Big(\frac{(1+t_1)^{a\nu}}{C_4 R_1} + \frac{C_3}{a(2\nu -1) +1} [(1+t)^{a(2\nu -1) +1}- (1+t_1)^{a(2\nu -1) +1}] \Big)^{-1},
\end{split}
\end{equation}
where $C_4= \varepsilon_0 C_2$. 

We next estimate the quantity $\int_{t}^{t+\tau}\| u_t(s)\|_{L^2}   ds$, for any $t\ge0$, $\tau\ge0$. 
 From \eqref{eq:5exG3} and the Cauchy-Schwarz inequality, 
\begin{equation}
\label{eq:5exG7}
\begin{split}
\int_{t}^{t+\tau}\| u_t(s)\|_{L^2}  ds \leq  & 
\big( \int_{t}^{t+\tau} (1+s)^{a}   ds \big)^{1/2} 
\big(2 \int_{t}^{t+\tau} \alpha(s) \| u_t(s)\|_{L^2}^2  ds \big)^{1/2} \cr
\leq  &  \tau^{1/2}(1 +t +\tau)^{a/2} (G(t) -G(t+\tau))^{1/2}\\
  \leq  &
 \tau^{1/2}(1 +t +\tau)^{a/2} G(t)^{1/2}. 
\end{split}
\end{equation}
In a first step, we shall assume that  $t_1 \leq t  \leq T_1$ and use~\eqref{eq:5exG6} to control $G(t)$.  Note that we do not require that $ t+\tau\leq T_1$. Indeed, the value of $G(t+\tau)$ plays no role in the previous estimate, so \eqref{eq:5exG6} only enters in the 
estimation of~$G(t)$.

\subsubsection{The case $2t_1 \leq t    \leq T_1$.}
 
 We deduce from 
\eqref{eq:5exG6} and \eqref{eq:5exG7} that,
\begin{equation}
\label{eq:5exG9}
\begin{split}
\int_{t}^{2t}\| u_t(s)\|_{L^2}   ds \leq & \frac{2^{a/2} (a(2\nu -1) +1)^{1/2}
t^{1/2} (1 + t)^{a/2}}{C_3^{1/2}  (1+t)^{(a(2\nu -1)+1)/2} 
[1- (\frac{1+ t_1}{1+2t_1})^{a(2\nu -1)+1}]^{1/2}}\cr
\leq & \frac{C_5(a(2\nu -1) +1)^{1/2}(1+t)^{-a(\nu -1)}}
{[1- (\frac{1+ t_1}{1+2t_1})^{a(2\nu -1)+1}]^{1/2}}.
\end{split}
\end{equation}
Without loss of generality, we may assume that we have chosen $m_0$ large enough so that $t_0\equiv n_{m_0}^{\gamma}$, $\gamma=\frac32$,  satisfies
$$
\frac{1+ t_1}{1+2t_1}  \leq \frac{1+ t_0}{1+2t_0}  \leq \frac{3}{4}.
$$
Thus, we deduce from the estimate \eqref{eq:5exG9} that
\begin{equation}
\label{eq:5exG10}
\int_{t}^{2t}\| u_t(s)\|_{L^2}   ds \leq C_6 (1+t)^{-a(\nu -1)},
\end{equation}
where $C_6= 2 C_5(a(2\nu -1) +1)^{1/2}$. Let $C_7= \max(C_5,C_6)$.
Since  $\nu>1$ and $a>0$,  the estimates 
\eqref{eq:5exG9} and \eqref{eq:5exG10} imply that
\begin{equation}
\label{eq:5exG11}
\int_{2t_1}^{T_1}\| u_t(s)\|_{L^2}   ds \leq  \sum_{n=1}^{\infty} \frac{C_7}{2^{an(\nu -1)}
t_1^{a(\nu -1)}} \leq \frac{C_7}{t_1^{a(\nu -1)} (1- 2^{-a(\nu -1)})} \leq C_8 
 t_1^{-a(\nu -1)}.
\end{equation}

\subsubsection{The case $t_1 \leq t  \leq \min(2t_1,T_1)$.} 
 
We set $T_2:=\min(2t_1,T_1)$. 
Going back to the inequality \eqref{eq:5exG1}, we now exploit the inequality on the first line. To do so, we need to take  the sign of  $\frac{d}{dt}\| u_t(t) \|_{H^{-1}}^2$ into account.

\noindent {\bf Case 1} $ \| u_t(t) \|_{H^{-1}}$ is nondecreasing on some interval $[t_1, t_1+\tau]\subset [t_1, T_2]$, with~$\tau>0$.

 In particular, $\frac{d}{dt}\| u_t(t) \|_{H^{-1}}\ge 0$ for all $t_1\le t\le t_1+\tau$. We now maximize $\tau$ with this property. I.e., $t^*=t_1+\tau\le T_2$ is maximal with the property that $ \| u_t(t) \|_{H^{-1}}$ is nondecreasing on $[t_1, t^*]$. Note that if $t^*<T_2$, then there exists a decreasing sequence $\{s_k \}_{k=1}^\infty$ in $[t^*, T_2]$ with 
\begin{equation}\label{eq:tstar} 
\frac{d}{dt}\| u_t(s_k) \|_{H^{-1}} <  0=\frac{d}{dt}\| u_t(t^*) \|_{H^{-1}}, \qquad s_k\to t^* .
\end{equation}
We remark that we cannot claim that   $\| u_t(t^*) \|_{H^{-1}}$ is a maximum, since $\| u_t(t) \|_{H^{-1}}$ might oscillate to the right of~$t=t^*$.
Let us recall that  
$$
\| v\|_{H^{-1}} \leq   \| v\|_{L^2}, \quad \forall v \in L^2,
$$
as can be seen by the Fourier transform and Plancherel. Indeed, 
\begin{equation*}
\| v\|_{H^{-1}} = \| (1+4\pi^2|\xi|^2)^{-\frac12} \hat{v}(\xi)\|_{L^2} \le \| \hat{v}(\xi)\|_{L^2} = \| v\|_{L^2}.
\end{equation*}
{}From the property \eqref{eq:5exG1}, we deduce that
\begin{equation}
\label{eq:5exG5A}
2 \int_{t_1}^{t^*} \alpha(s) \| u_{t}(s)\|_{L^2}^2   ds \leq \frac{\varepsilon_0 C_{1}}{(1+t_1)^{a \nu}} \|  u_t (t_1)\|_{H^{-1}}.
\end{equation}
Thus, there exists $0 \leq \tau_1 \leq t^* -t_1$ so that
\begin{equation*}
2 \alpha(t_1 +\tau_1) \| u_{t}(t_1+\tau_1)\|_{L^2}^2 \leq \frac{\varepsilon_0 C_{1}}{(t^* -t_1)(1+t_1)^{a \nu}} \|  u_t (t_1)\|_{H^{-1}},
\end{equation*} 
and hence
\begin{equation}
\label{eq:5exG5B}
 \| u_{t}(t_1+\tau_1)\|_{L^2}^2 \leq \frac{  \varepsilon_0 C_{1}}{ (t^* -t_1)(1+t_1)^{a (\nu -1)}}\|  u_t (t_1)\|_{H^{-1}}.
\end{equation}
Since by our nondecreasing assumption on $t_1$,
\begin{equation}
\label{eq:5exG5C}
 \| u_{t}(t_1)\|_{H^{-1}} \leq \|  u_t (t_1+\tau_1)\|_{H^{-1}},
\end{equation}
it follows from \eqref{eq:5exG5B} and \eqref{eq:5exG5C} that
\begin{equation*}
 \| u_{t}(t_1)\|_{H^{-1}}^2 \leq       \| u_{t}(t_1 +\tau_1)\|_{L^2}^2 \leq 
\frac{     \varepsilon_0 C_{1}}{ (t^* -t_1)(1+t_1)^{a (\nu -1)}}\|  u_t (t_1)\|_{H^{-1}},
\end{equation*} 
and 
\begin{equation}
\label{eq:5exG5D}
 \| u_{t}(t_1)\|_{H^{-1}} \leq \frac{     \varepsilon_0 C_{1}}{ (t^* -t_1)(1+t_1)^{a (\nu -1)}}.
\end{equation}
Using again the inequality \eqref{eq:5exG1}, we deduce from \eqref{eq:5exG5D} and~\eqref{eq:5exG5A} (replacing $t^{*}$ on the left-hand side by~$\infty$) that
\begin{equation}
\label{eq:5exG5E}
2\int_{t_1}^{ \infty}\alpha(s) \| u_t(s)\|_{L^2}^2   ds \leq \frac{\varepsilon_0^2     C_1^2}{(t^* -t_1)(1+t_1)^{a(2\nu -1)}}.
\end{equation}
Let $\ell$ be a fixed positive number, which will made more precise later.  It follows from \eqref{eq:5exG5E} that
\begin{equation}
\label{eq:5exG5F}
\begin{split}
\int_{t_1}^{t^*+ \ell}\| u_t(s)\|_{L^2}    ds &\leq  2^{a/2} C(\ell) (t^* -t_1)^{1/2} (1 +t_1)^{a/2}\Big( \int_{t_1}^{ \infty}\alpha(s) \| u_t(s)\|_{L^2}^2   ds \Big)^{1/2} \cr
&\leq \frac{\varepsilon_0   C_1 C(\ell)}{ (1 +t_1)^{a(\nu -1)}}. 
\end{split}
\end{equation}
Since we may take $t_1$ larger than any fixed time, it follows from \eqref{eq:5exG5F} and Corollary~\ref{cor:key} that $T_1$, and therefore $T_2$,  cannot be arbitrarily close to~$t_1$. 
More precisely, for any fixed  $L>0$ we can assume that $T_1\ge T_2 \ge t_1+L$.

It remains to estimate the term $\int_{t^*}^{T_{2}}\| u_t(s)\|_{L^2}   ds$ when 
$t^* + \ell\leq  T_{2}$.  
We  set
\begin{equation}
\label{eq:5exdefell}
\ell = \big[\frac{1 +2 a(\nu -1)}{a\nu}\big] +1,
\end{equation}
where $\big[x \big]$ denotes the integral part of $x$. We will show by recursion that there exists a finite sequence of numbers $0 \leq  \tilde{\tau}_j \leq 1$, $j=1,2, \ldots, \ell -1$, such that
\begin{equation}
\label{eq:5exG5G}
\int_{t^* +  \sum_{j=1}^{j= \ell -1}\tilde{\tau}_j}^{\infty} \frac{1}{(1+s)^{a}}
 \| u_t(s)\|_{L^2}^2   ds 
 \leq  \frac{\varepsilon_0^{\ell}  
 2^{2(\ell-1)a}}{(1+t^*)^{a(\nu+1) + (\ell -1)a\nu}} \|  u_t (t^*)\|_{H^{-1}}^2.
\end{equation}
{}From the first line of~\eqref{eq:5exG1}  and the vanishing in \eqref{eq:tstar}, we deduce that
\begin{equation}
\label{eq:5exG5H1}
\int_{t^*}^{ \infty} \frac{1}{(1+s)^{a}}  \| u_{t}(s)\|_{L^2}^2    ds \leq \frac{\varepsilon_0}
{(1+t^*)^{a (\nu+1)}} \|  u_t (t^*)\|_{H^{-1}}^2.
\end{equation}
Hence, there exists   $\tau_1$ with $0\leq \tau_1 \leq 1$ such that 
\begin{equation}
\label{eq:5exG5H2}
\| u_{t}(t^* +\tau_1)\|_{L^2}^2  \leq \frac{\varepsilon_0 (1 +t^* +\tau_1)^{a}}
{(1+t^*)^{a (\nu+1)}} \|  u_t (t^*)\|_{H^{-1}}^2 \leq 
\frac{\varepsilon_0 2^{a}}
{(1+t^*)^{a \nu}} \|  u_t (t^*)\|_{H^{-1}}^2 .
\end{equation}
If $\frac{d}{dt}\|  u_t (t^*+ \tau_1)\|_{H^{-1}}^2 \leq 0$, we set 
$\tilde{\tau}_1 = \tau_1$. If this is not the case, we may choose $0 <\tilde{\tau}_1 < \tau_1$ so that $t^* + \tilde{\tau}_1$ is the closest time on the left of $t^* +\tau_1$ so that 
$\frac{d}{dt}\|  u_t (t^*+ \tilde{\tau}_1)\|_{H^{-1}}^2 \leq 0$ (in fact, one then has $=0$ here). The existence of $\tilde{\tau}_1$ follows from the intermediate value theorem in view of the sequence $s_k$ in~\eqref{eq:tstar}. 
Consequently, 
\begin{equation}
\label{eq:5exG5H3}
\|  u_t (t^*+ \tilde{\tau}_1)\|_{H^{-1}} <  \|  u_t (t^*+ \tau_1)\|_{H^{-1}} \leq  
\|  u_t (t^*+ \tau_1)\|_{L^2}.
\end{equation}
Applying again the property \eqref{eq:5exG1}, we obtain,
 \begin{equation}
\label{eq:5exG5H4}
 \begin{split}
\int_{t^*+ \tilde{\tau}_1}^{ \infty} \frac{1}{(1+s)^{a}}  \| u_{t}(s)\|_{L^2}^2    ds & \leq
\frac{\varepsilon_0} {(1+t^* + \tilde{\tau}_1)^{a (\nu+1)}} \|  u_t (t^*+ \tilde{\tau}_1)\|_{H^{-1}}^2 \cr
&\leq \frac{\varepsilon_0    } {(1+t^* + \tilde{\tau}_1)^{a (\nu+1)}} \|  u_t (t^*+ {\tau}_1)\|_{L^2}^2 \cr
&  \leq \frac{\varepsilon_0^2     2^{a}} {(1+t^* + \tilde{\tau}_1)^{a (\nu+1)} (1+t^*)^{a \nu}}
 \|  u_t (t^*)\|_{H^{-1}}^2,
 \end{split}
\end{equation}
where the final estimate uses~\eqref{eq:5exG5H2}. 
Therefore, there exists   $\tau_2$ with $0\leq \tau_2 \leq 1$ such that 
\begin{equation}
\label{eq:5exG5H5}
\| u_{t}(t^* + \tilde{\tau}_1 +\tau_2)\|_{L^2}^2  \leq  \frac{\varepsilon_0^2     2^{2a}} 
{(1+t^* + \tilde{\tau}_1)^{a \nu} (1+t^*)^{a \nu}} \|  u_t (t^*)\|_{H^{-1}}^2.
\end{equation}
We proceed as before. If $\frac{d}{dt}\| u_t(t^* + \tilde{\tau}_1 +\tau_2)\|_{H^{-1}}^2 \leq 0$, we set 
$\tilde{\tau}_2 = \tau_2$. If this is not the case, we may choose $0 \leq \tilde{\tau}_2 < \tau_2 \leq 1$ so that $t^* + \tilde{\tau}_1 + \tilde{\tau}_2$ is the closest time on the left of $t^* +\tilde{\tau}_1 +\tau_2$ so that 
$\frac{d}{dt}\|  u_t (t^*+\tilde{\tau}_1 +\tilde{\tau}_2)\|_{H^{-1}}^2 \leq 0$.
It is important to note that we allow $\tilde\tau_{2}=0$ at this step. 
Thus, using in view of \eqref{eq:5exG1}, and~\eqref{eq:5exG5H5} we obtain,
 \begin{equation}
\label{eq:5exG5H6}
 \begin{split}
&\int_{t^*+ \tilde{\tau}_1+ \tilde{\tau}_2 }^{ \infty} \frac{1}{(1+s)^{a}}  \| u_{t}(s)\|_{L^2}^2   ds \\ & \leq
\frac{\varepsilon_0} {(1+t^* + \tilde{\tau}_1+ \tilde{\tau}_2)^{a (\nu+1)}} \|  u_t (t^*+ \tilde{\tau}_1+ \tilde{\tau}_2)\|_{H^{-1}}^2 \cr
&\leq \frac{\varepsilon_0    } {(1+t^* + \tilde{\tau}_1+ \tilde{\tau}_2)^{a (\nu+1)}} 
\|  u_t (t^*+ \tilde{\tau}_1+{\tau}_2)\|_{L^2}^2 \cr
&  \leq \frac{\varepsilon_0^3  2^{2a}} {(1+t^* + \tilde{\tau}_1+\tilde{\tau}_2)^{a (\nu+1)} 
(1+t^*+\tilde{\tau}_1)^{a \nu} (1+t^*)^{a \nu}}
 \|  u_t (t^*)\|_{H^{-1}}^2.
 \end{split}
\end{equation}
Arguing by recursion, we finally prove that there exists a sequence of  numbers $0 \leq  \tilde{\tau}_j \leq 1$, $j=1,2, \ldots, \ell -1$, such that
\begin{equation}
\label{eq:5exG5H7}
 \begin{split}
 \int_{t^* + \ell-1}^{\infty} \frac{1}{(1+s)^{a}} \| u_t(s)\|_{L^2}^2   ds &\leq 
  \int_{t^* +  \sum_{j=1}^{j= \ell -1}\tilde{\tau}_j}^{\infty} \frac{1}{(1+s)^{a}} 
  \| u_t(s)\|_{L^2}^2   ds \cr
  & \leq  \frac{\varepsilon_0^{\ell}   2^{2(\ell-1)a}}{(1+t^*)^{a(\nu+1) + (\ell -1)a\nu}} \|  u_t (t^*)\|_{H^{-1}}^2 \cr
 &\leq  \frac{ C^*(\ell) R_1^2}{(1+t^*)^{\ell a\nu + a}},
 \end{split}
\end{equation}
where $C^*(\ell)$ is a positive constant depending on $\ell$ only. 
 It follows from \eqref{eq:5exG5H7} that
\begin{equation}
\label{eq:5exG5H8}
\begin{split}
\int_{t^*+\ell}^{T_{2}}\| u_t(s)\|_{L^2}  ds &\leq 
 (2t_1)^{1/2}(1 + 2t_1)^{a/2} \Big( \int_{t^* + \ell-1}^{\infty} \frac{1}{(1+s)^{a}} \| u_{t}(s)\|_{L^2}^2   ds \Big)^{1/2} \cr
 & \leq  \frac{C^{**}(\ell) R_1}{(1+t_1)^{a(\nu -1)}}  ,
\end{split}
\end{equation}
where $C^{**}(\ell)$ is a positive constant depending on $\ell$ only.  The final estimate here uses the definition~\eqref{eq:5exdefell}. 

 The inequalities \eqref{eq:5exG11}, \eqref{eq:5exG5F}, and 
\eqref{eq:5exG5H8} imply that
\begin{equation}
\label{eq:5exG5J}
\int_{t_1}^{T_1}\| u_t(s)\|_{L^2}   ds  \leq C_8^*  t_1^{-a(\nu -1)},
\end{equation}
where $C_8^* >0$ is a positive constant (independent of $m_0$).

\noindent {\bf Case 2} $ \| u_t(t) \|_{H^{-1}}$ fails to be nondecreasing on any interval $[t_1, t_1+\tau]\subset [t_1, T_2]$, with~$\tau>0$.  

 This means that either (i) $\frac{d}{dt}\| u_t(t_{1}) \|_{H^{-1}}<0$ or (ii) $\frac{d}{dt}\| u_t(t_{1}) \|_{H^{-1}}=0$. In the latter case, there exists a decreasing sequence 
$s_k\to t_1$ as in~\eqref{eq:tstar}. Clearly, this sequence also exists if we have negativity as in (i).   Setting $t^*=t_1$ we may therefore assert that~\eqref{eq:5exG5H8} holds with $t^*=t_1$. Indeed, the recursive procedure that lead to this bound did not require vanishing in~\eqref{eq:tstar} at~$t^*$, but only nonpositivity. So the same argument goes through here, too. It remains to establish the analogue of~\eqref{eq:5exG5F} for the integral over $[t_1, t_1+\ell]$, i.e., 
\begin{equation}
\label{eq:5exG5F*}
\begin{split}
\int_{t_1}^{t_1+ \ell}\| u_t(s)\|_{L^2}   ds &\leq   
 \frac{\varepsilon_0   C_1 C(\ell)}{ (1 +t_1)^{a(\nu -1)}}. 
\end{split}
\end{equation}
 Since  $ \frac{d}{dt} \| u_t(t_1) \|_{H^{-1}} \leq 0$, we will apply the first inequality in \eqref{eq:5exG1}, which implies that
$$
\int_{t_1}^{+\infty} \alpha(s) \| u_t(s)\|_{L^2}^{2}   ds \leq  \frac{2 \varepsilon_0 C(R_1)}{(1 +t_1)^{a(\nu +1)}},
$$
and thus, by using a H\"older inequality, we obtain,
\begin{equation}
\label{eq:5exG5F*BIS}
\begin{split}
\int_{t_1}^{t_1+ \ell}\| u_t(s)\|_{L^2}    ds \leq & \Big(\int_{t_1}^{t_1+ \ell} (1+s)^{a} ds\Big)^{1/2} \Big(\int_{t_1}^{t_1 +\ell} \alpha(s) \| u_t(s)\|_{L^2}^2  ds \Big)^{1/2} \\
\leq & \frac{(2 \ell \varepsilon_0 C(R_1))^{1/2} (1+t_1 + \ell)^{a/2}}{(1+ t_1)^{a(\nu +1)/2}} \leq C(\ell, R_1) (1+ t_1)^{-a \nu/2} \\
\leq & C(\ell, R_1) (1+ t_1)^{-a (\nu -1)},
\end{split}
\end{equation}
since we could choose $\nu >1$ close to $1$ (choose $1 < \nu \leq 2$ so that $a\nu <1/3$).
\end{proof}
%%%%%%%%%%%%%%%%%%%%%%%%%%%%%%%
%%%%%%%%%%%%%%%%%%%%%%%%%%%%%%%%%%

\subsection{\texorpdfstring{$H_\nu\ge0$}{H-nu>= 0} and   the {\L}ojasiewicz-Simon inequality.}\label{subsec:LS}

In this section we assume that $H_{\nu}(t_0) >0$, where 
$t_0= n_{m_0}^{\gamma}+1$, with $\gamma=\frac32$,   and also that $H_{\nu}(t) >0$  on the maximal time interval $[t_0, T_1]$ on which $\| \vec u(t)\|_{ \mathcal{H}} \leq R_1$. 
We recall that the functional $H_{\nu}$ can be written as 
 \begin{equation}
\label{eq:functionHBIS}
H_{\nu}(t)= J(u(t)) - J(Q) + \frac{1}{2} \| u_t(t)\|_{L^2}^2 +  \frac{\varepsilon_0}{(1+t)^{a \nu}} \langle -\Delta u(t) +u(t)  -f(u(t)), u_t (t)\rangle_{H^{-1}}.
\end{equation}
Theorem \ref{th:Loja} of Appendix~A implies that, as long as $\| u(t) -Q\|_{H^1} \leq \rho_0$ and in  particular, as long as $\| \vec u(t) - (Q,0)\|_{\mathcal{H}} \leq \rho_0$,
we have
\begin{equation}
\label{eq:5auxLoja11}
\begin{split}
H_{\nu}(t) \leq & C_0 \| -\Delta u(t) + u(t)  -f(u(t))\|_{H^{-1}}^2 
+ \frac{1}{2}\| u_t(t)\|_{L^2}^2 \cr
&+  \frac{\varepsilon_0}{2(1+t)^{a \nu}} 
\big( \| -\Delta u(t) + u(t)  -f(u(t))\|_{H^{-1}}^2 + \| u_t(t)\|_{L^2}^2\big) \cr
\leq & C_1 \big(  \| -\Delta u(t) + u(t)  -f(u(t))\|_{H^{-1}}^2 + \| u_t(t)\|_{L^2}^2\big).
\end{split}
\end{equation}
This implies in particular that $H_\nu(t) \le C(R_1)$ for all $t\in [t_0, T_1]$. 
{}From the estimates \eqref{eq:5ex8} and \eqref{eq:5auxLoja11}, we deduce that, for $t \geq t_0$, as long as $\| \vec u(t) - (Q,0)\|_{\mathcal{H}} \leq \rho_0$, we have
\begin{equation}
\label{eq:5auxLoja12}
- H_{\nu}'(t) \geq \frac{C_2 \varepsilon_0}{(1+t)^{a\nu}} H_{\nu}(t),
\end{equation}
or also
\begin{equation}
\label{eq:5auxLoja12BIS}
 H_{\nu}'(t) + \frac{C_2 \varepsilon_0}{(1+t)^{a\nu}} H_\nu(t) \leq 0 .
\end{equation}
We now set 
\begin{equation}
\label{eq:5auxdeft0*}
t_{0}^{*} = \sup \{t \geq t_0  |  \| \vec u(s) - (Q,0)\|_{ \mathcal{H}} \le  \rho_0, \quad \forall s \in [t_0,t]\}.
\end{equation}
By the choice of parameters, $t_0^*<T_1$. 
We recall that we have chosen $t_0=n_{m_0}^{\gamma}+1$, see \eqref{eq:5auxdefm0}. Hence 
$$
\| (Q,0) - \vec u(s)\|_{\mathcal{H}} \leq \frac{\delta_0}{2} \leq \frac{\rho_0}{16}, \quad \forall s \in I_{m_{0}}.
$$
This implies that $t_0 \leq (n_{m_0} +1)^{\gamma} < t_{0}^{*}$.  
Applying Lemma \ref{lem:suite3}, we can also choose $m_0$ so that
\begin{equation}
\label{eq:5auxintegralt0t0}
\int_{n_{m_0}^{\gamma}}^{(n_{m_0}+1)^{\gamma}} \| u_t(s)\|_{L^2}   ds \leq \frac{\delta_0}{2} 
\leq \frac{\rho_0}{16}.
\end{equation}
Note that 
\begin{equation}
\label{eq:5auxnm0gamma}
(n_{m_0}+1)^{\gamma}-  n_{m_0}^{\gamma} \sim n_{m_0}^{\frac12}\sim t_{0}^{\frac13}. 
\end{equation}
Note that the differential inequality \eqref{eq:5auxLoja12BIS} is only useful provided  $H_{\nu}(t) >0$. Integrating 
\eqref{eq:5auxLoja12BIS} yields, for $t_0 \leq t \leq t_0^{*}$,
\begin{equation}
\label{eq:5decroisH}
\begin{split}
H_{\nu}(t) \leq & H_{\nu}(t_0) \exp \Big(-\int_{t_0}^t \frac{C_2 \varepsilon_0}{(1+s)^{a\nu}}   ds\Big) \cr
\leq & H_{\nu}(t_0) \exp \Big(-  \frac{C_2 \varepsilon_0}{1 -a\nu} 
[ (1+ t)^{1-a\nu} -(1+ t_0)^{1-a\nu}] \Big).
\end{split}
\end{equation}
In view of  \eqref{eq:5ex8}, \eqref{eq:5ex9} 
\begin{equation}
\label{eq:5ex13}
\int_{t}^{(1+\beta)t} \frac{3}{2(1+ s)^{a}} \| u_t(s)\|_{L^2}^2   ds \leq H_{\nu}(t) 
- H_{\nu}((1+\beta)t),
\end{equation}
 for $t_0 \leq t  \leq (1+\beta) t\leq t_0^{*}$, where $0 < \beta \leq 1$. 
Thus, since $H_{\nu} (t) \ge 0$ on $t_0 \leq t\leq T_1$, we infer that
\begin{equation}
\label{eq:5ex13BIS}
\int_{t}^{(1+\beta)t} \frac{1}{(1+ s)^{a}} \| u_t(s)\|_{L^2}^2    ds \leq \frac{2H_{\nu}(t_0)}{3}
 \exp \Big(-  \frac{C_2 \varepsilon_0}{1 -a\nu} 
[ (1+ t)^{1-a\nu} -(1+ t_0)^{1-a\nu}] \Big).
\end{equation}
By Cauchy-Schwarz,   for all $t_0 \leq t  \leq (1+\beta) t\leq t_0^*$, 
\begin{multline}
\label{eq:5ex14}
\int_{t}^{(1+\beta)t}  \| u_t(s)\|_{L^2}   ds \leq \Big(\int_{t}^{(1+\beta)t} (1+ s)^{a}   ds\Big)^{1/2}
\Big(\int_{t}^{(1+\beta)t}\frac{1}{(1+ s)^{a}} \| u_t(s)\|_{L^2}^2   ds\Big)^{1/2}\cr
\leq  \beta^{1/2} t^{1/2}(1+t)^{a/2} H_{\nu}(t_0)^{1/2}
 \exp \Big(-  \frac{C_2 \varepsilon_0}{2(1 -a\nu)} 
[ (1+ t)^{1-a\nu} -(1+ t_0)^{1-a\nu}] \Big).
\end{multline}
If $t=t_0 +\tau$ where $0 < \tau \leq t_0$, we see that
$$
 \exp \Big( -  \frac{C_2 \varepsilon_0}{2(1 -a\nu)} 
[ (1+ t)^{1-a\nu} -(1+ t_0)^{1-a\nu}] \Big) = \exp \Big( - 
\frac{\tau C_2 \varepsilon_0 }{2(1+ t_0')^{a\nu}} \Big),
$$
for some $t_{0}\le t_{0}'\le t$. For the 
 right-hand side to be  small we need $\tau \geq t_0^{a\nu + \eta}$ where $\eta>0$. 
Thus, since $a\nu<\frac13$ by our assumptions, we  fix $a\nu + \eta=\frac13$. 
In view of~\eqref{eq:5auxintegralt0t0} and \eqref{eq:5auxnm0gamma} the interval  
$t - t_0 \leq t_0^{\frac13}$ has already been dealt with. 

Let us now take $t \leq t_0^{*}$ so that $ t_0 +t_0^{\frac13} \leq t \leq 2t_0$ and set $t= t_0 +t_0^{\frac13}+\tau$. Then the estimate \eqref{eq:5ex14} implies that 
\begin{multline}
\label{eq:5ex15}
\int_{t_0 +t_0^{\frac13}}^{t_0 +t_0^{\frac13}+ \tau} \| u_t(s)\|_{L^2}    ds \\ 
  \leq 2 t_0^{1/2}(1+t_0)^{a/2} H_{\nu}(t_0)^{1/2}
 \exp \Big( -  \frac{C_2 \varepsilon_0}{2(1 -a\nu)} 
[ (1+ t_0 + t_0^{\frac13})^{1-a\nu} -(1+ t_0)^{1-a\nu}] \Big)  \\
\leq 2 t_0^{1/2}(1+t_0)^{a/2} H_{\nu}(t_0)^{1/2}
\exp \Big( -  \frac{C_2 \varepsilon_0 t_0^{\frac13}}{2(1+ t_0)^{a\nu}}   \Big).
\end{multline}
Since $H_\nu$ is uniformly bounded on $[t_0, T_1]$ by $C(R_1)$,  the right-hand side of the estimate \eqref{eq:5ex15} tends to $0$ as $t_0$ goes to infinity. Therefore, we choose $m_0$ (and thus $t_0$) large enough so that
 \begin{equation}
\label{eq:5ex16}
\int_{t_0 +t_0^{\frac13}}^{t_0 +t_0^{\frac13}+ \tau} \| u_t(s)\|_{L^2}   ds
\leq C_9 t_0^{-\mu} \leq \frac{\delta_0}{2} \leq \frac{\rho_0}{16},
\end{equation}
where $\mu >0$. If $2t_0 \leq t \leq t_0^{*}$, we deduce from \eqref{eq:5ex14} that 
\begin{multline}
\label{eq:5ex17}
\int_{t}^{2t}  \| u_t(s)\|_{L^2}   ds \\
\leq 2  t^{1/2}(1+t)^{a/2} H_{\nu}(t_0)^{1/2}
 \exp \Big( -  \frac{C_2 \varepsilon_0}{2(1 -a\nu)} 
 (1+ t)^{1-a\nu} \big (1 - (\frac{1+ t_0}{1+ 2 t_0})^{1-a\nu} \big)   \Big) .
\end{multline}
As in the analysis of negative $H_\nu$ above, we may choose $m_0$ large enough so that 
$$
\frac{1+t_0}{1+2t_0}\leq \frac{3}{4}.
$$
Then, the inequality \eqref{eq:5ex17} implies that
\begin{equation}
\label{eq:5ex18}
\begin{split}
\int_{t}^{2t}  \| u_t(s)\|_{L^2}   ds 
\leq 2  t^{1/2}(1+t)^{a/2} H_{\nu}(t_0)^{1/2}
 \exp \Big( -  \frac{C_2 \varepsilon_0}{8(1 -a\nu)} 
 (1+ t)^{1-a\nu}  \Big).
\end{split}
\end{equation}
Choosing $m_0$ (that is $t_0$) large enough, which in turn implies that $t$ is large enough, we deduce from the estimate \eqref{eq:5ex18} that
\begin{equation}
\label{eq:5ex19}
\int_{t}^{2t}  \| u_t(s)\|_{L^2}   ds \leq C_{10} (1+ t)^{-\mu}.
\end{equation}
As in~\eqref{eq:5exG11}, the estimates \eqref{eq:5ex16} and \eqref{eq:5ex19} imply that, for $m_0$ large enough, we have 
\begin{equation}
\label{eq:5ex20}
\int_{t_0 + t_0^{\frac13}}^{t_0^{*}}  \| u_t(s)\|_{L^2}   ds
\leq  \sum_{n=0}^{\infty} \frac{C_{11}}{2^{n \mu}
t_0^{\mu}} \leq \frac{C_{11}}{t_0^{\mu} (1- 2^{-\mu})} \leq C_{12} t_0^{-\mu}
 ,
\end{equation}
where $C_{11}= \max(C_9,C_{10})$.  Choosing $m_0$ larger if necessary, we can ensure that  
\begin{equation} 
\label{eq:5ex21}
\int_{t_0 + t_0^{\frac13}}^{t_0^{*}}\| u_t(s)\|_{L^2}   ds \leq C_{12}  t_0^{-\mu}
 \leq \frac{\delta_0}{2}.   
 \end{equation}
By Corollary~\ref{cor:key}, we conclude that for all $t \in
 [t_0+ t_0^{\frac13}, t_0^{*}]$,
 \begin{equation}
\label{eq:5ex24}
 \| \vec u(t) - (Q,0)\|_{\mathcal{H}} \leq \frac{\rho_0}{2},
\end{equation}
if $\delta_0>0$ is chosen small enough. 
The property \eqref{eq:5ex24} leads to a contradiction if $t_0^* < \infty$. Thus 
$t_0^{*} =T_1=  \infty$.

\subsubsection{ The case where $H_{\nu}(t)$ changes sign on the interval 
$[(n_{m_0}+1)^{\gamma},T_1]$.}

The only remaining scenario left to deal with is as follows: $H_{\nu}(t)$ is never negative on $I_{m_{0}}$, but does not remain positive on all of $\tilde I_{m_{0}}$ (the extension of $I_{m_{0}}$ for which $\vec u$ stays in the $R_{1}$ ball in~$\calH$).  Recall that $H_{\nu}$ decreases on that larger interval. Say $H_{\nu}(t_{2})=0$ with $t_{2}\in \tilde I_{m_{0}}\setminus I_{m_{0}}$.  If so, then~\eqref{eq:5ex24} holds up to time $t_{2}$. After that time, Proposition~\ref{prop:H-} applies with $t_{2}$ taking the role of~$t_{1}$. The only reason that $t_{1}\in I_{m_{0}}$ in that proposition lies with the fact that it guarantees the initial closeness~\eqref{eq:5auxdefm0} to $(Q,0)$. But this is now provided by~\eqref{eq:5ex24}. Hence, the argument following Proposition~\ref{prop:H-} applies unchanged and $\vec u$ remains close to the equilibrium for all times, cf.~\eqref{eq:5exG15}. 

\subsection{Conclusion of the proof for \texorpdfstring{$d>3$}{d>3} or \texorpdfstring{$d=3$}{d=3} and \texorpdfstring{$\theta \leq 4$}{theta <=4}}
Up to this point we have only proved that  the trajectory $\vec u(t)$ is bounded in $\mathcal{H}$ and stays forever in the ball $B_{\mathcal{H}}(0,R_1)$ for $t\geq t_0$, with $t_0$ large enough. But, taking into account that $H_\nu(t_n)\to0$ for some sequence $t_n\to\infty$, we also have shown that, for $t_0$ large enough $H_{\nu}$ does not become negative. Moreover, we claim that proof in the case of positive $H_\nu$  yields that $\vec u(t) - (Q,0)$ converges to $0$ at a sub-exponential rate. Indeed, since the trajectory $\vec u(t)$  stays forever in the ball $B_{\mathcal{H}}(0,R_1)$ for $t\geq 2 t_0$, with $t_0$ large enough, we may choose $\nu =1$ and we obtain by arguing as in \eqref{eq:5ex13} that, for any $\tau > t$, 
\begin{equation}
\label{eq:5ex13Neu}
\int_{t}^{\tau} \frac{3}{2(1+ s)^{a}} \| u_t(s)\|_{L^2}^2   ds \leq H_{1}(t) ,
\end{equation}
and thus, for $t$ large enough,
\begin{equation}
\label{eq:5ex13BISNeu}
\int_{t}^{+\infty} \frac{1}{(1+ s)^{a}} \| u_t(s)\|_{L^2}^2  ds \leq K(t_0)
 \exp \Big( -   \frac{C_2 \varepsilon_0}{1 -a} (1+t)^{1-a} \Big),
\end{equation}
where $K(t_0)$ is a positive constant.  {}From \eqref{eq:5ex13BISNeu}, we deduce that, for $t \geq 2t_0$,
\begin{equation}
\label{eq:5ex13TERNeu}
E(\vec u(t)) - E (Q,0)  \leq C K(t_0)
 \exp \Big( -   \frac{C_2 \varepsilon_0}{1 -a}(1+t)^{1-a}\Big).
\end{equation}
Next, arguing as in \eqref{eq:5ex17} to \eqref{eq:5ex20}, we obtain that, for $t \geq 2t_0$, 
\begin{equation}
\label{eq:5ex21Neu}
\int_{t}^{+\infty} \| u_t(s)\|_{L^2}   ds \leq K^*(t_0) t^{-\mu}
 \exp \big( -  k (1+t)^{1-a}\big).
\end{equation}
The inequality \eqref{eq:5ex13BISNeu} implies that, for $s >t > 2t_0$, we have
 \begin{equation}
\label{eq:5ex22Neu}
\| u(s) - u(t) \|_{L^2} \leq K^*(t_0) t^{-\mu} \exp \big( -  k (1+t)^{1-a}\big).
\end{equation}
Since we alreay know that the sequence $u(t_n)$ converges to $Q$, when $n$ goes to infinity 
(see Theorem \ref{ThBRS1}), we may pass to the limit in $s= t_n$ to wit
 \begin{equation}
\label{eq:5ex24Neu}
\| u(t) - Q \|_{L^2} \leq K^*(t_0) t^{-\mu} \exp \big( - k(1+t)^{1-a}\big).
\end{equation}
Finally, from Lemma~\ref{lem-key} and~\eqref{eq:delrho} we conclude that, for $t \geq 2t_0$,
\begin{equation}
\label{eq:5ex25Neu}
\| \vec u(t) - (Q,0) \|_{\mathcal{H}} \leq K_1(t_0) \exp \big( -   \eta k   (1+t)^{1-a}\big),
\end{equation}
where  $\eta >0$ depends only on the power of the non-linearity. We have therefore established  sub-exponential   convergence of the entire trajectory to an equilibrium.

\section{Proof of Theorem \ref{ThBRS2} if \texorpdfstring{$d=3$}{d=3} and \texorpdfstring{$4<\theta < 5$}{4< theta<5}}

\subsection{Applying observation inequalities}   

 Throughout this section, we let $d=3$ and assume that the Hypothesis~\ref{H2f} holds, with  $4 <\theta < 5$. 
 
The difficulty with $d=3$ and these values of $\theta$ lies with the fact that Lemma~\ref{lem:estmHprime} fails in that case. It is not possible to make an estimate such 
as~\eqref{eq:5ex3} pointwise in time.  Rather, we will bound an averaged expression of the left-hand side of~\eqref{eq:5ex3} by means of Strichartz estimates and the  following observation inequality from  a companion paper~\cite{BRS_OBS}. 

\begin{proposition}\label{prop:obs}
Let $u$ be a solution of~\ref{KGalpha} on the interval $I=[t_{0},t_{1}]$ with $\| \vec u(t)\|_{\calH}\le M$ for $t\in I$. 
There exists $C=C(M,|I|)$ so that 
\begin{equation}\label{eq:pu main}
\| \partial_{t} u \|_{L^{\infty}(I,L^{2})}\le C(M, |I|) \| \partial_{t} u\|_{L^{2}(I, L^{2})}.
\end{equation}
\end{proposition}

By inspection, $C(M, |I|)$ is decreasing in $|I|$. 
This allows us to establish the following space-time bound analogous to Lemma~\ref{lem:estmHprime}.   

\begin{lemma} \label{lem:estH+H-}
Let $u$ be a solution of~\ref{KGalpha} on the interval $[t,t+\tau]$ where $0<\tau\le 1$ with $\| \vec u(s)\|_{\calH}\le M$ for all $t\le s\le t+\tau$. 
Then 
\begin{equation}
\label{eq:5ex3Integral}
\int_{t}^{t+\tau}| \langle f'(u)u_t,u_t \rangle_{H^{-1}} |     ds  \leq
C (M,\tau) \int_{t}^{t+\tau}  \| u_t\|_{L^2}^2   ds,
\end{equation}
where the  constant $C(M;\tau)>0$ depends on $M, \tau$ and the nonlinearity~$f$. 
\end{lemma}
\begin{proof}
By the embedding $H^{2}(\R^{3})\embed L^{\infty}(\R^{3})$ and its dual $L^{1}(\R^{3}) \embed H^{-2}(\R^{3})$ we may estimate
\begin{multline} \label{eq:U1}
 \int_{t}^{t+\tau}  |\la f'(u(s)) u_{t}(s), u_{t}(s)\ra_{H^{-1}} |   ds \\
\begin{aligned}
 &\leq C \int_{t}^{t+\tau}  \| f'(u(s)) u_{t}(s) \|_{H^{-2}} \| u_{t}(s)\|_{2}   ds \\
&\leq C \int_{t}^{t+\tau}  \| (|u(s)|^{\beta}+|u(s)|^{\theta-1})  u_{t}(s) \|_{L^{1}} \| u_{t}(s)\|_{2}   ds   \\
&\leq C \|u_{t}\|_{L^{\infty}(I, L^{2})} \int_{t}^{t+\tau}  \Big( \| u(s) \|_{L^{2\beta}}^{\beta}+ \| u(s)\|_{2(\theta-1)}^{\theta-1}\Big)   \|u_{t}(s) \|_{L^{2}}    ds,   
\end{aligned}
\end{multline}
where $I=[t,t+\tau]$. Due to the hypotheses made on $\beta$, we may bound
$\int_{t}^{t+\tau}  \Big( \| u(s) \|_{L^{2\beta}}^{\beta}+ \| u(s)\|_{2(\theta-1)}^{\theta-1}\Big)   \|u_{t}(s) \|_{L^{2}}ds$ by 
$\int_{t}^{t+\tau}  \Big( 1 +   \| u(s) \|_{L^{2}} +\| u(s)\|_{2(\theta-1)}^{\theta-1}\Big)   
\|u_{t}(s) \|_{L^{2}}ds$.
Bounding the first term by Proposition~\ref{prop:obs} we may continue bounding the above
expression by
\begin{equation}
\begin{aligned}
&\leq C  \int_{I} \|u_{t}(s)\|_{2}^{2}   ds \int_{t}^{t+\tau}   \Big(1+ \| u(s) \|_{L^{2}} + \| u(s)\|_{2(\theta-1)}^{\theta-1}\Big)      ds  \label{eq:U2} \\
&\leq  C(M,\tau) \int_{I} \|u_{t}(s)\|_{2}^{2}   ds.      
\end{aligned}
\end{equation}
To pass to the final inequality we bound $\| u(s)\|_{2}\le M$ and we use Strichartz estimates to  control 
\begin{equation}\label{eq:ST1}
\int_{t}^{t+\tau} \| u(s)\|_{2(\theta-1)}^{\theta-1}    ds \le C(M,\tau),
\end{equation}
which follows by local wellposedness, via the assumption $\|\vec u(s)\|_{\calH}\le M$.  Indeed, for the endpoint $\theta=5$ one has an $L^{4}_{t}([t,t+\delta], L^{8}_{x}(\R^{3}))$ Strichartz 
estimate with $\delta=\delta(M)>0$ and for $\theta=4$ we control the $L^{3}([t,t+\delta], L^{6}_{x}(\R^{3}))$ norm of the solution.  We can then cover the interval $[t,t+\tau]$ with $\delta$-intervals. It now follows that we can similarly handle the entire range $4<\theta<5$. 
\end{proof}

In the previous proof we applied the observation inequality twice, the second time to pass from \eqref{eq:U1} to~\eqref{eq:U2}. This was strictly speaking not necessary, as we could have bounded \eqref{eq:U1} by means of Cauchy-Schwarz. This would then lead to an $L^{8}_{t,x}$ Strichartz estimate for Klein-Gordon, which is precisely the optimal one on the scale of $L^{8}_{x}(\R^{3})$ for the wave equation with energy data.

In analogy with Lemma~\ref{lem:Hnu} we can use the previous lemma to obtain decay of $H_{\nu}$, but we can only compare times which are not closer than $\tau_{0}$. 
Recall that
\begin{equation*}
H_{\nu}(t):= E(\vec u(t)) - E(Q,0) + \frac{\varepsilon_0}{(1+t)^{a \nu}} \langle -\Delta u(t) +u(t)  -f(u(t)), u_t (t)\rangle_{H^{-1}}.
\end{equation*}
For the remainder of this section, we fix some $\tau_{0}\in (0,1]$, say $\tau_{0}=\frac12$ (the constants are then uniform in $\frac12 \le\tau_{0}\le 1$).  
We remark that~\eqref{eq:5ex3Integral} holds for any $\tau_{0}\ge \frac12$ with the same constant, provided that we remain in the interval of existence (but we are interested in global solutions only here).  This follows by adding up the individual estimates~\eqref{eq:5ex3Integral} over intervals of size~$\frac12$. 
As in the previous section, we take $0<a<\frac13$ and $\nu>1$ fixed. 

\begin{cor}
\label{cor:Hdecay2}
Let $u(t)$ be a solution of~\ref{KGalpha}  for $t_{0}\le t\le t_{1}$ where $t_{1}\ge t_{0}+1$. Assume $\|\vec u(t)\|_{\calH}\le M$ for all $t_{0}\le t\le t_{1}$. 
Then for $t_{0}\ge t_{0}^*=t_{0}^*(M)$ one has 
\begin{gather}
\label{eq-Hdec} 
H_{\nu}(t') - H_{\nu}(t) \leq 0  \qquad \forall t_{0}\le t\le t+1 \le t'\le  t_{1}\\
\int_{t}^{t'}\frac{3}{2(1+s)^{a}} \| u_t\|_{L^2}^2    ds \leq H_{\nu}(t) -
 H_{\nu}(t') . \label{eq-Hdec2} 
\end{gather}
 The inequality in~\eqref{eq-Hdec} is strict if $\vec u\ne (Q,0)$. 
\end{cor}
\begin{proof}
Arguing by density, we first assume that the data $\vec u(0)\in H^2(\R^3) \times H^1(\R^3)$. Then the solution $\vec u(t)$ is classical, i.e., 
 $$u\in C^1([0, +\infty), H^1(\R^3)) \cap C^0([0, +\infty), H^2(\R^3)) \cap C^2([0,+\infty)), L^2(\R^3))$$
 and $H'_{\nu}$ satisfies the equality \eqref{eq:5auxequal1}. Hence,  for $t,t' \in J:=[t_{0},t_{1}]$ as above,
\begin{multline*}
H_{\nu}(t') - H_{\nu}(t) =  \smash{\int_{t}^{t'}}H_{\nu}'(s)   ds \cr \\
 =\smash{\int_{t}^{t'}} \Big[-2\alpha(s)
 \| u_t(s)\|_{L^2}^2 - \frac{\nu a \varepsilon_0 }{(1+s)^{a\nu +1}} \langle -\Delta u(s)+u(s) -f(u(s)), u_t(s) \rangle_{H^{-1}} \cr
 - \frac{\varepsilon_0}{(1+s)^{a\nu}} \|  -\Delta u (s)+ u(s) -f(u(s))\|_{H^{-1}}^2 \cr
 - \frac{2\varepsilon_0}{(1+s)^{a(\nu +1)}} \langle -\Delta u(s) + u(s) -f(u(s)), u_t(s) \rangle_{H^{-1}} \cr
 +  \frac{\varepsilon_0}{(1+s)^{a\nu}} 
\langle -\Delta u_t (s)+ u_t (s)- f'(u(s))u_t(s),u_t(s) 
\rangle_{H^{-1}}  \Big]   ds.
\end{multline*}
Applying Lemma~\ref{lem:estH+H-} and the Cauchy-Schwarz inequality (in space, not time, cf.~Lemma~\ref{lem:Hnu}), one infers that there exists some constant $C_{1}= C_1( M)$ so that 
\begin{multline}
\label{eq:5-2estaux8BIS}
H_{\nu}(t') - H_{\nu}(t) \leq \int_{t}^{t'} 
\Big[ -\frac{(2-\varepsilon_0 C_1  (1+s)^{-a(\nu-1)})}{(1+s)^{a}}
\| u_t\|_{L^2}^2 \\
 - \frac{\varepsilon_0}{2(1+s)^{a\nu}}
\|  -\Delta u + u -f(u)\|_{H^{-1}}^2 \Big]    ds \leq 0,
\end{multline}
and for $s\ge t_{0}$ large, 
 \begin{equation}
\label{eq:5-2estaux9BIS}
2-\varepsilon_0 C_1 (1+s)^{-a(\nu-1)} \geq 3/2.
\end{equation}
Using a density argument, one shows that the above estimates \eqref{eq:5-2estaux8BIS} still hold when the initial data only belong to $\mathcal{H}$. 
\end{proof}

We remark  that we may also take $\eps_{0}$ small rather than taking $t_{0}$ large. For example, if the power of the nonlinearity is at most~$3$, then 
we can set $\nu=1$ which then requires making $\eps_{0}$ small.  Next, we combine Lemma~\ref{lem:estH+H-} with the {\L}ojasiewicz-Simon inequality. 
The small parameter $\rho_{0}>0$ is from Theorem~\ref{th:Loja}. 

\begin{cor}
\label{cor:HnuLS}
Assume $u$ is a solution of \ref{KGalpha} on the interval $[t_{0},t_{1}]$ satisfying  
\begin{equation*}
\| \vec u(s) - (Q,0) \|_{\calH} \le \rho_{0},\quad H_{\nu}(s)>0,  \qquad\forall t_{0}\le s\le t_{1}.
\end{equation*} 
Then 
\begin{equation}\label{eq:Hnudec step}
H_{\nu}(t') \le H_{\nu}(t) \exp\Big( -\int_{t}^{t'} \frac{C_{2}\eps_{0}}{(1+s)^{a\nu}}   ds\Big),
\end{equation}
for all $t_{0}\le t\le t+1\le t' \le t_{1}$, where  $C_{2}$ is a constant that depends on $Q$. 
\end{cor}
\begin{proof}
Using Theorem~\ref{th:Loja} we obtain~the inequality \eqref{eq:5auxLoja11}, i.e., 
\begin{equation*}
H_{\nu}(s) \leq C_1 \big(  \| -\Delta u(s) + u(s)  -f(u(s))\|_{H^{-1}}^2 + 
\| u_t(s)\|_{L^2}^2\big).
\end{equation*}
We next choose a constant $C_2$, $0 < C_2 < \frac{1}{8C_1}$ so that
\begin{equation}
\label{eq:5-2auxLoja11BIS}
C_2\varepsilon_0 H_{\nu}(s)  \leq \frac{\varepsilon_0}{8}  \big(  \| -\Delta u(s) + u(s)  -f(u(s))\|_{H^{-1}}^2 + \| u_t(s)\|_{L^2}^2\big).
\end{equation}
Therefore, with $\Omega(s) =\Omega(t,s)= \int_{t}^{s} \frac{C_2\varepsilon_0}{(1+u)^{a\nu}}   du$, 
\begin{equation}
\label{eq:5-2auxLoja12TER}
\begin{split}
 e^{\Omega(t+\tau_{0})}  H_{\nu}(t+\tau_0) -  H_{\nu}(t) & =
\int_{t}^{t+\tau_0} \frac{d}{ds}\Big[ e^{\Omega(s)} H_{\nu}(s)\Big]   ds \\
&= \int_{t}^{t+\tau_0}  e^{\Omega(s)} 
\big[\frac{C_2\varepsilon_0}{(1+s)^{a\nu}} H_{\nu} (s) + H'_{\nu}(s) \big] 
  ds  \leq 0. 
\end{split}
\end{equation}
The final inequality holds due to 
\begin{multline}
\label{eq:5-2ex8TER}
 \int_{t}^{t+\tau_0} e^{\Omega(s)}  H'_{\nu}(s)   ds \leq \int_{t}^{t+\tau_0}   e^{\Omega(s)} 
 \Big[ -\frac{3}{2(1+s)^{a}}\| u_t(s)\|_{L^2}^2 \\
 - \frac{\varepsilon_0}{2(1+s)^{a\nu}}
\|  -\Delta u(s) + u(s) -f(u(s))\|_{H^{-1}}^2 \Big]    ds \leq 0,
\end{multline}
which follows by basically the same proof as that of Corollary~\ref{cor:Hdecay2} 
(note that~$e^{\Omega(s)}\le C(\tau_{0})$ for all $t\le s\le t+\tau_{0}$). 
We infer from \eqref{eq:5-2auxLoja12TER} that 
\begin{equation}
\label{eq:5-2auxLojaFinal}
H_{\nu}(t+\tau_0) \leq 
 H_{\nu}(t) \exp \Big(- \int_{t}^{t+\tau_0} \frac{C_2\varepsilon_0}{(1+s)^{a\nu}}   ds \Big) .
\end{equation}
Iterating this estimate along an arithmetic sequence $t+j\tau_{0}$, with a suitably chosen $\frac12\le\tau_{0}\le1$,  concludes the proof. 
\end{proof}

\subsection{The main argument}   

We now indicate how to adapt the proof technique of Theorem~\ref{ThBRS2}  from Section~\ref{sec:Main1} to fit this section. 
First, we deal with the case in which $H_{\nu}$ can assume negative values in the interval $I_{m_{0}}$, cf.\ Proposition~\ref{prop:H-}. 
Reflecting the fact that monotonicity of $H_{\nu}$ now holds in a slightly weaker sense, we need to modify that proposition accordingly. 

\begin{proposition}\label{prop:H-*}
Assume that there exists $t_1 \in I_{m_0}$ with $H_{\nu}(t_{1})\le 0$. Let $[t_1,T_1]$ be the maximal time interval  on which $\| \vec u(t)\|_{\mathcal{H}} \leq R_1$. Then  $H_{\nu}\le  0$ on $[t_{1}+1,T_{1}]$ provided $m_0$ is large enough. In addition, if $\nu$ is sufficiently close to~$1$, then 
\begin{equation}\label{eq:Hnegprop2}
\int_{t_1+1}^{T_1}\| u_t(s)\|_{L^2}   ds  \leq C  t_1^{-a(\nu -1)},
\end{equation}
where $C$ is a positive absolute constant (independent of $m_0$).
\end{proposition} 
\begin{proof}
By \eqref{eq-Hdec} we have that $H_{\nu}(t)\le 0$ for all $t\in [t_{1}+1,T_{1}]$. Then the exact same proof as that of Proposition~\ref{prop:H-} 
gives the desired conclusion. 
\end{proof}

Theorem~\ref{ThBRS2} follows from this result in  the same fashion as in the previous section. As already noted there,  $T_{1}$ lies arbitrarily far to the right of $I_{m_{0}}$ if $m_{0}$ is large.   Moreover, $\vec u(t_{1}+1)$ will still satisfy~\eqref{eq:H3estconv} whence Corollary~\ref{cor:H-1}  applies in this context. 

If, on the other hand, $t_{1}$ as in the previous proposition does not exist, then due to Corollary~\ref{cor:Hdecay2} the proof of Section~\ref{subsec:LS} applies verbatim, provided the assumption which we made there holds, i.e., that $H_{\nu}>0$ on $\tilde I_{m_{0}}$. Finally, if this is not the case then $H_{\nu}(t_{1})= 0$ for some $t_{1}\in \tilde I_{m_{0}}\setminus I_{m_{0}}$. Then by the monotonicity properties derived in this section, $H_{\nu}(t)<0$ for all $t\in \tilde I_{m_{0}}$ with $t\ge t_{1}+1$. Up to time $t_{1}$, Section~\ref{subsec:LS} applies, and by well-posedness the closeness~\eqref{eq:5ex24} remains correct for all $t_{1}\le t\le t_{1}+1$ ($m_{0}$ large). Hence, we can then apply the argument for negative $H_{\nu}$ to conclude as before.

\section{Proof of Theorem \ref{ThBRS3} }

Mutatis mutandis, we argue as in the proofs of Theorem \ref{ThBRS2} in Sections 5 and 6. The only changes are that we apply Lemma \ref{lem:suite4} and Proposition \ref{prop:H4Hyperb} instead of Lemma \ref{lem:suite3} and Proposition \ref{prop:H3estconv} and that we choose $\nu >1$ so that $a\nu < 1/2$.

\appendix
\section{The {\L}ojasiewicz-Simon inequality}

In this section we prove the following estimate which plays an instrumental role in our main argument. 

\begin{theorem} \label{th:Loja}
Assume that the hypotheses \ref{H1f} and \ref{H2f} hold. Let $(Q,0)$ be an equilibrium point of \ref{KGalpha} in $\mathcal{H}_{rad}$. Then there exist $0 <\rho_0 <1$ and $C_{0}>0$  such that, for any $u \in B_{H^1_{rad}}(Q,\rho_0)$, we have
\begin{equation}
\label{eq:Loja}
 | J(u) - J(Q)| \leq C_0 \| -\Delta u +u -f(u)\|_{H^{-1}(\R^d)}^{2},
\end{equation}
where $J(u)= E(u,0)$.
\end{theorem}
 
Results of this type go back to  Simon~\cite[Section 3.12, pages 74]{Simon1996}. We will invoke a somewhat more abstract formulation of Simon's theorem by 
 Haraux and Jendoubi~\cite{HaJen07}.   To state it, let $V$ and $H$ be two Hilbert spaces, with $V$ dense in~$H$. In addition, \cite{HaJen07} requires the embedding $V\embed H$ to be compact. 
 But this is not necessary as we will explain below. We identify $H$ with its dual $H^*$, whence $H\embed V^*$ is bounded. Let $G\in C^1(V)$ be real-valued, and set $\calM:= \nabla G$. 
 This means that 
 \begin{equation*}
 \langle (\nabla G)(v), w\rangle = \frac{d}{dt}\Big|_{t=0} G(v+tw),\quad \forall v,w \in V.
 \end{equation*}
The pairing on the left-hand side is the duality pairing between $V$ and $V^*$.  Thus,  $\calM : V\to V^*$ as a continuous map.  

\begin{theorem}[Simon, Haraux-Jendoubi]
\label{thm:SHJ}
Assume  $\calM \in C^1(V,V^*)$. Let $\phi\in V$ be such that $\calM(\phi)=0$. Suppose $L=\calM'(\phi)$, which is a bounded linear transformation $V\to V^*$,  has the form 
\begin{equation*}
L = \Lambda + B,
\end{equation*}
where $\Lambda:V\to V^*$ is an isomorphism, and $B:V\to V^*$ is compact. Denote  $N=\ker(L)$, which is of dimension $d<\infty$. If $d>0$ we assume that $\calM^{-1}(0)$ near $\phi$ is a smooth $d$-dimensional manifold, which then has $N$ as   tangent space at $\phi$.  

Under these assumptions, there exist  $\delta>0$ and a constant $C$ such that 
\begin{equation*}
\forall  \| u-\phi\|_H <\delta , \quad |G(u)-G(\phi)|\le C\| \calM(u) \|_{V^*}^2 .
\end{equation*}
\end{theorem}
\begin{proof}
We present the beginning of the proof here in some detail since we do not assume that $V\embed H$ is compact. The case $d=\dim (N)=0$ is easy since we may then invert $\Lambda$ by Fredholm's theorem. 
If $d>0$, we let $\Phi$ be the orthogonal projection in $H$ onto~$N$. Now define $\calL:V\to V^*$, $\calL=\Pi+L$.  Since $\Pi$ has finite rank, it is compact as a map $V\to V^*$. We do not need compactness of $V\embed H$ for this step.  We claim that $\ker (\calL)=\{0\}$. If so, Fredholm's theorem implies that $\calL=\Lambda + \Pi+B:V\to V^*$ is an isomorphism. 

To prove the claim we need to invoke the symmetry of $L$. This means that 
\begin{equation*}
\forall v,w\in V, \quad \langle L v,w\rangle = \langle  L w, v\rangle. 
\end{equation*}
The pairing on the left is again the duality pairing between $V$ and $V^*$. This identity  reflects the symmetry of the Hessian in calculus and follows from the fact that 
\begin{align} 
\langle L v,w\rangle &= \partial_t\Big|_{t=0} \partial_s\Big|_{s=0} G(\phi+ s v + tw) \\
& =  \partial_s\Big|_{s=0} \partial_t\Big|_{t=0} G(\phi+ s v + tw)  = \langle  L w, v\rangle.
\end{align}
Now suppose $\calL v=0$ whence $Lv=-\Pi v\in H\embed V^*$. Then 
\begin{equation*}
\langle L v, \Pi v \rangle = - \langle \Pi v , \Pi v \rangle = - \| \Pi v\|_{H}^2.
\end{equation*}
The final equality here follows from Riesz representation and density of $V$ in $H$. On the other hand, 
\begin{equation*}
\langle L v, \Pi v \rangle = \langle L \Pi v,  v \rangle =0,
\end{equation*}
since $\Pi v\in N=\ker L$. Thus $\Pi v=0$ and so $Lv=0$ which means that $v\in N$. But then $v=\Pi v =0$ and the claim holds. 

The remainder of the proof is identical to Haraux-Jendoubi and we refer the reader to pages 452--55 in~\cite{HaJen07}. 
\end{proof}

The rest of this appendix is devoted to the proof of Theorem~\ref{th:Loja}.  
We first recall some well-known facts. Let us consider the solutions 
$Q \in H^1_{rad}(\R^d)$ of the elliptic equation
\begin{equation}
\label{eq:Aelliptic}
-\Delta Q +Q -f(Q) =0.
\end{equation}
By elliptic theory, see for example~\cite{BeLions83II}, the solutions of 
\eqref{eq:Aelliptic} are exponentially decaying, and lie in $C^{3,b}$ for some $b>0$. We set
 $$
 L \equiv - \Delta + I - f'(Q).
 $$
In \cite[Lemma 2.9]{BRS1}, we have stated the following properties. The operator 
$L: H^2(\R^d) \subset L^2(\R^d) \to L^2(\R^d)$ is self-adjoint with domain $H^{2}(\R^{d})$. The spectrum $\sigma(L)$ consists of an essential part $[1,\infty)$, which is absolutely continuous, and finitely many eigenvalues of finite multiplicity all of which fall into $(-\infty, 1]$. The eigenfunctions are $C^{2,b^*}$ with $b^*>0$ and the ones associated with eigenvalues below $1$ are exponentially decaying.  Over the radial functions, all eigenvalues are simple.
Notice that $L$ can be extended to a continuous operator from $H^1(\R^d)$ into 
$H^{-1}(\R^d)$ (as well as from $H^1_{rad}(\R^d)$ into $H^{-1}_{rad}(\R^d)$).

In the case where $Q  \in H^1_{rad}(\R^d)$ is a hyperbolic solution in $H^1_{rad}(\R^d)$ of the elliptic equation \eqref{eq:Aelliptic}, the proof of Theorem \ref{th:Loja}  is elementary and we recall it now.  Note that this is precisely the case $\dim N=0$ which we skipped in the proof of Theorem~\ref{thm:SHJ} above. 

\begin{proof}[Proof of \eqref{eq:Loja} when $Q$ is an hyperbolic equilbrium (in $H^1_{rad}(\R^d)$)]
Let $u=Q+v$. By Taylor's formula,
\begin{equation}
\label{eq:AJ(u)1}
\begin{aligned}
J(u) - J(Q) = &\int_{\R^d}[\frac{1}{2}( |\nabla(v+Q) |^2 -  | \nabla Q |^2 + | v+Q |^2 
-  | Q |^2 ) \\
& \ \qquad \ \qquad  \ \qquad \qquad \qquad \qquad \hfill - F(v+Q) +F(Q) ]   dx \cr
=& \frac{1}{2} \| v\|_{H^1}^2 - \int_{\R^d}(F(v+Q) -F(Q) - f(Q)v)   dx \cr
= &  \frac{1}{2} \| v\|_{H^1}^2 -  \int_{\R^d} \int_0^1 (1-s) f'(Q+sv) v^2   ds dx .
\end{aligned}
\end{equation}
Using the hypothesis \ref{H2f} and the fact that $\| v\|_{H^1} \leq 1$, we deduce from \eqref{eq:AJ(u)1} that 
\begin{equation}
\label{eq:AJ(u)2}
 | J(u) - J(Q)|   \leq C \| v\|_{H^1}^2.
\end{equation}
Next, we write 
\begin{align}
\label{eq:ALaplaceu1}
 -\Delta u + u -f(u) &=   -\Delta v + v - (f(v+Q) -f(Q)) \\ & = Lv - \int_0^1 (f'(Q+sv) - f'(Q)) v   ds.
\end{align}
Since $L$ is an isomorphism from  $H^1_{rad}(\R^d)$ into $H^{-1}_{rad}(\R^d)$, we infer from 
\eqref{eq:ALaplaceu1} that 
\begin{equation}
\label{eq:ALaplaceu2}
v  = L^{-1}( -\Delta u + u - f(u))  - L^{-1}(\int_0^1 (f'(Q+sv) - f'(Q)) v   ds).
\end{equation}
Using the continuous Sobolev embedding of $L^{q'}(\R^d)$ into $H^{-1}(\R^d)$, where 
$q' = 2d/ (d+2)$, together with Hypothesis \ref{H2f}, we obtain that
\begin{multline}
\label{eq:Ao(v)}
\Big\|  \int_0^1 (f'(Q+sv) - f'(Q)) v   ds\Big\|_{H^{-1}} \\
\leq  C  
\big( \| v\|_{L^{2d(\beta +1)/(d+2)}}^{\beta +1} 
+   \| v\|_{L^{2d \theta /(d+2)}}^{\theta}\big) \leq \tilde{C} \| v\|_{H^1}^{1+ \beta},
\end{multline}
where $0 < \beta < \theta -1$. If $\rho_0>0$ is chosen so that $1 - \tilde{C}\rho_0^{\beta} > 0$, we deduce from the inequalities \eqref{eq:AJ(u)2} to 
\eqref{eq:Ao(v)}, that there exists a positive constant $C_0$ so that the inequality 
\eqref{eq:Loja} holds.

\end{proof}

\begin{proof}[Proof of \eqref{eq:Loja} in the general case] 
We will apply Theorem~\ref{thm:SHJ}. To fix notation, we have  $H^1_{rad}(\R^d) = V$, $H= L^2_{rad}(\R^d)$ and 
$V'=H^{-1}_{rad}(\R^d)$. The functional $G(u)$ is the energy functional $J(u)$. Clearly,
$$\nabla G(u)=  \mathcal{M}(u) = -\Delta u + u - f(u). $$ In our case, $\mathcal{M}$ is in 
$C^1( H^1_{rad}(\R^d), H^{-1}_{rad}(\R^d))$ and  $\mathcal{M}'(Q) = L$. We notice that
$Lv = -\Delta v + v -f'(Q)v$ and that  $-\Delta +I $ is an isomorphism from $H^1_{rad}(\R^d)$ into $H^{-1}_{rad}(\R^d)$ and also from 
 $H^2_{rad}(\R^d)$ into $L^2_{rad}(\R^d)$. To fulfill  the hypothesis (ii) of Theorem 1.1 of \cite{HaJen07}, we need to prove that the operator of multiplication by $-f'(Q)$ is a compact operator from 
 $H^1(\R^d)$ into $H^{-1}(\R^d)$ and even from $H^1(\R^d)$ into $L^2(\R^d)$. Let $R>1$. Using Hypothesis \ref{H2f}, we can write, for any $v \in H^1(\R^d)$,
 \begin{equation}
\label{eq:AsuperieurR}
\begin{split}
\int_{ |x| >R} |f'(Q) v|^2    dx \leq C \frac{1}{R} \| v\|_{L^2}^2
 \| |x| |Q(x)|^{2(\theta -1} \|_{L^{\infty}} \leq \tilde{C} \frac{1}{R} \| v\|_{L^2}^2.
\end{split}
\end{equation}
Let next $v_n$, $n \geq 1$, be a bounded sequence in $H^1(\R^d)$ (bounded by $C_0$). Without loss of generality, we may assume that $v_n$ converges weakly in $H^1(\R^d)$ to 
$v^* \in H^1(\R^d)$. For any $\varepsilon_m= 1/m$, $m \in \N$, we can choose $R_m >1$ large enough so that, by \eqref{eq:AsuperieurR},
$$
\int_{ |x| >R_m} |f'(Q) v|^2   dx \leq \tilde{C} C_0^2 \frac{1}{R_m} \leq \frac{\varepsilon_m^2}{4}.
$$
Since the  embedding of $H^1(B_{H^1}(0,R_m))$ into $L^2(B_{H^1}(0,R_m))$ is compact, there exists a subsequence $v_{n_m}$ of $v_n$ such that $v_{n_m}$ converges to $v^*$ in $L^2((B_{H^1}(0,R_m))$. In particular, there exists $n_{m^0}$ such that for 
$m > {m^0}$,
$$
 \| v_{n_m} - v^*\|_{L^2(B_{H^1}(0,R_m))} \leq \frac{\varepsilon_m}{2}.
$$
We deduce from the above two inequalities that, for $n_m > n_{m^0}$,
$$
\| v_{n_m} - v^*\|_{L^2(\R^d)} \leq \varepsilon_m.
$$
Repeating this argument, we can construct a subsequence $v_{n_j}$ of $v_n$ such that
$v_{n_j}$ converges to $v^*$ in $L^2(\R^d)$ as $n_j \to +\infty$. Thus the multiplication operator by 
$-f'(Q)$ is a compact operator from $H^1(\R^d)$ into $L^2(\R^d)$. This implies that 
$L$ is a Fredholm operator of index $0$ from $H^2_{rad}(\R^d)$ into $L^2_{rad}(\R^d)$ or from $H^1_{rad}(\R^d)$ into $H^{-1}_{rad}(\R^d)$. 
\end{proof}

%%%%%%%%%%%%%%%%%%%%%%%%%%%%%%%%%%%%%%%%%%%%%%%%%%%%%%%%

\end{document}